\documentclass[11pt,letterpaper]{amsart}

\makeatletter
\def\input@path{{./}{./draft/proof/}}
\makeatother

\numberwithin{equation}{section}
\numberwithin{figure}{section}

\usepackage{amssymb}
\usepackage{mathtools}
\usepackage[headings]{fullpage}
\usepackage{amsmath}
\usepackage{thmtools}
\usepackage{thm-restate}
\makeatletter
\@ifundefined{newcounteralias}{}{%
  \renewcommand\thmt@autorefsetup{%
    \@xa\def\csname\thmt@envname autorefname\@xa\endcsname\@xa{\thmt@thmname}}%
}
\makeatother
\usepackage{graphicx}
\usepackage{xparse}
\usepackage[dvipsnames]{xcolor}
\usepackage[normalem]{ulem}
\usepackage[
    backend=biber,
    isbn=false,
    url=false,
    sorting=nyt,
    maxbibnames=8,
]{biblatex}
\makeatletter
\AtBeginDocument{%
  \patchcmd{\@evenhead}{\leftmark{}{}}{\hbox{\leftmark{}{}}}{}{}%
}
\makeatother

\usepackage[colorlinks=true,linkcolor=blue, citecolor=blue, urlcolor=blue]{hyperref}
\usepackage{enumitem}
\usepackage{tikz}
\usepackage{tikz-cd}
\usepackage{caption}

\newcount\tableauRow
\newcount\tableauCol

\usetikzlibrary {positioning}
\definecolor {processblue}{cmyk}{0.96,0,0,0}

\makeatletter
\def\theenumi{\@alph\c@enumi}
\makeatother

\usepackage{units}
\usepackage{ytableau}

\theoremstyle{plain}
\newtheorem{theorem}[equation]{Theorem}

\newtheorem{lemma}[equation]{Lemma}
\newtheorem{corollary}[equation]{Corollary}
\newtheorem{proposition}[equation]{Proposition}
\theoremstyle{definition}

\newtheorem{conjecture}[equation]{Conjecture}
\newtheorem{remark}[equation]{Remark}

\newtheorem{example}[equation]{Example}

\newtheorem{definition}[equation]{Definition}

\newtheorem{notation}[equation]{Notation}

\newtheorem{discussion}[equation]{Discussion}

\newtheorem{observation}[equation]{Observation}

\newtheorem{construction}[equation]{Construction}

\newcommand{\Ext}{\operatorname{Extn}}
\newcommand{\del}{\operatorname{del}}

\DeclareMathOperator{\Cat}{Cat\textbf{}}
\newcommand{\Rec}{\mathrm{Rec}}
\newcommand{\std}{\mathrm{std}}

\newcommand{\Av}{\operatorname{Av}}
\newcommand{\CoGr}{\operatorname{CoGr}}
\newcommand{\Vex}{\operatorname{Vex}}
\newcommand{\CoVex}{\operatorname{CoVex}}
\newcommand{\Sm}{\operatorname{Sm}}
\newcommand{\Lev}{\operatorname{Lev}}

\newcommand{\Grass}{\mathcal{G}}

\newcommand{\GL}{\mathrm{GL}}

\newcommand{\iidsim}{\stackrel{\mathrm{iid}}{\sim}}

\DeclareMathOperator{\FS}{\mathrm{FS}}
\DeclarePairedDelimiterXPP\Prob[1]{\mathbb{P}}(){}{#1}
\DeclarePairedDelimiterXPP\E[1]{\mathbb{E}}{[}{]}{}{#1}

\DeclarePairedDelimiterXPP\Var[1]{\mathbb{V}\mathrm{ar}}(){}{#1}
\DeclarePairedDelimiterXPP\Cov[1]{\mathbb{C}\mathrm{ov}}(){}{#1}
\DeclarePairedDelimiterXPP\Corr[1]{\mathbb{C}\mathrm{orr}}(){}{#1}
\newcommand{\Unif}[1]{\mathrm{Unif}(#1)}

\usepackage{verbatim}

\DeclareMathOperator{\rec}{\mathcal{R}}
\DeclareMathOperator{\BS}{BS}

\title{The record statistic and forward stability of Schubert products}
\author{Andrew Hardt}
\author{Reuven Hodges}
\author{Hanzhang Yin}

\begin{document}

\begin{abstract}
We initiate a probabilistic study of forward stability for products of Schubert polynomials through the record statistic (left-to-right maxima) of permutations. Building on the explicit record formula for forward stability obtained by Hardt and Wallach, we study random pairs of permutations drawn from three natural families: uniform permutations, Grassmannian permutations, and Boolean permutations. For each family, we determine record probabilities and use them to analyze the asymptotic behavior of forward stability. For uniform and Grassmannian permutations, we obtain asymptotics for the mean together with limiting distribution results. For Boolean permutations, we prove linear-order growth of the mean, and our analysis also produces an explicit time-inhomogeneous Markov chain that yields an exact linear-time uniform sampler. Beyond these cases, we prove that the record-set statistic is equidistributed on the avoidance classes of $132$ and $231$, and consequently the corresponding forward stability distributions coincide. We conclude with conjectures for numerous further permutation classes and a conjectural recursive criterion for when two avoidance classes have the same record-set distribution.
\end{abstract}
\maketitle

\section{Introduction}\label{sec:introduction}

Records (left-to-right maxima in the one-line notation of a permutation) are among the earliest and most classical permutation statistics. For a uniform random permutation
$w\in S_n$, R\'enyi proved that the record events at positions $1,2,\dots,n$ are mutually independent and satisfy
$\Prob{\text{record at position }j}=1/j$ \cite{Renyi1962}.  Beyond the uniform model, record behavior becomes
combinatorially rich; when one restricts to structured families (e.g.\ pattern avoidance classes, Grassmannian, Boolean), record events are no
longer independent and their correlations reflect fine structure of the class.  A main point of this paper is that this classical
statistic is not merely a probabilistic curiosity—it is the correct statistic for a geometric stabilization problem in Schubert
calculus.  In work of the first author and Wallach \cite{HardtWallach2024Schubert}, the \emph{forward stability} $\FS(u,v)$ of the product of
Schubert polynomials $\mathfrak{S}_u\mathfrak{S}_v$--the smallest $N$ for which the product is fully realized in the cohomology ring of the flag variety $\GL_N/B$--admits an explicit max formula that can be expressed in terms of the record events of $u$ and $v$.  This makes forward stability amenable to probabilistic analysis via records, and one can study $\FS(u,v)$ through the extremes of an explicit record-driven process.

We exploit this bridge between records and Schubert products to quantify $\FS(u,v)$ for random
independent pairs $(u,v)$ drawn from three fundamental families: uniform permutations, Grassmannian permutations (those with at most one
descent), and Boolean permutations (those whose reduced words contain no repeated simple reflections).  We focus on these three classes
because they exhibit markedly different record behavior—and consequently markedly different stabilization behavior—while also arising
naturally in contexts where Schubert polynomials and Schubert varieties are especially well studied.  We first determine record
probabilities in each family and then derive asymptotics for the expected value $\E{\FS(u,v)}$ and distributional limits,
illustrating how changes in record density and correlation structure produce distinct stabilization regimes.

\subsection{The Record Statistic}\label{subsec:record-statistic}

For $n\ge 1$ let $S_n$ denote the symmetric group on $\{1,\dots,n\}$, written in one-line notation
$w=w(1)\,w(2)\cdots w(n)$.  For $w\in S_n$ and $j\in[n]$, define the left-inversion count
\[
d_j(w)\;:=\;\left|\{\,k<j:\ w(k)>w(j)\,\}\right|.
\]
We say that $j$ is a \emph{record position} of $w$ if $d_j(w)=0$.  Define the record indicator and its complement by
\[
\rec_j(w)\;:=\;\mathbf{1}\{d_j(w)=0\},
\qquad
\chi_j(w)\;:=\;1-\rec_j(w).
\]

For a finite set $\mathcal W$ we write $\Unif{\mathcal W}$ for the uniform distribution on $\mathcal W$, and write $w\sim \Unif{\mathcal W}$ when $w$ is drawn from this distribution. We also write $\chi_j^{\mathcal W}$ for the random variable $\chi_j(w)$ under $w\sim\Unif{\mathcal W}$. Besides $S_n$ itself, the two main families we consider are the following.  A permutation $w\in S_n$ is \emph{Grassmannian} if there
is at most one index $m\in\{1,\dots,n-1\}$ with $w(m)>w(m+1)$; let $\Grass_n$ denote the set of Grassmannian permutations in $S_n$.
A permutation $w\in S_n$ is \emph{Boolean} if it can be written as a product of simple transpositions $s_i=(i\ \ i{+}1)$ in which no
generator $s_i$ appears more than once; let $\mathcal B_n$ denote the set of Boolean permutations in $S_n$. Our first result records the probability that a random permutation from each of these families has a record at a given position.

\begin{theorem}[Record probabilities across permutation families]\label{thm:record-probabilities}
Fix $n\ge 1$ and $j\in[n]$.
\begin{enumerate}[label=(\arabic*)]
\item \textbf{Uniform (R\'enyi \cite{Renyi1962}).} If $w\sim \Unif{S_n}$, then
\[
\Prob{\rec_j(w)=1}=\frac{1}{j}.
\]

\item \textbf{Grassmannian.} If $w\sim \Unif{\Grass_n}$, then
\[
\Prob{\rec_j(w)=1}=\frac{\displaystyle \sum_{k=j}^{n}\binom{n}{k}+2^{j-1}-n}{2^{n}-n}
\]

\item \textbf{Boolean.} If $w\sim \Unif{\mathcal B_n}$, then
\[
\Prob{\rec_j(w)=1}
=
\frac{F_{2n-2} - F_{2j-4}F_{2(n-j)-1}}{F_{2n-1}}
=
\frac{2\phi^{2n-2} + 2\phi^{2-2n} + \phi^{2n-4j+3} - \phi^{4j-2n-3}}{\sqrt{5}\,(\phi^{2n-1} + \phi^{1-2n})},
\]
where $F_m$ denotes the Fibonacci sequence (extended to $m\in\mathbb Z$ by $F_{-m}=(-1)^{m+1}F_m$) and
$\phi=\frac{1+\sqrt5}{2}$.
\end{enumerate}
\end{theorem}

The Grassmannian case is combinatorially straightforward once one parametrizes a permutation by its descent position $k$ and the $k$-subset $V$ of values in the prefix: the record event $\rec_j(w)=1$ translates into a simple explicit constraint on $(k,V)$. By contrast, the Boolean case requires a more involved structural argument. Imposing a record at $j$ forces a natural bijective splitting of a Boolean permutation into a left factor with simple reflections supported on $\{1,\dots,j-1\}$ and a right factor supported on $\{j,\dots,n\}$, leading to a recursion for the resulting enumeration. Solving this recursion yields a Fibonacci closed form.

\subsection{Backward and Forward Stability}

Let $S_n$ denote the symmetric group on $\{1, \dots, n\}$, and let $S_{\mathbb{Z}_+}$ denote the infinite symmetric group on the positive integers. For $w \in S_n$, the length $\ell(w)$ is the number of inversions of $w$. Let $w_0\in S_n$ be the longest element, given by $w_0(i)=n+1-i$.

A \emph{Schubert polynomial} $\mathfrak{S}_w$ is a polynomial representative for the Schubert class corresponding to $w \in S_n$ in the cohomology of the flag variety. These polynomials form a basis for the polynomial ring and their structure constants encode intersection numbers in algebraic geometry. For our purposes, the key property is that products of Schubert polynomials exhibit \emph{stability}.

If $w\in S_n$, let $w\times 1$ denote the permutation in $S_{n+1}$ fixing $n+1$ and acting as $w$ on $1,\ldots,n$. We can thus view $w$ as living in all $S_m$ for $m\ge n$.
The \emph{Schubert structure constants} for $u,v,w\in S_n$ are given by
\[
[X_u][X_v] = \sum_{w\in S_n} c_{u,v}^w[X_w],
\]
where $[X_u], [X_v], [X_w]$ are the classes of the corresponding Schubert varieties in the cohomology ring of the flag variety $\GL_n/B$. Lascoux and Sch\"utzenberger \cite{LascouxSchutzenberger} introduced polynomial representatives for these classes. Since these representatives are independent of $n$, one has $c_{u\times 1, v\times 1}^{w\times 1}\!=\!c_{u,v}^w$, and moreover for all sufficiently large $m$, the number of permutations $w$ appearing in the Schubert product $[X_u][X_v]$ taken over $S_m$ is independent of $m$. Let $\FS(u,v)$ be the smallest positive integer $m$ such that $w\in S_m$ for all permutations $w$ with $c_{u,v}^w\!\ne\!0$; it represents the smallest integer $m$ such that the Schubert product $\mathfrak{S}_u\mathfrak{S}_v$ can be fully realized in the flag variety $\GL_m/B$.

Similarly, if $w\in S_n$, let $1\times w$ denote the permutation in $S_{n+1}$ fixing $1$ and acting on $2,\ldots,n+1$ as $(1\times w)(i) := w(i-1)+1$. For $k\ge0$, write $1^k\times w$ for the $k$-fold iterate of this operation, with $1^0\times w=w$. Li \cite[Theorem 1.3(1)–(2)]{Li-back-stable-conjecture} proved that $c_{1\times u,\,1\times v}^{1\times w}=c_{u,v}^w$ and that, after sufficiently many shifts, no new terms occur. Following \cite[Definition~1.1]{HardtWallach2024Schubert}, define $\BS(u,v)$ to be the smallest nonnegative integer $k$ after which further shifts introduce no new terms; that is, for every integer $r\ge0$,
\[
\mathfrak{S}_{1^{k+r}\times u}\mathfrak{S}_{1^{k+r}\times v}
=\sum_{w\in S_{\mathbb{Z}_+}}c_{1^k\times u,\,1^k\times v}^{w}\mathfrak{S}_{1^r\times w}.
\]
In particular, $\BS(\mathrm{id},\mathrm{id})=0$.

The first author and Wallach~\cite{HardtWallach2024Schubert} provided explicit formulas for these numbers, and also showed that the same formulas give the analogous stability numbers in equivariant cohomology and K-theory~\cite{HardtWallach-Grothendieck}.

\subsection{Expectation and distribution of forward stability}\label{subsec:main-results}

For a permutation $w\in S_{\mathbb{Z}_+}$, let $\FS(w) = \min\{n\ge 1 \mid w\in S_n\}$, and for $i\ge 1$, define
\[
\Lambda_i(w):=\left|\{\,j\ge i:\exists\, j'<j \text{ with } w(j')>w(j)\,\}\right|
=
\sum_{j\ge i}\chi_j(w).
\]
Since $w$ fixes all sufficiently large indices, we have $\chi_j(w)=0$ for all large $j$, so $\Lambda_i(w)$ is a finite sum.
The first author and Wallach \cite[Theorem~1.6]{HardtWallach2024Schubert} show that for all $u, v \in S_{\mathbb{Z}_{+}}$,

\begin{equation}\label{eq:mainFSEq}
\FS(u, v)=\max_{1\le i \le 1+\max(\FS(u), \FS(v))}\left(\Lambda_i(u)+\Lambda_i(v)+i-1\right).
\end{equation}

Thus $\FS(u,v)$ is the maximum of a record-driven random--walk--type process $\{Y_i(u,v)\}$ defined by $Y_i(u,v):=\Lambda_i(u)+\Lambda_i(v)+i-1$ for $1\le i\le 1+\max(\FS(u),\FS(v))$. This characterization reduces questions about stabilization of Schubert products to extreme-value analysis for
statistics built from record events. We write $X_n\xrightarrow{d}X$ for convergence in distribution, and
$\mathcal{N}(0,1)$ for the standard normal distribution.

\begin{restatable}{theorem}{theoremUniformMain}
Let $u,v\iidsim\Unif{S_n}$. As $n\to\infty$,
\[
\E{\FS(u,v)} = 2n-2\ln n-2\gamma + c + o(1),
\]
where $c\in[1,\pi^2/6]$ and $\gamma$ is the Euler--Mascheroni constant, and
\[
\dfrac{\FS(u, v) - (2n - 2\ln n)}{ \sqrt{2 \ln n}} \xrightarrow{d} \mathcal{N}(0, 1).
\]
\end{restatable}

\begin{restatable}{theorem}{theoremGrassmannianMain} \label{thm:GrassmannianMain}
Let $u,v\iidsim\Unif{\Grass_n}$. As $n\to\infty$,
\[
\E{\FS(u, v)} = \frac{3}{2}n-\frac{1}{2\sqrt{\pi}}\sqrt{n}-2+o(1),
\]
and, with $Z_1,Z_2\sim\mathcal{N}(0,1)$ independent,
\[
\frac{\FS(u, v) - \frac{3}{2}n}{\sqrt{n}} \xrightarrow{d} -\frac{1}{2} \max(Z_1, Z_2).
\]
\end{restatable}

\begin{restatable}{theorem}{theoremBooleanExpectation}
\label{thm:BooleanExpectation}
Let $u,v\iidsim\Unif{\mathcal{B}_n}$. Then
\[
\E{\FS(u,v)} = n+O(1).
\]
\end{restatable}

\begin{center}
\includegraphics[width=0.92\textwidth]{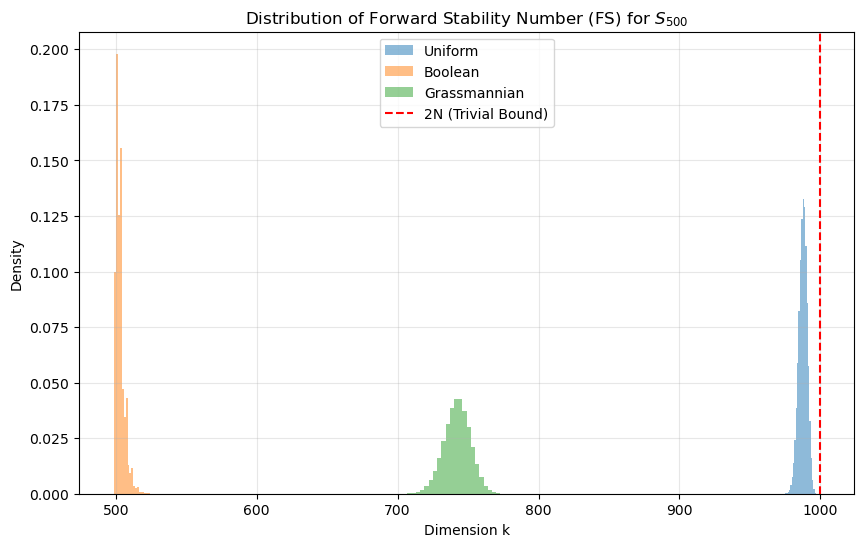}
\captionof{figure}{Empirical density histograms of $\FS(u,v)$ at $n=500$ for the three sampling models considered in the paper: uniform permutations, Boolean permutations, and Grassmannian permutations.}
\label{fig:fs-s500-distribution}
\end{center}

We obtain analogous results for $\BS(u,v)$ in Section~\ref{sec:backward-stability}. In the Boolean case, our analysis produces an explicit time-inhomogeneous Markov chain together with transition probabilities that can be turned into an exact uniform sampler for Boolean permutations which runs in linear time; see Remark~\ref{rem:boolean-sampling}.

\subsection{Conjectures and record-equidistribution phenomena}

Beyond the three primary families treated in this paper, our Monte Carlo experiments suggest a broad range of asymptotic behaviors for
$\E{\FS(u,v)}$ across classical permutation classes.  We formulate a collection of conjectures in Section~\ref{sec:conjectures} that organize these classes into three
regimes: a record-dense regime with $\E{\FS(u,v)}=n+o(n)$, an intermediate regime with $\E{\FS(u,v)}=cn+o(n)$ for some
$1<c<2$, and a record-sparse regime with $\E{\FS(u,v)}=2n-o(n)$.

A complementary theme is that the record-set statistic can be \emph{equidistributed} across distinct permutation classes, and in that
case the resulting forward stability distributions coincide.  We prove, in Section~\ref{sec:132-vs-231}, that the record-set statistic is equidistributed on the Catalan avoidance classes $\Av_n(132)$ and $\Av_n(231)$, and consequently $\FS(u,v)$ has the same distribution when $(u,v)$ is sampled uniformly from either class.

This example suggests a broader classification problem.  In the final part of the paper, we formulate a conjectural recursive criterion for when two patterns define avoidance classes with the same record-set distribution. Since $\FS(u,v)$ measures stabilization of Schubert products in the cohomology of $\GL_N/B$, exact coincidences in the $\FS$-distribution suggest that these classes may be shadowing a deeper commonality in the corresponding collections of Schubert products. A natural next question is whether there is a conceptual \emph{geometric} reason that forces distinct classes to have the same forward stability behavior.

\subsection{Outline of the paper}
In Section~\ref{sec:preliminaries}, we review the combinatorial construction of the forward stability process and the underlying record statistics, along with probabilistic preliminaries. Section~\ref{sec:uniform-expectation} explores the case of uniform permutations, where we derive the expectation and prove a central limit theorem for the stabilization number. Section~\ref{sec:grassmannian} analyzes Grassmannian permutations, focusing on the influence of the descent position on stabilization. Section~\ref{sec:boolean-expectation} investigates Boolean permutations, utilizing their characterization via reduced words and block factorizations. We derive exact counting formulas for left-to-right maxima using Fibonacci recurrences and the Tagiuri--Vajda identity, establishing linear expectation bounds for the stability number. In Section~\ref{sec:backward-stability}, we discuss the relationship between forward and backward stability via conjugation by the longest element. Finally, in the last two sections we suggest several open questions and future directions: Section~\ref{sec:conjectures} concerns the forward stability for other classes of permutations, and Section~\ref{sec:record-equivalence} concerns the record-set equivalence for pattern avoidance classes.

\section{Preliminaries and Notation}\label{sec:preliminaries}

In this section, we give notation for the increments of the random walk $\{Y_i(u,v)\}$ and provide some brief probabilistic preliminaries.

\subsection{Random walk increments}

As mentioned in the introduction, the \emph{forward stability statistic} $\FS(u,v)$ is the maximum of the random walk given by the $Y_i$:
\[
\FS(u,v) \;=\; \max_{1\le i\le m+1} Y_i(u, v).
\]
We now define notation for the increments of this random walk. For $1\le i\le n$, let
\[
\Delta_i(u, v) \;:=\; Y_{i+1}(u, v)-Y_i(u, v).
\]
Using $\Lambda_{i+1}(w)-\Lambda_i(w)=-\chi_i(w)$ (valid for all $i\ge 1$), we obtain
\begin{equation}\label{eq:delta_i}
\Delta_i(u, v) \;=\; 1-\chi_i(u)-\chi_i(v)\in \{-1,0,1\},
\end{equation}
and hence $|\Delta_i(u, v)|\le 1$.

Much of our analysis of $\FS(u,v)$ involves the increments $\Delta_i(u, v)$. The distribution of $\Delta_i$ depends on $i$, and (for some classes of permutations) these increments are mutually dependent.

Also note that for $u, v\sim\text{Unif}(\mathcal{W})$ and $m = \max{(\FS(u), \FS(v))}\leq n$, we have
\begin{equation}\label{eq:upper_inequality}
\FS(u,v) \;=\; \max_{1\leq i\leq m + 1}Y_i(u, v) \leq   \max_{1\leq i\leq n + 1}Y_i(u, v) =: \widetilde{\FS}_n(u,v),
\end{equation}
with equality when $m=n$.

\subsection{Some probabilistic preliminaries}

\begin{definition}[Total variation distance]\label{def:tv}
Let $\Omega$ be a finite set and let $\mu,\nu$ be probability measures on $\Omega$.
The \emph{total variation distance} is (see \cite[Proposition 4.2]{LevinPeres2017Markov})
\begin{equation}\label{tv-definition}
\|\mu-\nu\|_{\mathrm{TV}}
:=\sup_{A\subseteq\Omega}\left|\mu(A)-\nu(A)\right|
=\frac12\sum_{x\in\Omega}\left|\mu(\{x\})-\nu(\{x\})\right|.
\end{equation}
\end{definition}

\begin{definition}[Distributional equivalence]
We say $X \stackrel{d}{=} Y$ to mean that the
random variables $X$ and $Y$ are equal in distribution; that is,
$\Prob{X \le t} = \Prob{Y \le t}$ for all $t \in \mathbb{R}$.
\end{definition} 

Next, we define $\alpha$-mixing, which will be the key probabilistic tool in the proof of Theorem~\ref{thm:BooleanExpectation}.

\begin{definition}\label{def:sigma-generated}
A \emph{$\sigma$-algebra} on a set $\Omega$ is a collection $\mathcal{H}$ of subsets of $\Omega$ such that:
(i) $\Omega\in\mathcal{H}$, (ii) if $A\in\mathcal{H}$ then $A^c\in\mathcal{H}$, and (iii) if $A_1,A_2,\dots\in\mathcal{H}$ then
$\bigcup_{m\ge 1}A_m\in\mathcal{H}$.

Let $(X_i)_{i\in\mathbb{Z}}$ be a sequence of random variables on a finite probability space $\Omega$.
For any finite collection of random variables $\mathcal{S}=(Y_1,\dots,Y_m)$, define an equivalence relation on $\Omega$ by
\[
\omega\sim_{\mathcal S}\omega'
\quad\Longleftrightarrow\quad
(Y_1(\omega),\dots,Y_m(\omega))=(Y_1(\omega'),\dots,Y_m(\omega')).
\]
The equivalence classes form a partition of $\Omega$; we call them the \emph{atoms} determined by $\mathcal S$.
We write $\sigma(\mathcal S)$ for the collection of all unions of these atoms.

Equivalently, $A\subseteq\Omega$ lies in $\sigma(\mathcal S)$ if and only if membership in $A$ is determined by the values of
$(Y_1,\dots,Y_m)$, meaning that whenever $\omega\sim_{\mathcal S}\omega'$ we have $\mathbf{1}_A(\omega)=\mathbf{1}_A(\omega')$.
In particular, $\sigma(\mathcal S)$ is a $\sigma$-algebra. If $\mathcal T$ is an infinite collection of random variables on $\Omega$, we define $\sigma(\mathcal T)$ to be the smallest
$\sigma$-algebra containing $\sigma(\mathcal S)$ for every finite subcollection $\mathcal S\subseteq\mathcal T$.
\end{definition}
\begin{definition}\label{def:rosenblatt-alpha}
For integers $k$ and $h\ge 1$, set
\[
\mathcal{M}_k:=\sigma(X_i:\ i\le k),
\qquad
\mathcal{G}_{k+h}:=\sigma(X_i:\ i\ge k+h).
\]
The (Rosenblatt) strong-mixing coefficient at lag $h$ is
\[
\alpha_X(h):=\sup_{k\in\mathbb{Z}}\ \sup_{A\in\mathcal{M}_k,\ B\in\mathcal{G}_{k+h}}
\left|\Prob{A\cap B}-\Prob{A}\Prob{B}\right|.
\]
We say that $(X_i)$ is \emph{strong mixing} if $\alpha_X(h)\to 0$ as $h\to\infty$, and it is \emph{geometrically strong mixing} if
\[
\alpha_X(h)\le C e^{-c h}
\qquad\text{for some }C,c>0\text{ and all }h\ge 1.
\]
\end{definition}

In our applications $(X_i)$ is a finite sequence, say $(X_i)_{i=0}^{N}$. We interpret its mixing coefficients by extending it to a bi-infinite
sequence $(\widetilde X_i)_{i\in\mathbb{Z}}$ as follows. Choose a fixed $x_\star$ in the state space of $(X_i)_{i=0}^{N}$, and set
\[
\widetilde X_i :=
\begin{cases}
x_\star, & i<0,\\
X_i, & 0\le i\le N,\\
x_\star, & i>N.
\end{cases}
\]
We then define $\alpha_X(h)$ to be the Rosenblatt coefficient $\alpha_{\widetilde X}(h)$ of the bi-infinite sequence $(\widetilde X_i)_{i\in\mathbb{Z}}$
from Definition~\ref{def:rosenblatt-alpha}; this definition is independent of the choice of $x_\star$, since the added coordinates are deterministic
and hence generate only trivial $\sigma$-algebras. This extension only serves to remove boundary effects. Indeed, since $\widetilde X_i$ is constant for
$i<0$ and for $i>N$, the $\sigma$-algebras $\sigma(\widetilde X_i:\ i\le k)$ and $\sigma(\widetilde X_i:\ i\ge k+h)$ are trivial whenever the index
set $\{i\le k\}$ or $\{i\ge k+h\}$ lies entirely outside $\{0,1,\dots,N\}$. For such values of $k$, the inner supremum over
$A\in\sigma(\widetilde X_i:\ i\le k)$ and $B\in\sigma(\widetilde X_i:\ i\ge k+h)$ equals $0$. Consequently, in the definition of
$\alpha_{\widetilde X}(h)$ it is enough to take the supremum over those $k$ for which both index sets $\{i\le k\}$ and $\{i\ge k+h\}$ intersect
$\{0,1,\dots,N\}$.

\section{Uniform Permutations}\label{sec:uniform-expectation}

We now determine the expectation of $\FS(u,v)=\max_{1\le i\le m+1} Y_i(u,v)$ for $u,v\sim\Unif{S_n}$, where $m=\max(\FS(u),\FS(v))$.

\begin{lemma}[R{\'e}nyi's record theorem]\label{lem:renyi_records}
Fix $n\ge 1$ and let $w\sim\Unif{S_n}$. Then the record indicators $\rec_1(w),\dots,\rec_n(w)$ are mutually independent and satisfy, for each $1\le j\le n$,
\[
\Prob{\rec_j(w)=1}=\E{\rec_j(w)}=\frac{1}{j}.
\]
\end{lemma}

\begin{proof}
See R{\'e}nyi~\cite[Lemma~1]{Renyi1962}. \qedhere
\end{proof}

Since $\chi_j(w)=1-\rec_j(w)$, we also have
\begin{equation} \label{eq:chi_marginals}
\Prob{\chi_j(w)=0}=\frac{1}{j},
\qquad
\E{\chi_j(w)}=\frac{j-1}{j}.
\end{equation}

\begin{lemma}\label{lem:Yi_expectation}
Fix $n\ge 1$, and let $u,v\sim \Unif{S_n}$. Then for $1\le i\le n+1$,
\[
\E{Y_i(u, v)}=2n-i+1-2H_n+2H_{i-1},
\]
where $H_i := \sum_{k=1}^{i}\frac{1}{k}$ and $H_0:=0$.
\end{lemma}

\begin{proof}
Let $w\sim\Unif{S_n}$.
By linearity of expectation and~\eqref{eq:chi_marginals},
  \[
  \E{\Lambda_i(w)}=\sum_{j=i}^{n}\E{\chi_j(w)}=(n-i+1)-(H_n-H_{i-1}).
  \]

  Since $u\sim\Unif{S_n}$ and $v\sim\Unif{S_n}$, we have
  $\E{\Lambda_i(u)}=\E{\Lambda_i(v)}$, and therefore
  \[
  \E{Y_i(u, v)}=\E{\Lambda_i(u)}+\E{\Lambda_i(v)}+i-1=2\E{\Lambda_i(w)}+i-1=2n-i+1-2H_n+2H_{i-1}. \qedhere
  \]
\end{proof}

\begin{corollary}\label{lem:prob_delta_i}
For $u, v\sim\Unif{S_n}$, and $1\leq i\leq n$, $\Prob{\Delta_i(u, v)=1}=\dfrac{1}{i^2}$.
\end{corollary}

\begin{proof}
  Fix $1\le i\le n$.  By \eqref{eq:delta_i}, we have $\Delta_i(u,v)=1-\chi_i(u)-\chi_i(v)$,
  so $\Delta_i(u,v)\in\{-1,0,1\}$.  Since $u, v\sim\Unif{S_n}$ are independent, the random variables $\chi_i(u)$ and $\chi_i(v)$ are independent.    By~\eqref{eq:chi_marginals}, $\Prob{\chi_i(w)=0}=\frac{1}{i}$ for $w\sim\Unif{S_n}$, and therefore,
  \[
  \Prob{\Delta_i(u,v)=1}
  =\Prob{\chi_i(u)=0,\ \chi_i(v)=0}
  =\frac{1}{i}\cdot\frac{1}{i}
  =\frac{1}{i^2}.  \qedhere
  \]
  \end{proof}
 \subsection{An upper bound}

We next bound the expected overshoot of the walk $\{Y_i(u,v)\}$ above a fixed index $i$ by comparing its increments to the indicators of upward steps.
\begin{lemma}\label{lem:max_bound}
Fix $n\ge 1$, and let $u,v\sim\Unif{S_n}$.  For every $1\le i\le n$,
\[
\E*{\max_{i\le k\le n+1}\left(Y_k(u,v)-Y_i(u,v)\right)}
\;\le\;\sum_{j=i}^{n}\frac{1}{j^2}.
\]
\end{lemma}

\begin{proof}
Fix $i\in\{1,\dots,n\}$.  For $k\ge i$ we may write
\[
Y_k(u,v)-Y_i(u,v)=\sum_{j=i}^{k-1}\Delta_j(u,v).
\]
Define $I_j(u,v):=\mathbf{1}\{\Delta_j(u,v)=1\}$ for $j\ge i$.  Since $\Delta_j(u,v)\in\{-1,0,1\}$, we have the pointwise inequality
\[
\Delta_j(u,v)\le I_j(u,v)\qquad (j\ge i).
\]
Hence for every $k\in\{i,\dots,n+1\}$,
\[
Y_k(u,v)-Y_i(u,v)
=\sum_{j=i}^{k-1}\Delta_j(u,v)
\le \sum_{j=i}^{k-1}I_j(u,v)
\le \sum_{j=i}^{n}I_j(u,v),
\]
since the partial sums of $\{I_j(u,v)\}$ are nondecreasing in $k$.
Taking the maximum over $k\in\{i,\dots,n+1\}$ gives the pathwise bound
\[
\max_{i\le k\le n+1}\left(Y_k(u,v)-Y_i(u,v)\right)\le \sum_{j=i}^{n}I_j(u,v).
\]
Taking expectations and using linearity, together with Corollary~\ref{lem:prob_delta_i}, yields
\[
\E*{\max_{i\le k\le n+1}\left(Y_k(u,v)-Y_i(u,v)\right)}
\le \sum_{j=i}^{n}\Prob{\Delta_j(u,v)=1}
=\sum_{j=i}^{n}\frac{1}{j^2}. \qedhere
\]
\end{proof}

\begin{proposition}
For $u,v\sim\Unif{S_n}$
\[
\E{\FS(u,v)} = 2n-2\ln n-2\gamma + c + o(1),
\]
where $c\in[1,\pi^2/6]$ and $\gamma$ is the Euler--Mascheroni constant.
\end{proposition}

\begin{proof}
Assume $n\ge2$. Fix $u,v\sim\Unif{S_n}$ and set $m=\max(\FS(u),\FS(v))$.
By~\eqref{eq:upper_inequality}, we have $\FS(u,v)\le \widetilde{\FS}_n(u,v)$, with equality when $m=n$. When $m<n$, both $u(n)=n$ and $v(n)=n$, so $\Prob{m<n}=\Prob{\rec_n(u)=1}\Prob{\rec_n(v)=1}=1/n^2$. On this event, $Y_i(u,v)=i-1$ for $i\ge m+2$, hence $\widetilde{\FS}_n(u,v)=\max\{n,\FS(u,v)\}$ and
\[
0\le \widetilde{\FS}_n(u,v)-\FS(u,v)\le n\,\mathbf{1}\{m<n\}.
\]
Taking expectations gives $\E{\widetilde{\FS}_n(u,v)}-\E{\FS(u,v)}\le n\cdot n^{-2}=o(1)$, so it suffices to prove the stated asymptotics for $\E{\widetilde{\FS}_n(u,v)}$.

Since $\chi_1(w)=0$ for every $w$, we have $\Delta_1(u,v)=1$ and $\widetilde{\FS}_n(u,v)=\max_{2\le i\le n+1}Y_i(u,v)$. Decomposing,
\[
\widetilde{\FS}_n(u,v)
=Y_2(u,v)+\max_{2\le k\le n+1}\left(Y_k(u,v)-Y_2(u,v)\right).
\]
By Lemma~\ref{lem:max_bound} with $i=2$,
\[
0\le \E{\widetilde{\FS}_n(u,v)}-\E{Y_2(u,v)}\le \sum_{j=2}^{n}\frac{1}{j^2}.
\]
By Lemma~\ref{lem:Yi_expectation}, $\E{Y_2(u,v)}=2n+1-2H_n$, and using $H_n=\ln n+\gamma+o(1)$ and $\sum_{j=2}^{n}j^{-2}=\pi^2/6-1+o(1)$,
\[
2n-2\ln n-2\gamma+1+o(1)
\;\le\;
\E{\widetilde{\FS}_n(u,v)}
\;\le\;
2n-2\ln n-2\gamma+\frac{\pi^2}{6}+o(1).
\]

It remains to show that the additive term converges. Define
\[
Z_n(u,v):=\max_{2\le k\le n+1}\left(Y_k(u,v)-Y_2(u,v)\right)\ge 0,
\]
so $\widetilde{\FS}_n(u,v)=Y_2(u,v)+Z_n(u,v)$. Setting $I_j(u,v):=\mathbf{1}\{\Delta_j(u,v)=1\}$, the proof of Lemma~\ref{lem:max_bound} (with $i=2$) gives
\[
0\le \E{Z_n(u,v)}\le \sum_{j=2}^{n}\E{I_j(u,v)}
=\sum_{j=2}^{n}\Prob{\Delta_j(u,v)=1}
=\sum_{j=2}^{n}\frac{1}{j^2}.
\]

Set $a_n:=\E{Z_n(u,v)}$ for $u,v\sim\Unif{S_n}$. We claim $(a_n)$ is nondecreasing. Fix $N>n$ and sample $u,v\sim\Unif{S_N}$. Define
\[
Z_n^{(N)}(u,v):=\max_{2\le k\le n+1}\sum_{j=2}^{k-1}\Delta_j(u,v),
\]
which depends only on $\{\chi_j(u),\chi_j(v):2\le j\le n\}$. By Lemma~\ref{lem:renyi_records} and independence, these indicators have the same joint law as in $S_n$, so $\E[Z_n^{(N)}(u,v)]=a_n$. Since $Z_N(u,v)\ge Z_n^{(N)}(u,v)$ pathwise,
\[
a_N=\E{Z_N(u,v)}\ge \E{Z_n^{(N)}(u,v)}=a_n.
\]
Thus $(a_n)$ is nondecreasing and bounded above by $\pi^2/6-1$, so the limit $C:=\lim_{n\to\infty}a_n$ exists with $0\le C\le \pi^2/6-1$.

Finally,
\[
\E{\FS(u,v)}
=\E{\widetilde{\FS}_n(u,v)}+o(1)
=\left(2n+1-2H_n\right)+C+o(1)
=2n-2\ln n-2\gamma+c+o(1),
\]
where $c:=1+C\in[1,\pi^2/6]$. \qedhere
\end{proof}

\subsection{Asymptotic Distribution}

Having established bounds on the expectation, we now turn to the limiting distribution of the statistic $\FS(u, v)$. We show that $\FS(u,v)=Y_2(u,v)+o_p(\sqrt{\ln n})$, so after normalization by $\sqrt{\ln n}$ the contribution from the maximization over the tail vanishes and the fluctuations are governed by $Y_2(u,v)$. This implies that in the limit the forward stability number follows a Gaussian distribution.

Recall that $m=\max(\FS(u),\FS(v))$ and $\FS(u,v)=\max_{1\le i\le m+1}Y_i(u,v)$. We decompose the statistic relative to the value at $i=2$:
\[
\FS(u, v) = Y_2(u,v) + \max_{1 \le i \le m+1} \left(Y_i(u,v) - Y_2(u,v)\right).
\]
Let
\[
E_n(u, v) := \max_{1 \le i \le m+1} \left(Y_i(u,v) - Y_2(u,v)\right)
\]
denote the \emph{excess} achieved by the process relative to its value at index $2$. Note that $E_n \ge 0$ since the maximum includes the term $Y_2(u,v)-Y_2(u,v)=0$.

The random variable $Y_2(u,v)$ is determined entirely by the number of records (left-to-right maxima) in the permutations $u$ and $v$.

\begin{proposition}[CLT for records]\label{prop:record_clt}
Let $w\sim\Unif{S_n}$ and let
\[
R_n(w):=\sum_{j=1}^n \rec_j(w)
\]
be the number of left-to-right maxima of $w$. Then
\[
\frac{R_n(w)-\ln n}{\sqrt{\ln n}} \xrightarrow{d} \mathcal{N}(0,1).
\]
\end{proposition}

\begin{proof}
By Lemma~\ref{lem:renyi_records}, the indicators $\rec_1(w),\dots,\rec_n(w)$ are mutually independent with
$\Prob{\rec_j(w)=1}=1/j$. Hence $R_n(w)=\sum_{j=1}^n \rec_j(w)$ is a sum of independent Bernoulli random variables.
Set $p_j:=1/j$ and $s_n^2:=\Var{R_n(w)}=\sum_{j=1}^n p_j(1-p_j)$.
Thus, we have $R_n(w) - H_n = \sum_{j=1}^n (\rec_j(w) - p_j)$, where $H_n:=\sum_{j=1}^n p_j$.
We first prove the central limit theorem with the natural centering and scaling,
\[
\frac{R_n(w)-H_n}{s_n}\xrightarrow{d}\mathcal{N}(0,1).
\]
To verify Lyapunov's condition with $\delta=1$, note that for each Bernoulli random variable,
\[
\E*{\left|\,\rec_j(w)-p_j\,\right|^3}\le \E*{\left|\,\rec_j(w)-p_j\,\right|^2}=p_j(1-p_j)\le p_j.
\]
Therefore
\[
\sum_{j=1}^n \E*{\left|\,\rec_j(w)-p_j\,\right|^3}\le \sum_{j=1}^n \frac1j = H_n.
\]
Moreover, for $j\ge 2$ we have $1-p_j=1-1/j\ge 1/2$, so
\[
s_n^2=\sum_{j=1}^n p_j(1-p_j)\ge \sum_{j=2}^n \frac1j\cdot\frac12=\frac12(H_n-1).
\]
Consequently,
\[
\frac{\sum_{j=1}^n \E*{\left|\,\rec_j(w)-p_j\,\right|^3}}{s_n^3}
\;\le\;
\frac{H_n}{\left(\tfrac12(H_n-1)\right)^{3/2}}
\;\xrightarrow[n\to\infty]{}\;0,
\]
and Lyapunov's theorem gives
\[
\frac{R_n(w)-H_n}{s_n}\xrightarrow{d}\mathcal{N}(0,1).
\]

Finally, since $H_n=\ln n+O(1)$ and
\[
s_n^2=\sum_{j=1}^n \frac1j\left(1-\frac1j\right)=H_n-\sum_{j=1}^n\frac1{j^2}=\ln n+O(1),
\]
we have $\frac{s_n}{\sqrt{\ln n}}\to 1$ and $\frac{H_n-\ln n}{\sqrt{\ln n}}\to 0$. Hence
\[
\frac{R_n(w)-\ln n}{\sqrt{\ln n}}
=
\frac{R_n(w)-H_n}{s_n}\cdot \frac{s_n}{\sqrt{\ln n}}
+\frac{H_n-\ln n}{\sqrt{\ln n}}
\xrightarrow{d}\mathcal N(0,1)
\]
by Slutsky's theorem (see \cite[Lemma~2.8(i)--(ii)]{vanDerVaart1998AsymptoticStatistics}).
\end{proof}

\begin{lemma}\label{lemma:distY2}
Let $u,v\sim\Unif{S_n}$ be independent. Then
\[
\frac{Y_2(u,v)-(2n-2\ln n)}{\sqrt{2\ln n}}
\;\xrightarrow{d}\;
\mathcal{N}(0,1)
\qquad\text{as }n\to\infty.
\]
\end{lemma}

\begin{proof}
For $w\in S_n$,
\[
\Lambda_2(w)=\sum_{j=2}^n \chi_j(w)=\sum_{j=2}^n \left(1-\rec_j(w)\right)
=(n-1)-\sum_{j=2}^n \rec_j(w)=n-R_n(w),
\]
since $\rec_1(w)=1$ and $R_n(w)=\sum_{j=1}^n \rec_j(w)$.
Consequently,
\begin{equation}\label{eq:Y2-records}
Y_2(u,v)=\Lambda_2(u)+\Lambda_2(v)+1
=2n+1-\left(R_n(u)+R_n(v)\right).
\end{equation}
Hence, we obtain
\[
\frac{Y_2(u,v)-(2n-2\ln n)}{\sqrt{2\ln n}}
=
\frac{2\ln n+1-\left(R_n(u)+R_n(v)\right)}{\sqrt{2\ln n}}
=
-\frac{\left(R_n(u)+R_n(v)\right)-2\ln n}{\sqrt{2\ln n}}
+\frac{1}{\sqrt{2\ln n}}.
\]
Since $1/\sqrt{2\ln n}\to 0$, it suffices to prove
\[
\frac{\left(R_n(u)+R_n(v)\right)-2\ln n}{\sqrt{2\ln n}}
\;\xrightarrow{d}\;
\mathcal{N}(0,1).
\]
Let $u, v\sim\Unif{S_n}$. Set
\[
X_n:=\frac{R_n(u)-\ln n}{\sqrt{\ln n}},
\qquad
Y_n:=\frac{R_n(v)-\ln n}{\sqrt{\ln n}}.
\]
Since $u$ and $v$ are independent, the random variables $R_n(u)$ and $R_n(v)$ are independent, and hence
$X_n$ and $Y_n$ are independent. Therefore, for all $(s,t)\in\mathbb{R}^2$, we have
\[
\varphi_{(X_n,Y_n)}(s,t)
:=\mathbb{E}\!\left[e^{i(sX_n+tY_n)}\right]
=\mathbb{E}\!\left[e^{isX_n}\right]\mathbb{E}\!\left[e^{itY_n}\right]
=:\varphi_{X_n}(s)\,\varphi_{Y_n}(t).
\]

By Proposition~\ref{prop:record_clt}, we have $X_n\xrightarrow{d}Z_1$ and $Y_n\xrightarrow{d}Z_2$, where
$Z_1,Z_2\sim\mathcal{N}(0,1)$. Hence, by the one-dimensional case of \cite[Theorem~2.13]{vanDerVaart1998AsymptoticStatistics},
\[
\varphi_{X_n}(s)\to \varphi_{Z_1}(s)=e^{-s^2/2},
\qquad
\varphi_{Y_n}(t)\to \varphi_{Z_2}(t)=e^{-t^2/2}.
\]
Combining with the factorization gives, for each $(s,t)\in\mathbb{R}^2$,
\[
\varphi_{(X_n,Y_n)}(s,t)=\varphi_{X_n}(s)\varphi_{Y_n}(t)\to e^{-s^2/2}\,e^{-t^2/2}
=\exp\!\left(-\frac{s^2+t^2}{2}\right).
\]
The limit is the characteristic function of an $\mathbb{R}^2$-valued normal vector with independent standard
normal coordinates (uniqueness of characteristic functions; see \cite[Lemma~2.15]{vanDerVaart1998AsymptoticStatistics}).
Applying the two-dimensional case of \cite[Theorem~2.13]{vanDerVaart1998AsymptoticStatistics}
yields
\[
(X_n,Y_n)\xrightarrow{d}(Z_1,Z_2),
\]
where $Z_1$ and $Z_2$ are independent $\mathcal{N}(0,1)$ random variables.
Therefore, since
\[
\left(\frac{R_n(u)-\ln n}{\sqrt{\ln n}},\ \frac{R_n(v)-\ln n}{\sqrt{\ln n}}\right)
\xrightarrow{d} (Z_1,Z_2),
\]
and the map $g:\mathbb{R}^2\to\mathbb{R}$ defined by $g(x,y)=(x+y)/\sqrt2$ is continuous, the continuous mapping theorem (see \cite[Theorem~2.3 (i)]{vanDerVaart1998AsymptoticStatistics})
implies
\[
\frac{1}{\sqrt2}\left(\frac{R_n(u)-\ln n}{\sqrt{\ln n}}+\frac{R_n(v)-\ln n}{\sqrt{\ln n}}\right)
=
g\!\left(\frac{R_n(u)-\ln n}{\sqrt{\ln n}},\ \frac{R_n(v)-\ln n}{\sqrt{\ln n}}\right)
\xrightarrow{d}
\frac{Z_1+Z_2}{\sqrt2}.
\]
Since $(Z_1,Z_2)$ is a bivariate normal vector with independent $\mathcal{N}(0,1)$ coordinates, the linear
combination $(Z_1+Z_2)/\sqrt2$ is $\mathcal{N}(0,1)$.
\end{proof}

We now show that the remainder term from the maximization is negligible on the $\sqrt{\ln n}$ scale.  By the definition of $E_n$ above, we have $\FS(u,v)=Y_2(u,v)+E_n(u, v)$.

\begin{lemma}\label{lemma:vanishingZn}
For $u,v\sim\Unif{S_n}$, the excess term $E_n(u, v)$ satisfies
\[
\frac{E_n(u, v)}{\sqrt{2\ln n}} \xrightarrow{p} 0
\qquad\text{as } n\to\infty.
\]
\end{lemma}

\begin{proof}
Since $m\le n$, we have the pointwise bound
\[
E_n(u, v)
\le
\max_{2\le i\le n+1}\left(Y_i(u,v)-Y_2(u,v)\right).
\]

Applying Lemma~\ref{lem:max_bound} with $i=2$ yields
\[
\E{E_n(u, v)}
\le
\E*{\max_{2\le i\le n+1}\left(Y_i(u,v)-Y_2(u,v)\right)}
\le
\sum_{j=2}^{n}\frac{1}{j^2}
\le
\sum_{j=2}^{\infty}\frac{1}{j^2}.
\]
Let $\varepsilon>0$. By Markov's inequality,
\[
\Prob*{ \frac{E_n(u, v)}{\sqrt{2\ln n}} > \varepsilon }
=
\Prob{ E_n(u, v) > \varepsilon \sqrt{2\ln n} }
\le
\frac{\E{E_n(u, v)}}{\varepsilon \sqrt{2\ln n}}
\;\xrightarrow[n\to\infty]{}\;0,
\]
since $\E{E_n(u, v)}$ is bounded uniformly in $n$ and $\sqrt{\ln n}\to\infty$.
This proves $E_n(u, v)/\sqrt{2\ln n}\to 0$ in probability.
\end{proof}

We can now state the limiting distribution of the forward stability statistic for uniform permutations.

\theoremUniformMain*
\begin{proof}
By definition of $E_n$,
\[
\FS(u,v)=\max_{1\le i\le m+1}Y_i(u,v)=Y_2(u,v)+E_n,
\]
and hence
\[
\frac{\FS(u, v) - (2n - 2\ln n)}{\sqrt{2 \ln n}}
=
\frac{Y_2(u,v)-(2n-2\ln n)}{\sqrt{2\ln n}}
+
\frac{E_n}{\sqrt{2 \ln n}}.
\]
By Lemma~\ref{lemma:distY2}, the first term converges in distribution to $\mathcal{N}(0, 1)$, and by
Lemma~\ref{lemma:vanishingZn}, the second term converges in probability to $0$.
The result follows from Slutsky's theorem (see \cite[Lemma~2.8 (i)]{vanDerVaart1998AsymptoticStatistics}).
\end{proof}

This confirms that for uniform permutations, the forward stability number is concentrated around $2n-2\ln n$ with fluctuations of order $\sqrt{\ln n}$, and that these fluctuations are governed by the record structure of the underlying permutations.

\section{Grassmannian Permutations}\label{sec:grassmannian}

Recall that $w\in S_n$ is \emph{Grassmannian} if it has at most one descent. Equivalently, every non-identity Grassmannian permutation has a unique descent position $k\in\{1,\dots,n-1\}$ such that
\[
w(1)<w(2)<\cdots<w(k)
\qquad\text{and}\qquad
w(k+1)<w(k+2)<\cdots<w(n),
\]
and the descent occurs between positions $k$ and $k+1$.  Given $k\in\{0,1,\dots,n\}$ and a $k$-subset $V\subseteq[n]$, there is a unique Grassmannian permutation $w=w(k,V)$ whose first $k$ values are the elements of $V$ in increasing order and whose remaining values are the elements of $[n]\setminus V$ in increasing order. The identity permutation is the unique Grassmannian permutation with no descent. Let $\Grass_n$ denote the set of all such Grassmannian permutations in $S_n$, and let $\Unif{\Grass_n}$ denote the uniform probability distribution over this set.

\subsection{Probabilities of records of \texorpdfstring{$\Unif{\Grass_n}$}{Unif(Grass(n))}}

Let $w\sim \Unif{\Grass_n}$.  Let $k_w$ denote the descent position of $w$ (with $k_w:=0$ if $w=\mathrm{id}$), and let
$V_w\subseteq[n]$ be the set of its first $k_w$ values.  Set $M_w:=\max(V_w)$, with $M_w:=0$ if $k_w=0$.  Then
\[
\rec_j(w)=
\begin{cases}
1, & \text{if } j\le k_w \text{ or } j>M_w,\\
0, & \text{if } k_w<j\le M_w.
\end{cases}
\]

\begin{proposition}\label{prop:grass_rec_uniform_exact}
Fix $n\ge 1$ and $j\in[n]$, and let $w\sim \Unif{\Grass_n}$. Then
\[
\Prob{\rec_j(w)=1}
=
\frac{\displaystyle \sum_{k=j}^{n}\binom{n}{k}+2^{j-1}-n}{2^{n}-n}.
\]
\end{proposition}

\begin{proof}
We count Grassmannian permutations with a record at position $j$.
First, $|\Grass_n|=2^n-n$, since for each $k\in\{1,\dots,n-1\}$ there are $\binom{n}{k}-1$ non-identity permutations with descent
position $k$, together with the identity permutation.

The identity contributes $1$ to the count.  If $k\ge j$, then $\rec_j(w)=1$ for every permutation with descent position $k$,
contributing $\binom{n}{k}-1$ for each $k=j,\dots,n-1$.  If $1\le k\le j-1$, then $\rec_j(w)=1$ is equivalent to $M_w<j$, i.e.\
$V_w\subseteq [j-1]$, and this gives $\binom{j-1}{k}-1$ possibilities for each such $k$ (excluding $V_w=[k]$, which would force $w=\mathrm{id}$).
Therefore,
\begin{align*}
\#\{w\in \Grass_n:\rec_j(w)=1\}
&=
1+\sum_{k=j}^{n-1}\left(\binom{n}{k}-1\right)+\sum_{k=1}^{j-1}\left(\binom{j-1}{k}-1\right)\\
&=
1+\left(\sum_{k=j}^{n}\binom{n}{k}-1\right)-(n-j)+\left(2^{j-1}-1\right)-(j-1)\\
&=\sum_{k=j}^{n}\binom{n}{k}+2^{j-1}-n.
\end{align*}
Dividing by $|\Grass_n|=2^n-n$ gives the claimed formula for $\Prob{\rec_j(w)=1}$.
\end{proof}

\subsection{The Modified Distribution \texorpdfstring{$\mathbb{G}_n$}{G\_n}}
For most of this section, we switch to a slightly modified distribution $\mathbb{G}_n$ with cleaner combinatorics than $\Unif{\Grass_n}$. All intermediate analysis below is carried out under $\mathbb{G}_n$, and we transfer the resulting asymptotics back to the uniform model $\Unif{\Grass_n}$ to prove Theorem~\ref{thm:GrassmannianMain} at the end of the section.

We define the \emph{Grassmannian distribution} $\mathbb{G}_n$ on $\Grass_n$ as follows: choose a subset
$V\subseteq[n]$ uniformly from the $2^n$ subsets,
set $k:=|V|$, and output the Grassmannian permutation $w=w(k,V)$ whose first $k$ values are the elements of $V$ in increasing order
and whose remaining values are the complement in increasing order. We refer to $V$ as the \textit{prefix}. When $w\neq \mathrm{id}$, the parameter $k$ is the descent position of $w$; when $w=\mathrm{id}$, the construction has $n+1$ possible values of $k$, and in that case $k$ is simply the sampled encoding parameter.

For $w\neq \mathrm{id}$ this representation is unique, so under $\mathbb{G}_n$ every non-identity Grassmannian permutation occurs with the same probability.
The identity permutation has multiple representations: for each $k\in\{0,1,\dots,n\}$ we have $\mathrm{id}=w\left(k,\{1,2,\dots,k\}\right)$. Consequently, if $w\sim\mathbb{G}_n$, then
\[
\Prob{w=w'}=
\begin{cases}
2^{-n}, & w'\neq \mathrm{id},\\[1mm]
(n+1)\,2^{-n}, & w'=\mathrm{id}.
\end{cases}
\]
Thus, $\mathbb{G}_n$ differs from $\Unif{\Grass_n}$ only by an extra weight on the identity permutation, which is exponentially small. We use the modified distribution $\mathbb{G}_n$ because it provides clean order-statistic formulas and gives the descent position a perfectly symmetric distribution, while the discrepancy from $\Unif{\Grass_n}$ remains confined to the identity event.

The next lemma records the distribution of the sampled descent parameter.
\begin{lemma}\label{lem:grass_binomial_descent}
Let $w \sim \mathbb{G}_n$ be a Grassmannian permutation generated by selecting a prefix subset $V \subseteq [n]$, and let $K := |V|$ be the sampled parameter in this construction. Then $K$ follows a binomial distribution, $K \sim \mathrm{Bin}(n, 1/2)$.
\end{lemma}

\subsection{Structural Decomposition}
The primary goal of this subsection is to establish the structural properties of Grassmannian permutations necessary to prove Proposition~\ref{prop:grass_fs_exact_prob}, which gives an exponential approximation to the forward stability number.

The first two lemmas concern the structure of a single $w\sim \mathbb{G}_n$. Let $M_w = \max(V_w)$ denote the largest value in the prefix, with the convention that $M_w=0$ when $k=0$ (so in particular $M_w=w(k)$ whenever $k\ge 1$). Note that it is always the case that $M_w\ge k$, and this inequality is strict unless $w$ is the identity permutation.
More precisely, Lemma~\ref{lem:grass_FS_equals_M}, Lemma~\ref{lem:grass_lambda} Lemma~\ref{lem:grass_path_shape} Corollary~\ref{cor:grass_fs_value}, Lemma~\ref{lem:grass_tail_bounds} are the ingredients used to prove Proposition\ref{prop:grass_fs_exact_prob}.

\begin{lemma}\label{lem:grass_FS_equals_M}
Let $w=w(k,V)\sim \mathbb{G}_n$. Then $w(j)=j$ for all $j>M_w$. Moreover, if $w\neq \mathrm{id}$ then $M_w\ge 1$ and $\FS(w)=M_w$, while if $w=\mathrm{id}$ then $\FS(w)=1$.
\end{lemma}

\begin{proof}Write $V=\{v_1<\cdots<v_k\}$ and $[n]\setminus V=\{c_1<\cdots<c_{n-k}\}$.
By construction,
\[
w(i)=v_i\quad(1\le i\le k),
\qquad
w(k+t)=c_t\quad(1\le t\le n-k).
\]

Among the values $1,2,\dots,M_w$, exactly $k$ lie in $V$, so exactly $M_w-k$ lie in $[n]\setminus V$.
Equivalently,
\[
\{c_1,\dots,c_{M_w-k}\} = [M_w]\setminus V,
\qquad
c_{M_w-k+r}=M_w+r\ \ \text{for }1\le r\le n-M_w.
\]
Therefore, for each $1\le r\le n-M_w$ we have
\[
w(M_w+r)=w\left(k+(M_w-k)+r\right)=w\left(k+(M_w-k+r)\right)=c_{M_w-k+r}=M_w+r,
\]
so $w(j)=j$ for all $j>M_w$. For $j>n$ this also holds under the standard embedding of $S_n$ into $S_{\mathbb Z_+}$.

Now suppose $w\neq \mathrm{id}$. Then $V\neq\{1,\dots,k\}$, which forces $M_w>k$.
In particular, the position $M_w$ lies in the suffix (equivalently $M_w\ge k+1$), and
\[
w(M_w)=w\left(k+(M_w-k)\right)=c_{M_w-k}.
\]
Since $c_{M_w-k}\in [M_w]\setminus V$, we have $c_{M_w-k}\le M_w-1$, hence $w(M_w)\neq M_w$.
Combined with $w(j)=j$ for all $j>M_w$, this implies that the largest index moved by $w$ is exactly $M_w$,
so $\FS(w)=M_w$. Finally, if $w=\mathrm{id}$ then $w(j)=j$ for all $j\ge 1$, so $\FS(w)=1$ by definition.
\end{proof}

\begin{lemma}\label{lem:grass_lambda}
Let $w=w(k,V) \sim \mathbb{G}_n$. Then
\[
\Lambda_i(w) =
\begin{cases}
M_w - k & \text{if } 1 \le i \le k, \\
M_w - i + 1 & \text{if } k < i \le M_w, \\
0 & \text{if } i > M_w.
\end{cases}
\]
\end{lemma}

\begin{proof}
If $V=\{1,\dots,k\}$ (equivalently, $w$ is the identity permutation), then every position is a record and $\Lambda_i(w)=0$ for all $i$, which agrees with the stated formula since $M_w=k$. Thus, we may assume $V\neq\{1,\dots,k\}$, so $k\ge 1$ and $w(k)=M_w>k$. By definition, $\Lambda_i(w)$ counts the number of indices $j$ with $i\le j\le \FS(w)$ such that $w(j)$ is a non-record.
Since the prefix $w(1), \dots, w(k)$ is increasing, all elements at indices $j \le k$ are records. Thus, non-records occur exclusively in the suffix.
For $j > k$, $w(j)$ is a non-record if and only if $w(j) < \max(w(1), \dots, w(j-1))$. Since $w(k)=M_w$ is the maximum of the prefix and the suffix is increasing, we have $\max(w(1),\dots,w(j-1))=M_w$ for $k<j\le M_w$, while for $j>M_w$ we have $w(j)=j>M_w$ by Lemma~\ref{lem:grass_FS_equals_M} and hence $w(j)$ is a record. Therefore $w(j)$ is a non-record if and only if $w(j)<M_w$.

The set of values in the suffix is $\{1, \dots, n\} \setminus V$, appearing in increasing order. The non-records are precisely the values in this set strictly less than $M_w$. Since $V$ contains $M_w$ and exactly $k-1$ values smaller than $M_w$, the suffix contains exactly $M_w - k$ values smaller than $M_w$. Because the suffix is sorted, these $M_w - k$ values form a contiguous block at the start of the suffix, occupying indices $k+1, \dots, M_w$.

For $i \le k$, the sum $\Lambda_i$ includes all non-records, so $\Lambda_i = M_w - k$.
For $k < i \le M_w$, the function decreases by 1 at each step as the index $i$ passes the positions of the non-records. For $i > M_w$, no non-records remain.
\end{proof}

Next, we turn to the forward stability number. Let $u =u(k_u,V_u), v =v(k_v,V_v)\in \Grass_n$. Define $k_{\min} := \min(k_u, k_v)$ and $k_{\max} := \max(k_u, k_v)$.
Recall that $Y_i(u, v) = \Lambda_i(u) + \Lambda_i(v) + i - 1$.

  \begin{lemma}\label{lem:grass_path_shape}
  If $u, v\neq \mathrm{id}$ and $k_{\max} < \min(M_u, M_v)$, then the sequence $(Y_i(u, v))_{1\le i\le m+1}$ behaves as follows:
  \begin{enumerate}
      \item $Y_i(u, v) < Y_{i+1}(u, v)$ for $1 \le i \le k_{\min}$;
      \item $Y_i(u, v) = Y_{i+1}(u, v)$ for $k_{\min}+1 \le i \le k_{\max}$;
      \item $Y_i(u, v) > Y_{i+1}(u, v)$ for $k_{\max}+1 \le i \le \min(M_u, M_v)$;
      \item $Y_i(u, v) = Y_{i+1}(u, v)$ for $\min(M_u, M_v)+1 \le i \le m$.
  \end{enumerate}
  \end{lemma}

  \begin{proof}
  Given $\Delta_i(u, v) := Y_{i+1}(u, v) - Y_i(u, v)$, we have $\Delta_i(u, v) = \big(\Lambda_{i+1}(u)-\Lambda_i(u)\big) + \big(\Lambda_{i+1}(v)-\Lambda_i(v)\big) + 1$.
  Let $w = w(k_w,V_w)\sim \mathbb{G}_n$. By the explicit formula for $\Lambda_i(w)$ given in Lemma~\ref{lem:grass_lambda}, the increments $\Lambda_{i+1}(w)-\Lambda_i(w)$ are given by
  \[
  \Lambda_{i+1}(w)-\Lambda_i(w)=
  \begin{cases}
  0,& i\le k_w,\\
  -1,& k_w<i\le M_w,\\
  0,& i> M_w.
  \end{cases}
  \]
For $1 \le i \le k_{\min}$, we have $i \le k_u$ and $i \le k_v$. Thus, both $\Lambda$ increments are $0$, giving $\Delta_i(u,v) = 0 + 0 + 1 = 1 > 0$. This implies $Y_i(u, v) < Y_{i+1}(u, v)$.

For $k_{\min} < i \le k_{\max}$, assume without loss of generality that $k_u = k_{\min}$ and $k_v = k_{\max}$. Then $k_u < i \le M_u$ since $i \le k_{\max} < M_u$, yielding a $\Lambda$ increment of $-1$ for $u$. For $v$, we have $i \le k_v$, yielding a $\Lambda$ increment of $0$. Thus, $\Delta_i(u,v) = -1 + 0 + 1 = 0$, implying $Y_i(u, v) = Y_{i+1}(u, v)$.

For $k_{\max} < i \le \min(M_u, M_v)$, we have $i > k_u$ and $i > k_v$, while simultaneously $i \le M_u$ and $i \le M_v$. Thus, both $\Lambda$ increments are $-1$, giving $\Delta_i(u,v) = -1 + (-1) + 1 = -1 < 0$. This implies $Y_i(u, v) > Y_{i+1}(u, v)$.

For $\min(M_u, M_v) < i \le \max(M_u, M_v)$, assume without loss of generality that $M_u \le M_v$, so $\max(M_u, M_v) = M_v$. Then $i > M_u$, yielding a $\Lambda$ increment of $0$ for $u$. For $v$, we have $k_v \le k_{\max} < \min(M_u, M_v) < i \le M_v$, yielding a $\Lambda$ increment of $-1$. Thus, $\Delta_i(u,v) = 0 + (-1) + 1 = 0$, implying $Y_i(u, v) = Y_{i+1}(u, v)$.
  \end{proof}

\begin{corollary}\label{cor:grass_fs_value}
Under the same hypotheses as Lemma~\ref{lem:grass_path_shape}, we have
\[
\FS(u,v) = M_u + M_v - k_{\max}.
\]
\end{corollary}

\begin{proof}
By definition, $\FS(u,v) = \max_{1 \le i \le m+1} Y_i(u, v)$. By Lemma~\ref{lem:grass_path_shape}, the sequence $Y_i(u, v)$ achieves its global maximum anywhere on the plateau $[k_{\min}+1, k_{\max}+1]$. Evaluating $Y_i(u, v)$ at the right endpoint $i = k_{\max}+1$ yields
\[
\FS(u,v) = Y_{k_{\max}+1}(u,v) = (M_u - k_{\max}) + (M_v - k_{\max}) + k_{\max} = M_u + M_v - k_{\max},
\]
by Lemma~\ref{lem:grass_lambda}.
\end{proof}

\begin{lemma}\label{lem:grass_tail_bounds}
  Let $w = w(k_w,V_w)\sim \mathbb{G}_n$. There exists a constant $c > 0$ such that for all sufficiently large $n$, $\Prob{k_w \ge 3n/4} \le e^{-cn}$ and $\Prob{M_w \le 3n/4} \le e^{-cn}$.
  \end{lemma}

  \begin{proof}
  By Lemma~\ref{lem:grass_binomial_descent}, the sampled parameter satisfies $k_w \sim \mathrm{Bin}(n, 1/2)$. The standard Chernoff bound for a binomial random variable (see \cite[Theorem~1.1, Eq.~(1.6)]{DubhashiPanconesi2009}) yields $\Prob{k_w \ge 3n/4} \le e^{-c_1 n}$ for some $c_1 > 0$.

  For $M_w$, recall that $V_w \subseteq [n]$ is chosen uniformly at random and $M_w = \max(V_w)$. For any real $x$, the event $\{M_w \le x\}$ is equivalent to $\{V_w \subseteq [\lfloor x \rfloor]\}$. Therefore,
  \[
  \Prob{M_w \le 3n/4} \le 2^{\lfloor 3n/4 \rfloor - n} \le 2^{-n/4} = e^{-(n \ln 2)/4}.
  \]
  Taking $c = \min(c_1, (\ln 2)/4)$, we obtain the claim.
  \end{proof}

    \begin{proposition}\label{prop:grass_fs_exact_prob}
    Let $u, v \sim \mathbb{G}_n$, and let
    \[
    \Xi \;:=\; \FS(u, v) - \left(M_u + M_v - k_{\max}\right).
    \]
    Then $\Prob{\Xi \neq 0} = O(e^{-cn})$ for some constant $c > 0$. Consequently, $\E{|\Xi|}$ is exponentially small.
    \end{proposition}

    \begin{proof}
    Let $\mathcal{I} := \{u = \mathrm{id}\} \cup \{v = \mathrm{id}\}$. As noted previously, $\Prob{w = \mathrm{id}} = (n+1)2^{-n}$. By the union bound,
    \[
    \Prob{\mathcal{I}} \le 2(n+1)2^{-n} = O(e^{-cn}).
    \]
    Define the event $\mathcal{E} := \{k_{\max} < \min(M_u, M_v)\}$. On the intersection $\mathcal{E} \cap \mathcal{I}^c$, Corollary~\ref{cor:grass_fs_value} applies directly, guaranteeing that $\Xi = 0$. By De Morgan's laws, we have $(\mathcal{E} \cap \mathcal{I}^c)^c = \mathcal{E}^c \cup \mathcal{I}$.
    Then, by the union bound, we obtain
    \[
    \mathbb{P}(\Xi \neq 0) \le \mathbb{P}(\mathcal{E}^c \cup \mathcal{I}) \le \mathbb{P}(\mathcal{E}^c) + \mathbb{P}(\mathcal{I}).
    \]
    To bound $\Prob{\mathcal{E}^c}$, note that if $\mathcal{E}^c$ occurs, then $k_{\max} \ge \min(M_u, M_v)$. This forces at least one of the following four events to occur: $\{k_u \ge 3n/4\}$, $\{k_v \ge 3n/4\}$, $\{M_u \le 3n/4\}$, or $\{M_v \le 3n/4\}$. Applying the union bound alongside Lemma~\ref{lem:grass_tail_bounds} yields
    \[
    \Prob{\mathcal{E}^c} \le 2\Prob{k_w \ge 3n/4} + 2\Prob{M_w \le 3n/4} = O(e^{-cn}).
    \]
    Combining these bounds proves $\Prob{\Xi \neq 0} = O(e^{-cn})$.

    Given that $\FS(u,v) \le 2\max(\FS(u), \FS(v)) \le 2n$. Since $M_u + M_v - k_{\max} \in [0, 2n]$ as well, we have $|\Xi| \le 2n$ deterministically.
    By conditioning on whether $\Xi$ is zero or non-zero, the law of total expectation gives $\E{|\Xi|} = \E{|\Xi| \mid \Xi \neq 0} \Prob{\Xi \neq 0} + \E{|\Xi| \mid \Xi = 0} \Prob{\Xi = 0}$.
    The second term obviously vanishes, since $|\Xi| = 0$ perfectly when $\Xi = 0$. For the first term, since $|\Xi| \le 2n$, $\E{|\Xi| \mid \Xi \neq 0}\leq 2n$.
    Therefore, we have
    \[
    \E{|\Xi|} \le 2n \Prob{\Xi \neq 0} = 2n \, O(e^{-cn}) = O(e^{-c'n})
    \]
    for some $c' > 0$.
    \end{proof}

\subsection{Asymptotic Expectation}

\begin{lemma}\label{lem:expected_max}
Let $w \sim \mathbb{G}_n$. Then $\E{M_w} = n - 1 + 2^{-n}$.
\end{lemma}

\begin{proof}
Recall that $V_w\subseteq[n]$ is a uniformly random subset and $M_w=\max(V_w)$, with the convention $M_w=0$ when $V_w=\varnothing$. For $m\in\{0,1,\dots,n\}$, the event $\{M_w\le m\}$ is equivalent to $\{V_w\subseteq[m]\}$, and hence
\[
\Prob{M_w\le m}=2^{m-n}.
\]
Therefore, for a nonnegative integer-valued random variable,
\[
\E{M_w}=\sum_{m=1}^n \Prob{M_w\ge m}
=\sum_{m=1}^n \left(1-\Prob{M_w\le m-1}\right)
=\sum_{m=1}^n \left(1-2^{m-1-n}\right)
=n-1+2^{-n}. \qedhere
\]
\end{proof}

\begin{lemma}\label{lem:expected_max_descent}
Let $k_u, k_v$ be the descent positions of independent, uniformly sampled Grassmannian permutations $u, v \sim \mathbb{G}_n$. Then
\[
\E{\max(k_u, k_v)} = \frac{n}{2} + \frac{1}{2\sqrt{\pi}}\sqrt{n} + O(n^{-1/2}).
\]
\end{lemma}
\begin{proof}
By Lemma~\ref{lem:grass_binomial_descent}, the sampled descent parameters $k_u$ and $k_v$ are independent and identically distributed as $\mathrm{Bin}(n,1/2)$.
Using the identity $\max(x,y)=\frac{x+y}{2}+\frac{|x-y|}{2}$, we have $\E{\max(k_u,k_v)} = \frac{n}{2} + \frac{1}{2}\E{|k_u-k_v|}$. Define the reversed variable $\widetilde k_v:=n-k_v$. By the symmetry of $p=1/2$, $\widetilde k_v\sim \mathrm{Bin}(n,1/2)$ and remains independent of $k_u$. Thus, their sum $S := k_u + \widetilde k_v \sim \mathrm{Bin}(2n, 1/2)$, allowing us to rewrite the difference as $k_u - k_v = S - n$.

A classical identity of De Moivre (see, e.g., Diaconis--Zabell \cite[Eq.~(1.1)]{DiaconisZabell1991}) gives the exact mean absolute deviation for $S\sim\mathrm{Bin}(2n,1/2)$ as $\E{|S-n|} = n\binom{2n}{n}2^{-2n}$. Applying this, followed by Stirling's approximation $\binom{2n}{n}=\frac{4^n}{\sqrt{\pi n}}\left(1+O(n^{-1})\right)$, yields,
\begin{align*}
\E{\max(k_u,k_v)} &= \frac{n}{2} + \frac{1}{2}\E{|S-n|} = \frac{n}{2}+\frac{n}{2}\binom{2n}{n}2^{-2n} \\
&= \frac{n}{2}+\frac{n}{2} \left[ \frac{4^n}{\sqrt{\pi n}}\left(1+O(n^{-1})\right) \right] 2^{-2n}
= \frac{n}{2}+\frac{1}{2\sqrt{\pi}}\sqrt{n}+O(n^{-1/2}). \qedhere
\end{align*}
\end{proof}

\begin{proposition}\label{prop:grass-exp}
Let $u, v\iidsim \mathbb{G}_n$. As $n\to\infty$,
\[
\E{\FS(u, v)} = \frac{3}{2}n-\frac{1}{2\sqrt{\pi}}\sqrt{n}-2+o(1).
\]
\end{proposition}
\begin{proof}
By Proposition \ref{prop:grass_fs_exact_prob},
\[
\FS(u, v) = M_u + M_v - \max(k_u, k_v) + \Xi,
\]
with $|\E{\Xi}|\le \E{|\Xi|}=O(e^{-cn})$. Taking expectations and using Lemmas \ref{lem:expected_max} and
\ref{lem:expected_max_descent} gives
\begin{align*}
\E{\FS(u, v)}
&=\E{M_u}+\E{M_v}-\E{\max(k_u,k_v)}+\E{\Xi}\\
&=(n-1+2^{-n})+(n-1+2^{-n})-\left(\frac{n}{2}+\frac{1}{2\sqrt{\pi}}\sqrt{n}+o(1)\right)+O(e^{-cn})\\
&=\frac{3}{2}n-\frac{1}{2\sqrt{\pi}}\sqrt{n}-2+2^{1-n}+o(1)+O(e^{-cn})\\
&=\frac{3}{2}n-\frac{1}{2\sqrt{\pi}}\sqrt{n}-2+o(1). \qedhere
\end{align*}
\end{proof}

\subsection{Asymptotic Distribution}
\begin{proposition}\label{prop:grass-dist}
Let $u, v\sim \mathbb{G}_n$, and let $Z_1, Z_2$ be independent standard normal random variables. Then
\[
\frac{\FS(u, v) - \frac{3}{2}n}{\sqrt{n}} \xrightarrow{d} -\frac{1}{2} \max(Z_1, Z_2).
\]
\end{proposition}
\begin{proof}
By Proposition~\ref{prop:grass_fs_exact_prob}, $\FS(u, v) = M_u + M_v - \max(k_u, k_v) + \Xi$. After normalizing by $\sqrt{n}$ and rearranging, we obtain
\[
\frac{\FS(u, v) - \frac{3}{2}n}{\sqrt{n}} = - \frac{(n - M_u)}{\sqrt{n}} - \frac{(n - M_v)}{\sqrt{n}} - \frac{\max(k_u, k_v) - n/2}{\sqrt{n}} + \frac{\Xi}{\sqrt{n}}.
\]

We first show that the $M$ and $\Xi$ terms vanish in probability. By Lemma~\ref{lem:expected_max}, $\E{n - M_u}=1-2^{-n}=O(1)$. Since $n-M_u\ge 0$, Markov's inequality gives, for any $\epsilon>0$,
\[
\mathbb{P}\!\left(\frac{n-M_u}{\sqrt{n}}>\epsilon\right)\le \frac{\E{n-M_u}}{\epsilon\sqrt{n}}\xrightarrow{n\to\infty}0,
\]
so $(n-M_u)/\sqrt{n}\xrightarrow{p}0$, and similarly $(n-M_v)/\sqrt{n}\xrightarrow{p}0$.

By Proposition~\ref{prop:grass_fs_exact_prob}, $\Prob{\Xi\neq 0}=O(e^{-cn})$.
Moreover, since $m\le n$ we have $\FS(u,v)\le 2m\le 2n$, and $M_u+M_v-k_{\max}\in[0,2n]$, so $|\Xi|\le 2n$
deterministically. Hence for any $\epsilon>0$,
\[
\mathbb{P}\!\left(\frac{|\Xi|}{\sqrt{n}}>\epsilon\right)
\le \Prob{\Xi\neq 0}\xrightarrow{n\to\infty}0,
\]
so $\Xi/\sqrt{n}\xrightarrow{p}0$.

Next, by construction the descent positions satisfy $k_u,k_v\stackrel{\text{i.i.d.}}{\sim}\mathrm{Bin}(n,1/2)$. By the Central Limit Theorem, the normalized vector
\[
\left(\frac{k_u-n/2}{\sqrt{n}/2},\frac{k_v-n/2}{\sqrt{n}/2}\right)
\]
converges in distribution to $(Z_1,Z_2)$, where $Z_1,Z_2\sim\mathcal{N}(0,1)$ are independent. By the Continuous Mapping Theorem (see \cite[Theorem~2.3 (i)]{vanDerVaart1998AsymptoticStatistics}) and \cite[Theorem~2.13]{vanDerVaart1998AsymptoticStatistics},
\[
\frac{\max(k_u,k_v)-n/2}{\sqrt{n}}
=\frac12\max\!\left(\frac{k_u-n/2}{\sqrt{n}/2},\frac{k_v-n/2}{\sqrt{n}/2}\right)
\xrightarrow{d}\frac12\max(Z_1,Z_2).
\]

Finally, Slutsky's theorem (see \cite[Lemma~2.8 (i)]{vanDerVaart1998AsymptoticStatistics}) implies that $\Xi/\sqrt{n}$ converging to $0$ in probability vanishes, and therefore
\[
\frac{\FS(u, v) - \frac{3}{2}n}{\sqrt{n}} \xrightarrow{d} -\frac{1}{2}\max(Z_1,Z_2),
\]
as claimed.
\end{proof}

\subsection{Transfer from \texorpdfstring{$\mathbb{G}_n$}{G(n)} to \texorpdfstring{$\Unif{\Grass_n}$}{Unif(Grass(n))}} Proposition~\ref{prop:grass-exp} and Proposition~\ref{prop:grass-dist} give the Grassmannian forward-stability asymptotics under the sampling model $\mathbb{G}_n$.  We now transfer these conclusions to the uniform model $\Unif{\Grass_n}$, proving Theorem~\ref{thm:GrassmannianMain}.

\begin{lemma}\label{lem:grass_TV_single}
For $n\ge 2$,
\[
\left\|\mathbb{G}_n-\Unif{\Grass_n}\right\|_{\mathrm{TV}}
=\frac{n(2^n-n-1)}{2^n(2^n-n)}
\le \frac{n}{2^n}.
\]
\end{lemma}
\begin{proof}
Write $N:=|\Grass_n|=2^n-n$. For $w\sim\mathbb{G}_n$,
\[
\Prob{w=\mathrm{id}}=\frac{n+1}{2^n},
\qquad
\Prob{w=x}=\frac{{1}}{2^n}\ \ \text{for each }x\in\Grass_n\setminus\{\mathrm{id}\}.
\]
For $\widetilde w\sim\Unif{\Grass_n}$ we have $\Prob{\widetilde w=x}=1/N$.
For $x\neq\mathrm{id}$, since $N<2^n$ we have $1/N>1/2^n$, and hence
\[
\left|\frac1{2^n}-\frac1N\right|
=\frac1N-\frac1{2^n}
=\frac{2^n-N}{2^nN}
=\frac{n}{2^nN}.
\]
For $x=\mathrm{id}$ and $n\ge 2$ one checks $(n+1)/2^n>1/N$, so
\[
\left|\frac{n+1}{2^n}-\frac1N\right|
=\frac{n+1}{2^n}-\frac1N
=\frac{(n+1)N-2^n}{2^nN}
=\frac{n(2^n-n-1)}{2^nN}.
\]

Therefore, using Definition~\ref{def:tv} and splitting $\Grass_n$ into $\{\mathrm{id}\}$ and $\Grass_n\setminus\{\mathrm{id}\}$,
\begin{align*}
\left\|\mathbb{G}_n - \Unif{\Grass_n}\right\|_{\mathrm{TV}}
&= \frac{1}{2}\left(
\left|\frac{n+1}{2^n}-\frac{1}{N}\right|+
\sum_{x\in\Grass_n\setminus\{\mathrm{id}\}}
\left|\frac{1}{2^n}-\frac{1}{N}\right|
\right) \\
&=\frac12\left[\frac{n(N-1)}{2^nN} \;+\; \frac{n(N-1)}{2^nN}\right]
=\frac{n(N-1)}{2^nN}
=\frac{n(2^n-n-1)}{2^n(2^n-n)}
\le \frac{n}{2^n}. \qedhere
\end{align*}
\end{proof}

\begin{proof}[Proof of Theorem~\ref{thm:GrassmannianMain}]
  Let $u,v\sim\mathbb{G}_n$ and $\widetilde u,\widetilde v\sim\Unif{\Grass_n}$.
  Let $(u, \widetilde u)$ and $(v, \widetilde v)$ be optimal couplings of $\mathbb{G}_n$ and $\Unif{\Grass_n}$ such that $u,\widetilde{u}$ are independent of $v,\widetilde{v}$.
  By definition of the optimal coupling, we have
  $\mathbb{P}(u \neq \widetilde u) = \|\mathbb{G}_n - \Unif{\Grass_n}\|_{\mathrm{TV}}$ and
  $\mathbb{P}(v \neq \widetilde v) = \|\mathbb{G}_n - \Unif{\Grass_n}\|_{\mathrm{TV}}$.
  By linearity of expectation and Jensen's inequality, we obtain that
  \begin{align*}
    |\E{\FS(u,v)} - \E{\FS(\widetilde u, \widetilde v)}|
    &= |\E{\FS(u,v) - \FS(\widetilde u, \widetilde v)}| \le \E{|\FS(u,v) - \FS(\widetilde u, \widetilde v)|} \\
    &\hspace{-25pt}= \E{|\FS(u,v) - \FS(\widetilde u, \widetilde v)| \cdot \mathbf{1}_{\{(u,v) = (\widetilde u, \widetilde v)\}} + |\FS(u,v) - \FS(\widetilde u, \widetilde v)| \cdot \mathbf{1}_{\{(u,v) \neq (\widetilde u, \widetilde v)\}}} \\
  \end{align*}
  Given that $|\FS(u,v) - \FS(\widetilde u, \widetilde v)| \cdot \mathbf{1}_{\{(u,v) = (\widetilde u, \widetilde v)\}} = 0$ and $|\FS(u,v) - \FS(\widetilde u, \widetilde v)| \le 2n$, we have
  \begin{align*}
    \left|\E{\FS(u,v)} - \E{\FS(\widetilde u, \widetilde v)}\right|&\le \E{|\FS(u,v) - \FS(\widetilde u, \widetilde v)| \cdot \mathbf{1}_{\{(u,v) \neq (\widetilde u, \widetilde v)\}}} \\
    &\le \E{2n \cdot \mathbf{1}_{\{(u,v) \neq (\widetilde u, \widetilde v)\}}} \\
    &= 2n \mathbb{P}\left((u,v) \neq (\widetilde u, \widetilde v)\right) \\
    &\le 2n \left( \mathbb{P}(u \neq \widetilde u) + \mathbb{P}(v \neq \widetilde v) \right) \\
    &= 4n \|\mathbb{G}_n - \Unif{\Grass_n}\|_{\mathrm{TV}}.
  \end{align*}
  Combining this with Lemma~\ref{lem:grass_TV_single}, we obtain
  \[
    \left|\E{\FS(u,v)} - \E{\FS(\widetilde u,\widetilde v)}\right|\le 4n \|\mathbb{G}_n - \Unif{\Grass_n}\|_{\mathrm{TV}} \le \frac{4n^2}{2^n}.
  \]
  Since $4n^2/2^n \to 0$ as $n \to \infty$, $\E{\FS(\widetilde u,\widetilde v)}$ has the same $o(1)$ asymptotic approximation as $\E{\FS(u,v)}$ derived under $\mathbb{G}_n$ in Proposition~\ref{prop:grass-exp}.

  For the distributional convergence, let
  $X_n:=(\FS(u,v)-\frac{3}{2}n)/\sqrt{n}$ and $\widetilde X_n:=(\FS(\widetilde u,\widetilde v)-\frac{3}{2}n)/\sqrt{n}$.
  Since $X_n=\widetilde X_n$ whenever $(u,v)=(\widetilde u,\widetilde v)$, we have
  \begin{align*}
    \Prob{X_n\neq \widetilde X_n}
    &\le \mathbb{P}\left((u,v)\neq (\widetilde u,\widetilde v)\right) \\
    &\le \mathbb{P}(u\neq \widetilde u)+\mathbb{P}(v\neq \widetilde v) \\
    &= 2\|\mathbb{G}_n-\Unif{\Grass_n}\|_{\mathrm{TV}}
    \le \frac{2n}{2^n}
    =o(1).
  \end{align*}
  It follows that, for every $\varepsilon>0$,
  $\Prob{|X_n-\widetilde X_n|>\varepsilon}\le \Prob{X_n\neq \widetilde X_n}=o(1)$,
  so $X_n-\widetilde X_n\to 0$ in probability.
  By Proposition~\ref{prop:grass-dist}, $X_n \Rightarrow -\frac{1}{2}\max(Z_1,Z_2)$.
  Therefore, by Slutsky's theorem (see \cite[Lemma~2.8(i)]{vanDerVaart1998AsymptoticStatistics}),
  $\widetilde X_n \Rightarrow -\frac{1}{2}\max(Z_1,Z_2)$, completing the proof of Theorem~\ref{thm:GrassmannianMain}.
\end{proof}

\section{Boolean Permutations}\label{sec:boolean-expectation}

A \emph{reduced word} for a permutation $w\in S_n$ is a shortest product $s_{a_1}s_{a_2}\cdots s_{a_\ell}$ of simple transpositions $s_i=(i,\,i{+}1)$ that equals $w$.
Since $w$ may admit many reduced words, we compare them lexicographically by their index sequences:
$s_{a_1}\cdots s_{a_\ell} < s_{b_1}\cdots s_{b_\ell}$ if $a_k < b_k$ at the first index $k$ where they differ.
We write $\underline{w}$ for the unique \emph{lexicographically-first} (smallest) reduced word of $w$.

A permutation $w\in S_n$ is \emph{boolean} if it admits a reduced word in which each simple reflection appears at most once. Boolean permutations are fully commutative: any two reduced words for $w$ are connected by commutations $s_as_b=s_bs_a$ for $|a-b|\ge 2$ \cite[Theorem~4.3]{Tenner2007}. In particular, all reduced words for $w$ have the same support, and since commutations preserve generator multiplicities, the existence of a reduced word with no repeated generator implies that no reduced word for $w$ repeats a generator. Let $\mathcal{B}_n$ denote the set of boolean permutations in $S_n$. It is known, see \cite[Corollary~4.4]{Tenner2007}, that for $n\ge 1$,
\begin{equation}\label{eq:bool-cardinality}
|\mathcal{B}_n|=F_{2n-1}.
\end{equation}
For this section, let $w\sim\Unif{\mathcal{B}_n}$. We will use the shorthand $\chi_i := \chi_i^{\mathcal{B}_n}$.

\begin{remark}
  \label{def:cont-comm-inv}
  A \emph{contiguous commuting inversion} in a reduced word is an adjacent subword $s_a s_b$ with $a > b+1$.
  Such a pair commutes since $|a-b|\ge 2$ based on the fact that $w$ is boolean, so swapping it to $s_b s_a$ yields a lexicographically smaller reduced word for the same permutation.
  Consequently, the lexicographically-first reduced word $\underline{w}$ contains no contiguous commuting inversion.
  \end{remark}

\subsection{Structural Decomposition}

We derive two formulas for the number
\[
N(n,j)\ :=\ \left|\{\,w\in\mathcal{B}_n:\ \rec_j(w)=1\,\}\right|
\]
of boolean permutations $w\in\mathcal{B}_n$ having a record at position $j$. In particular,
\[
\Prob{\rec_j(w)=1}=\frac{N(n,j)}{|\mathcal{B}_n|}=\frac{N(n,j)}{F_{2n-1}}.
\]

For any pair $j\ge i$, let $s_{j\downarrow i}$ denote the product
\[
s_{j\downarrow i}:=s_js_{j-1}\cdots s_i,
\]
and call such an element a \emph{(decreasing) block}. Any reduced word of a permutation has a canonical decomposition as a product of blocks by choosing the blocks to be as large as possible. For instance, the block factorization of the reduced expression
\[
s_2s_3s_1s_7s_6s_5s_8s_7s_6
\]
is
\[
s_{2\downarrow 2}\,s_{3\downarrow 3}\,s_{1\downarrow 1}\,s_{7\downarrow 5}\,s_{8\downarrow 6}.
\]

\subsubsection{A block-boundary encoding}

Fix $n\ge 2$ and let $w\in\mathcal{B}_n$. Let $\underline{w}$ denote the lexicographically-first reduced word of $w$ and write its
maximal block factorization
\[
\underline{w}=s_{j_1\downarrow i_1}\,s_{j_2\downarrow i_2}\cdots s_{j_k\downarrow i_k},
\qquad
i_1\le j_1<i_2\le j_2<\cdots<i_k\le j_k.
\]

For each $t\in[n-1]$, define
\begin{align*}
\mathsf{Pres}_t(w)&:=\mathbf{1}\{s_t\text{ appears in }\underline{w}\},\\
\mathsf{Cont}_t(w)&:=\mathbf{1}\{\text{some block }s_{j\downarrow i}\text{ of }\underline{w}\text{ contains }s_t\text{ and }i<t\}.
\end{align*}

Now set the \emph{block-boundary symbol} $\mathsf{Tag}_t(w)\in\{0,S,C\}$ by
\[
\mathsf{Tag}_t(w):=
\begin{cases}
0, & \mathsf{Pres}_t(w)=0,\\
S, & \mathsf{Pres}_t(w)=1\ \text{and}\ \mathsf{Cont}_t(w)=0,\\
C, & \mathsf{Pres}_t(w)=1\ \text{and}\ \mathsf{Cont}_t(w)=1.
\end{cases}
\]

If $t$ is outside the range $1\leq t\leq n-1$, we use the convention $\mathsf{Pres}_t(w) = \mathsf{Cont}_t(w) = \mathsf{Tag}_t(w) = 0$.

\begin{definition}\label{def:block-tag-word}
For each $m\ge 1$, a \emph{block-tag word} of length~$m$ is a word $T=(T_1,\dots,T_m)\in\{0,S,C\}^m$ such that
\begin{enumerate}[label=(\roman*)]
\item $T_1\in\{0,S\}$, and
\item for every $2\le r\le m$, the condition $T_r=C$ forces $T_{r-1}\in\{S,C\}$.
\end{enumerate}
 For each $m\ge 1$, let $\mathcal{T}_m$ denote the set of block-tag words of length~$m$, and write $\mathcal{T}_0:=\{\varnothing\}$.
\end{definition}

\begin{lemma}\label{lem:block-boundary-encoding}
The map
\[
w\ \longmapsto\ \mathsf{Tag}(w):=\left(\mathsf{Tag}_1(w),\dots,\mathsf{Tag}_{n-1}(w)\right)\in\{0,S,C\}^{n-1}
\]
is a bijection from $\mathcal{B}_n$ onto $\mathcal{T}_{n-1}$.
\end{lemma}
\begin{proof}
Given $w$, we clearly have $\mathsf{Tag}_1(w)\neq C$, since no block can contain $s_1$ with minimal generator index strictly less than $1$.
For $t\ge 2$, if $\mathsf{Tag}_t(w)=C$, then some block, say $s_{j\downarrow i}$, of $\underline{w}$ contains $s_t$ and has minimal generator index $i<t$. Since the
block is decreasing, it also contains $s_{t-1}$. Hence $\mathsf{Pres}_{t-1}(w)=1$, and therefore
$\mathsf{Tag}_{t-1}(w)\in\{S,C\}$. This proves that $\mathsf{Tag}(w)$ satisfies (i)--(ii).

Conversely, let $T=(T_1,\dots,T_{n-1})$ satisfy (i)--(ii). We reconstruct a reduced word
$\underline{w_T}$ as follows. Let
\[
I:=\{t:\ T_t\in\{S,C\}\}\subseteq[n-1],
\]
and partition $I$ into maximal consecutive intervals
\[
[a_1,b_1],\ [a_2,b_2],\ \dots,\ [a_m,b_m]
\qquad\text{with}\qquad
a_1\le b_1<a_2\le b_2<\cdots<a_m\le b_m.
\]
For the first interval, its first entry cannot be $C$ by~(i). For every later interval, the preceding symbol is $0$, so condition~(ii) implies that its first entry cannot be $C$. Hence the first entry in each interval is labeled $S$. Within each interval $[a_q,b_q]$, further split at the positions
where $T_t=S$. This yields a collection of consecutive intervals
\[
[i_1,j_1],\ [i_2,j_2],\ \dots,\ [i_k,j_k]
\qquad\text{with}\qquad
i_1\le j_1<i_2\le j_2<\cdots<i_k\le j_k,
\]
and we define
\[
\underline{w_T}:=s_{j_1\downarrow i_1}\,s_{j_2\downarrow i_2}\cdots s_{j_k\downarrow i_k}.
\]

Each simple reflection appears at most once in $\underline{w_T}$ by construction. Therefore no braid relation
$s_as_{a+1}s_a=s_{a+1}s_as_{a+1}$ can occur in $\underline{w_T}$, since such a relation already requires a repeated generator.
Moreover, commutation relations preserve the multiplicity of each generator, so starting from $\underline{w_T}$ no sequence of commutations
can ever produce a subword $s_a^2$. Thus no length-reducing relation is ever available, and $\underline{w_T}$ is a reduced word for a
Boolean permutation $w_T\in\mathcal{B}_n$.

Moreover, the block factorization of $\underline{w_T}$ is maximal by construction. Indeed, each interval $[i_r,j_r]$ contributes
exactly the block $s_{j_r\downarrow i_r}$, and we cut precisely at the positions labeled $S$. It follows that for $t\notin I$ the generator
$s_t$ does not appear, so $\mathsf{Tag}_t(w_T)=0$. If $t=i_r$ for some $r$, then $t$ is the minimal generator index of its block,
so $\mathsf{Tag}_t(w_T)=S$. If $i_r<t\le j_r$, then the block $s_{j_r\downarrow i_r}$ contains $s_t$ and satisfies $i_r<t$, so
$\mathsf{Tag}_t(w_T)=C$. Thus $\mathsf{Tag}(w_T)=T$.

It remains to check that $\underline{w_T}$ is the lexicographically-first reduced word of $w_T$. Since
$w_T$ is Boolean, it is fully commutative, so any two reduced words for $w_T$ are connected by commutation moves
$s_as_b=s_bs_a$ with $|a-b|\ge 2$. In this fully commutative setting, Corollary~6.2 of Hersh~\cite{Hersh05} shows that
the only condition to check for a reduced word to be lexicographically first is condition~(1) of that corollary, namely that every adjacent commuting pair
appears in increasing-index order. In $\underline{w_T}$, this condition holds; within a block the indices decrease by $1$, so no
adjacent commuting pair occurs, while across distinct consecutive blocks the indices increase by the ascending-block property.
Therefore $\underline{w_T}$ is the lexicographically-first reduced word of $w_T$.

We have shown that if $T$ satisfies (i)--(ii), then $\mathsf{Tag}(w_T)=T$, and earlier we showed that
$\mathsf{Tag}(w)$ satisfies (i)--(ii) for every $w\in\mathcal{B}_n$. Therefore $w\longmapsto \mathsf{Tag}(w)$ is a bijection
from $\mathcal{B}_n$ onto the set of sequences $T\in\{0,S,C\}^{n-1}$ satisfying (a)--(b), with inverse
$T\longmapsto w_T$.
\end{proof}

\begin{lemma}\label{lem:chi-as-local}
Let $w\in\mathcal{B}_n$. For every $j\in[n]$,
\[
\chi_j(w)=
\mathbf{1}\!\left\{\,\mathsf{Tag}_{j-1}(w)\in\{S,C\}\ \text{and}\ \mathsf{Tag}_j(w)\neq S\,\right\}.
\]
\end{lemma}

\begin{proof}
The case $j=1$ is immediate, since position $1$ is always a record and $\mathsf{Tag}_j(w)=0$. Fix $2\le j\le n-1$.

\noindent\textit{Case 1:} $\mathsf{Tag}_{j-1}(w)=0$. Equivalently, $s_{j-1}$ does not appear in $\underline{w}$. Then no swap crosses the boundary between positions $j-1$ and $j$, so the
first $j-1$ positions still contain exactly the values $\{1,\dots,j-1\}$. Hence every entry in positions $1,\dots,j-1$ is less than $j$, so $j$ is
a record and $\chi_j(w)=0$.

\noindent\textit{Case 2:} $\mathsf{Tag}_{j-1}(w)\in\{S,C\}$ and $\mathsf{Tag}_j(w)=0$. Then $s_{j-1}$ appears, while $s_j$ does not appear. We read $\underline{w}$ from left to right as a sequence of adjacent swaps applied to the identity word
$1\,2\,\cdots\,n$, and each $s_i$ swaps the entries in positions $i$ and $i+1$.  Since $w$ is boolean, $s_{j-1}$ appears exactly once in $\underline{w}$, so before this occurrence every swap has index $\neq j-1$ and hence preserves the set of values occupying the first $j-1$ positions, which remains $\{1,2,\dots,j-1\}$.  Applying $s_{j-1}$ swaps positions $j-1$ and $j$, bringing some value smaller than $j$ into position $j$.  No later swap can change the entry in position $j$, since that would require another occurrence of $s_{j-1}$ or $s_j$.  Consequently, $w(j)<j$, forcing $\rec_j(w)=0$. Hence $\chi_j(w)=1$.

\noindent\textit{Case 3:} $\mathsf{Tag}_{j-1}(w)\in\{S,C\}$ and $\mathsf{Tag}_j(w)=C$. Then $s_{j-1}$ and $s_j$ both appear, and the block of $\underline{w}$ containing $s_j$ has minimal generator index strictly less than
$j$. Equivalently, that block contains both $s_j$ and $s_{j-1}$, and neither occur anywhere else in the word. Applying $s_j$ swaps positions $j$ and $j+1$, bringing some value greater than $j$ into position $j$, and the subsequent $s_{j-1}$ moves this value into position $j-1$, leaving position $j$ occupied by a value less than or equal $j$.  No later swap can move that greater than $j$ value back to the right of position $j-1$ or change the entry in position $j$, since that would require another occurrence of $s_{j-1}$ or $s_j$.  Consequently, the final permutation has a value greater than $j$ among its first $j-1$ entries while $w(j)\le j$, so $\rec_j(w)=0$. Hence $\chi_j(w)=1$.

\noindent\textit{Case 4:} $\mathsf{Tag}_{j-1}(w)\in\{S,C\}$ and $\mathsf{Tag}_j(w)=S$. Then $s_{j-1}$ and $s_j$ both appear, and the block containing $s_j$ has minimal generator index exactly $j$, so it has the form
$s_{k\downarrow j}$ for some $k\ge j$. Before this block is applied, no occurrence of $s_j$ has yet appeared, so the first $j$ positions
still contain exactly the values $\{1,\dots,j\}$. The letter $s_j$ swaps positions $j$ and $j+1$, thereby importing some value greater
than $j$ into position $j$. Since the block has minimal index $j$, it contains no $s_{j-1}$, so this value is not moved further to the
left within the block. By the ascending block factorization, every later block uses only indices strictly greater than $j$, and hence fixes
positions $1,\dots,j$. Therefore, in the final permutation, the entry in position $j$ is greater than every entry in positions $1,\dots,j-1$,
so $\rec_j(w)=1$ and thus $\chi_j(w)=0$.

These four cases prove the stated formula for $2\le j\le n-1$. For $j=n$, we have $\rec_n(w)=1$ if and only if $w(n)=n$. For Boolean permutations, this is equivalent to the absence of $s_{n-1}$ in any
reduced word, hence to $\mathsf{Tag}_{n-1}(w)=0$, and $\mathsf{Tag}_n(w)$ is guaranteed to be 0. Therefore $\chi_n(w)=1-\rec_n(w)=\mathbf{1}\!\left\{\,\mathsf{Tag}_{n-1}(w)\in\{S,C\}\,\right\}$, as claimed.
\end{proof}

\subsubsection{Formulas for $N(n,j)$}

For $I\subseteq[n]$, let $S_I$ be the subgroup of $S_n$ consisting of permutations that fix $[n]\setminus I$ pointwise, and set $\mathcal{B}_I:=\mathcal{B}_n\cap S_I$. Let $\mathcal{B}_{n,j}:=\{\,w\in\mathcal{B}_n: \rec_j(w)\!=\!1 \}$; the next lemma shows that this set has a simple structure.

\begin{lemma}\label{lem:bool-rec-recurrence}
Define $\mathcal{C}_{1,j}=\{w\in \mathcal{B}_{n,j}: \mathsf{Tag}_{j-1}(w)=0\}$ and $\mathcal{C}_{2,j}=\mathcal{B}_{n,j}\setminus \mathcal{C}_{1,j}$.
Then every $w\in \mathcal{B}_{n,j}$ admits a unique factorization $w=w_1w_2$ with
\[
w_1\in \mathcal{B}_{[j]}\subseteq \mathcal{B}_n,
\qquad
w_2\in \mathcal{B}_{[j,n]}\subseteq \mathcal{B}_n,
\]
and the map $w\mapsto (w_1,w_2)$ restricts to bijections
\[
\mathcal{C}_{1,j}\ \longrightarrow\ \mathcal{B}_{[j-1]}\times \mathcal{B}_{[j,n]},
\qquad\qquad
\mathcal{C}_{2,j}\ \longrightarrow\ (\mathcal{B}_{[j]}\setminus \mathcal{B}_{[j-1]})\times (\mathcal{B}_{[j,n]}\setminus \mathcal{B}_{[j+1,n]}).
\]
\end{lemma}

\begin{proof}
By Lemma~\ref{lem:chi-as-local}, $\mathcal{C}_{2,j} = \{w\in \mathcal{B}_{n,j}: \mathsf{Tag}_{j-1}(w)\in\{S,C\} \ \text{and}\ \mathsf{Tag}_j(w)= S\}$. In the proof of that lemma, $\mathcal{C}_{1,j}$ is Case 1, while $\mathcal{C}_{2,j}$ is Case 4.

By the proof of Lemma~\ref{lem:block-boundary-encoding}, $\underline{w}$ is an ascending product of decreasing blocks. Call a simple reflection $s_i$ with $i<j$ a \emph{small simple reflection}, and $s_i$ with $i\geq j$ a \emph{large simple reflection}. We claim that in $\mathcal{B}_{n,j}$ no small simple reflection shares a block with a large simple reflection. For $\mathcal{C}_{1,j}$, this is because $s_{j-1}$ doesn't appear in $\underline{w}$, while for $\mathcal{C}_{2,j}$, it is because $\mathsf{Cont}_j(w)=S$, indicating that the block containing $s_j$ has the form $s_{k\downarrow j}$. Hence no block of $\underline{w}$ meets both index sets
$\{1,\dots,j-1\}$ and $\{j,\dots,n-1\}$, and since $\underline{w}$ is an ascending product of blocks, every generator $s_i$ with $i\le j-1$
occurs to the left of every generator $s_i$ with $i\ge j$.  Thus $\underline{w}$ splits as a concatenation
\[
\underline{w}=\underline{w_1}\,\underline{w_2},
\]
where $\underline{w_1}$ consists of the small simple reflections appearing in $\underline{w}$ and $\underline{w_2}$ consists of the large simple reflections appearing in $\underline{w}$.  Let $w_1,w_2$ be the permutations represented by these reduced words.  By
construction, $w_1\in\mathcal{B}_{[j]}$ and $w_2\in\mathcal{B}_{[j,n]}$, and $w=w_1w_2$.  This factorization is unique since
$S_{[j]}\cap S_{[j,n]}=\{1\}$.

Now, if $w_1\in \mathcal{B}_{[j-1]}$, then $s_{j-1}$ does not appear, so $\underline{w_1}$ uses only generators $s_i$ with $i\le j-2$ and
the first $j-1$ positions of $w$ contain only values less than $j$.  Hence $w\in\mathcal{C}_{1,j}$, while $w_2$ may be any element of
$\mathcal{B}_{[j,n]}$, giving the first bijection.

On the other hand, if $w_1\notin \mathcal{B}_{[j-1]}$ then $s_{j-1}$ does appear, and as argued above, $s_j$ also
appears, hence $w_2\notin \mathcal{B}_{[j+1,n]}$.  Conversely, fix
\[
w_1\in \mathcal{B}_{[j]}\setminus \mathcal{B}_{[j-1]},
\qquad
w_2\in \mathcal{B}_{[j,n]}\setminus \mathcal{B}_{[j+1,n]},
\]
and set $w:=w_1w_2$.  Then $s_{j-1}$ appears in $\underline{w_1}$, and since $s_{j-1}$ is the only generator in $S_{[j]}$ that can move the
value $j$ out of position $j$, this forces the value $j$ to lie in some position less than $j$ after applying $w_1$; because $w_2\in S_{[j,n]}$
fixes positions less than $j$ pointwise, this remains true in $w$, so $w\notin\mathcal{C}_{1,j}$.

Moreover, since $s_j$ appears in $\underline{w_2}$ and appears nowhere else, at the moment $s_j$ is applied no value less than or equal $j$ can be in
position $j+1$.  Hence this swap imports a value greater than $j$ into position $j$.
Because $w_2$ fixes positions less than $j$, all entries in positions less than $j$ are less than or equal $j$. Thus, $\rec_j(w)=1$ and therefore $w\in\mathcal{C}_{2,j}$.
This gives the second bijection.
\end{proof}

\begin{example}
Let $n=5$ and $j=3$. The sets $\mathcal{C}_{1,j}$ and $\mathcal{C}_{2,j}$
are
\[
\mathcal{C}_{1,j}
=
\{1,\, s_3,\, s_4,\, s_3s_4,\, s_4s_3,\, s_1,\, s_1s_3,\, s_1s_4,\, s_1s_3s_4,\, s_1s_4s_3\},
\]
\[
\mathcal{C}_{2,j}
=
\{s_2s_3,\, s_2s_3s_4,\, s_2s_4s_3,\, s_1s_2s_3,\, s_1s_2s_3s_4,\, s_1s_2s_4s_3,\, s_2s_1s_3,\, s_2s_1s_3s_4,\, s_2s_1s_4s_3\},
\]
where we list elements of $\mathcal{B}_{5,3}$ by their lexicographically-first reduced word.  In each case, $w_1$ is represented by the subword of $\underline{w}$ consisting of the letters $s_i$ with $i\le 2$, and $w_2$ is
represented by the subword of $\underline{w}$ consisting of the letters $s_i$ with $i\ge 3$.
\end{example}

The decomposition in Lemma~\ref{lem:bool-rec-recurrence} gives a first expression for $N(n,j)$.

\begin{corollary}\label{cor:Nnj-fib}
For all $n\ge j\ge 1$,
\begin{equation}\label{eq:Nnj-fib}
N(n,j)
=
F_{2j-3} F_{2n-2j+1}
+
(F_{2j-1}-F_{2j-3})(F_{2n-2j+1} - F_{2n-2j-1}).
\end{equation}
\end{corollary}

Next, we prove a second, simpler expression for $N(n,j)$.  We will use the Tagiuri--Vajda identity~\cite[Chapter~III, Eq.~(20a)]{V89}, valid for all integers $r,a,b\in\mathbb Z$,
\begin{equation}\label{eq:tagiuri-vajda}
F_{r+a}F_{r+b}-F_rF_{r+a+b}=(-1)^rF_aF_b.
\end{equation}

\begin{lemma}\label{lem:fib-identity}
For all integers $m,k\in\mathbb Z$,
\begin{equation}\label{eq:fib-identity}
F_{m+k+1}=F_{m+1}F_{k+1}+F_mF_k.
\end{equation}
\end{lemma}

\begin{proof}
Apply \eqref{eq:tagiuri-vajda} with $(r,a,b)=(1,m,k)$.
\end{proof}

\begin{proposition}\label{prop:two-expressions}
For all integers $n,j\in\mathbb Z$, one has
\[
F_{2j-3}F_{2n-2j+1} + (F_{2j-1}-F_{2j-3})(F_{2n-2j+1}-F_{2n-2j-1})
=
F_{2n-2}-F_{2j-4}F_{2n-2j-1}.
\]
\end{proposition}

\begin{proof}
Using $F_{2j-1}-F_{2j-3}=F_{2j-2}$ and
$F_{2n-2j+1}-F_{2n-2j-1}=F_{2n-2j}$, the left-hand side simplifies to
\[
F_{2j-3}F_{2n-2j+1}+F_{2j-2}F_{2n-2j}.
\]
Since $F_{2n-2j}=F_{2n-2j+1}-F_{2n-2j-1}$, this becomes
\[
(F_{2j-3}+F_{2j-2})F_{2n-2j+1}-F_{2j-2}F_{2n-2j-1}
=
F_{2j-1}F_{2n-2j+1}-F_{2j-2}F_{2n-2j-1}.
\]
Write $F_{2j-2}=F_{2j-3}+F_{2j-4}$ to obtain
\[
F_{2j-1}F_{2n-2j+1}-F_{2j-2}F_{2n-2j-1}
=
\left(F_{2j-1}F_{2n-2j+1}-F_{2j-3}F_{2n-2j-1}\right)
-
F_{2j-4}F_{2n-2j-1}.
\]
It remains to show
\begin{equation}\label{eq:intermediateFibIdent}
F_{2j-1}F_{2n-2j+1}-F_{2j-3}F_{2n-2j-1}=F_{2n-2}.
\end{equation}
Apply Lemma~\ref{lem:fib-identity} with $(m,k)=(2j-2,\,2n-2j)$ to get
\[
F_{2n-1}=F_{2j-1}F_{2n-2j+1}+F_{2j-2}F_{2n-2j},
\]
and with $(m,k)=(2j-3,\,2n-2j-1)$ to get
\[
F_{2n-3}=F_{2j-2}F_{2n-2j}+F_{2j-3}F_{2n-2j-1}.
\]
Subtracting yields
\[
F_{2n-1}-F_{2n-3}=F_{2j-1}F_{2n-2j+1}-F_{2j-3}F_{2n-2j-1}.
\]
Since $F_{2n-1}-F_{2n-3}=F_{2n-2}$, \eqref{eq:intermediateFibIdent} follows, proving the proposition.
\end{proof}

Combining~\eqref{eq:Nnj-fib} and Proposition~\ref{prop:two-expressions}
yields the main result of this subsection.

\begin{theorem}\label{thm:Nnj-recurrence}
For all integers $n\ge 1$ and all $1\le j\le n$,
\[
N(n,j) = F_{2n-2} - F_{2j-4}F_{2(n-j)-1},
\]
where $F_{-n} = (-1)^{n+1}F_n$.
\end{theorem}

\subsubsection{Probability of records}

In this subsection, we record a closed form for the count $N(n,j)$ and for the probability $\Prob{\rec_j(w)=1}$ when
$w\sim\Unif{\mathcal{B}_n}$, both written in terms of the golden ratio $\phi$.  We then prove two exponential comparison bounds for these
probabilities under the changes $(n,j)\mapsto(n+1,j)$ and $(n,j)\mapsto(n+1,j+1)$.  We use Binet's formula
\[
F_n = \frac{\phi^n - (-\phi)^{-n}}{\sqrt{5}},
\qquad
\phi = \frac{1+\sqrt{5}}{2},
\]
together with $\phi-\phi^{-1}=1$.

Using this expression, we record the following closed form for $N(n,j)$, obtained from Theorem~\ref{thm:Nnj-recurrence} by substituting
Binet's formula and simplifying.
\begin{equation}\label{eq:Nnj-exact}
N(n,j) = \frac{2}{5}\phi^{2n-2} + \frac{2}{5}\phi^{2-2n} + \frac{1}{5}\phi^{2n-4j+3} - \frac{1}{5}\phi^{4j-2n-3}.
\end{equation}
Since $|\mathcal{B}_n|=F_{2n-1}=\frac{\phi^{2n-1}+\phi^{1-2n}}{\sqrt5}$, we also have for $w\sim\Unif{\mathcal{B}_n}$,
\begin{equation}\label{eq:bool-rec-prob-phi}
\Prob{\rec_j(w)=1}
=
\frac{N(n,j)}{F_{2n-1}}
=
\frac{2\phi^{2n-2} + 2\phi^{2-2n} + \phi^{2n-4j+3} - \phi^{4j-2n-3}}{\sqrt{5}\,(\phi^{2n-1} + \phi^{1-2n})}.
\end{equation}

\subsection{The Expectation of Forward Stability of Boolean Permutations}\label{sec:boolean-fs-expectation}

We now prove that, for $u,v\iidsim\Unif{\mathcal{B}_n}$, $\E{\FS(u,v)}=n+O(1)$.
The proof is organized around a finite-state encoding of a uniform Boolean permutation by its presence process and associated tag word.
From this encoding we recover the local record indicators, and hence the increments $\Delta_i(u,v)$ of the forward-stability walk, while losing only a bounded amount of lag when passing mixing estimates through the successive derived processes.
We begin by proving geometric strong mixing for the underlying finite-state process and then transfer those bounds to the increment process.
Next we establish a uniform drift gap for the increments and combine it with the mixing estimates to obtain a block Bernstein tail bound for partial sums of the centered increments.
We then apply this concentration estimate to the forward-stability walk and deduce that its maximum has expectation $n+O(1)$.

\subsubsection{Geometric strong mixing for the increment process}\label{subsec:boolean-mixing}

We give a finite-state (transfer-matrix) description of the block-tag word associated with a uniform random Boolean permutation. This implies geometric
(Rosenblatt) strong mixing, also called $\alpha$-mixing, for the record-indicator process $(\chi_i)_{i=1}^n$, and hence for the increment
process $\Delta_i(u,v)$. We use the $\sigma$-algebra and $\alpha$-mixing framework from Section~\ref{sec:preliminaries}
(Definitions~\ref{def:rosenblatt-alpha} and the finite-sequence extension procedure described there).

The following lemma will be used to pass mixing bounds from a simple underlying process to the derived processes $(\chi_i)$ and $(\Delta_i)$.

\begin{lemma}\label{lem:alpha-factor}
\begin{enumerate}[label=(\roman*)]
\item Let $(V_i)_{i=0}^{N}$ be a finite sequence of random variables, and let $(Y_i)_{i=0}^{N}$ be a coordinatewise function of $V$,
meaning that there exists a function $g$ such that
\[
Y_i=g(V_i)\qquad(0\le i\le N).
\]
Then for all $h\ge 1$,
\[
\alpha_Y(h)\le \alpha_V(h).
\]

\item Let $(A_i)_{i=0}^{N}$ and $(B_i)_{i=0}^{N}$ be independent finite sequences of random variables, and set
\[
W_i:=(A_i,B_i)\qquad(0\le i\le N).
\]
Then for all $h\ge 1$,
\[
\alpha_W(h)\le \alpha_A(h)+\alpha_B(h).
\]

\item Let $(X_i)_{i=0}^{N}$ be a finite sequence of random variables, and let $(\varepsilon_i)_{i=0}^{N}$ be an independent sequence of random variables,
independent of the entire sequence $(X_i)_{i=0}^{N}$. Set
\[
U_i:=(X_i,\varepsilon_i).
\]
Then for all $h\ge 1$,
\[
\alpha_U(h)=\alpha_X(h).
\]
\end{enumerate}
\end{lemma}

\begin{proof}
For part (i), let $(\widetilde V_i)_{i\in\mathbb Z}$ and $(\widetilde Y_i)_{i\in\mathbb Z}$ be the constant bi-infinite extensions of
$(V_i)_{i=0}^{N}$ and $(Y_i)_{i=0}^{N}$. Since $Y_i=g(V_i)$ for all $i$, after constant extension we have
\[
\sigma(\widetilde Y_i:\ i\le k)\subseteq\sigma(\widetilde V_i:\ i\le k),
\qquad
\sigma(\widetilde Y_i:\ i\ge k+h)\subseteq\sigma(\widetilde V_i:\ i\ge k+h)
\]
for every $k$ and $h\ge 1$. By the definition of $\alpha$, this implies
\[
\alpha_{\widetilde Y}(h)\le \alpha_{\widetilde V}(h)
\]
for all $h\ge 1$, that is, $\alpha_Y(h)\le \alpha_V(h)$.

For part (ii), let $(\widetilde A_i)_{i\in\mathbb Z}$, $(\widetilde B_i)_{i\in\mathbb Z}$, and $(\widetilde W_i)_{i\in\mathbb Z}$ be the constant
bi-infinite extensions of $(A_i)_{i=0}^{N}$, $(B_i)_{i=0}^{N}$, and $(W_i)_{i=0}^{N}$, respectively. Fix $k\in\mathbb Z$ and $h\ge 1$, and set
\[
\mathcal{M}_k^A:=\sigma(\widetilde A_i:\ i\le k),
\qquad
\mathcal{G}_{k+h}^A:=\sigma(\widetilde A_i:\ i\ge k+h),
\]
\[
\mathcal{M}_k^B:=\sigma(\widetilde B_i:\ i\le k),
\qquad
\mathcal{G}_{k+h}^B:=\sigma(\widetilde B_i:\ i\ge k+h).
\]
Since $\widetilde W_i=(\widetilde A_i,\widetilde B_i)$, we have
\[
\sigma(\widetilde W_i:\ i\le k)=\mathcal{M}_k^A\vee \mathcal{M}_k^B,
\qquad
\sigma(\widetilde W_i:\ i\ge k+h)=\mathcal{G}_{k+h}^A\vee \mathcal{G}_{k+h}^B,
\]
where $\mathcal{H}_1\vee\mathcal{H}_2$ denotes the smallest $\sigma$-algebra containing both $\mathcal{H}_1$ and $\mathcal{H}_2$.
Because the entire sequence $(A_i)_{i=0}^{N}$ is independent of the entire sequence $(B_i)_{i=0}^{N}$, the bi-infinite extensions
$(\widetilde A_i)_{i\in\mathbb Z}$ and $(\widetilde B_i)_{i\in\mathbb Z}$ are independent, and therefore the $\sigma$-fields
\[
\mathcal{M}_k^A\vee \mathcal{G}_{k+h}^A
\qquad\text{and}\qquad
\mathcal{M}_k^B\vee \mathcal{G}_{k+h}^B
\]
are independent. Hence Bradley~\cite[Theorem~5.1(a)]{Bradley05} gives
\[
\alpha\!\left(\mathcal{M}_k^A\vee \mathcal{M}_k^B,\ \mathcal{G}_{k+h}^A\vee \mathcal{G}_{k+h}^B\right)
\le
\alpha\!\left(\mathcal{M}_k^A,\mathcal{G}_{k+h}^A\right)
+
\alpha\!\left(\mathcal{M}_k^B,\mathcal{G}_{k+h}^B\right).
\]
Taking the supremum over $k$ yields
\[
\alpha_W(h)\le \alpha_A(h)+\alpha_B(h)
\]
for all $h\ge 1$.

For part (iii), part (ii) applied with $A_i:=X_i$ and $B_i:=\varepsilon_i$ gives
\[
\alpha_U(h)\le \alpha_X(h)+\alpha_\varepsilon(h).
\]
Since $(\varepsilon_i)_{i=0}^{N}$ is an independent sequence, we have $\alpha_\varepsilon(h)=0$ for all $h\ge 1$, and therefore
\[
\alpha_U(h)\le \alpha_X(h).
\]
On the other hand, $X_i$ is a coordinatewise function of $U_i$, namely $X_i=\pi_1(U_i)$ where $\pi_1$ is projection onto the first coordinate.
Therefore part (i), applied with $V_i:=U_i$ and $Y_i:=X_i$, gives
\[
\alpha_X(h)\le \alpha_U(h).
\]
Combining the two inequalities yields $\alpha_U(h)=\alpha_X(h)$ for all $h\ge 1$.
\end{proof}

\subsubsection{A 2-state transfer matrix and uniform contraction}
By Lemma~\ref{lem:block-boundary-encoding}, the map $w\mapsto \mathsf{Tag}(w)$ is a bijection from
$\mathcal{B}_n$ onto $\mathcal{T}_{n-1}$. Thus, if $w\sim\Unif{\mathcal{B}_n}$, then $\mathsf{Tag}(w)$ is uniform on $\mathcal{T}_{n-1}$.

For $0 \le t \le m$, write $T|_{[t]} := (T_1, \dots, T_t)$ for the prefix of $T$ of length~$t$
with $T|_{[0]} := \varnothing$.

We record the number of ways a fixed prefix can be extended to the right.

\begin{definition}\label{def:extension-counts}
Fix $t\ge 1$, $m\ge 0$, and $P\in\mathcal{T}_t$. Define the extension set
\[
\Ext_m(P):=\{\,T\in\mathcal{T}_{t+m}: T|_{[t]}=P\,\}.
\]
By Definition~\ref{def:block-tag-word}, given a prefix $P = (P_1, \dots, P_t)$, we want to count how many valid ways to append $m$ more symbols to get a block-tag word of length $t+m$. we set
\[
U_m(0):=|\Ext_m(P)| \quad(P_t=0),
\qquad
U_m(1):=|\Ext_m(P)| \quad(P_t\in\{S,C\}).
\]
\end{definition}

\begin{lemma}\label{lem:transfer-matrix}
For every $m\ge 0$,
\[
\begin{pmatrix} U_{m+1}(0) \\ U_{m+1}(1) \end{pmatrix}
=
\begin{pmatrix} 1 & 1 \\ 1 & 2 \end{pmatrix}
\begin{pmatrix} U_m(0) \\ U_m(1) \end{pmatrix},
\qquad
U_0(0)=U_0(1)=1.
\]
Consequently, $U_m(0)=F_{2m+1}$ and $U_m(1)=F_{2m+2}$ for all $m\ge 0$.
\end{lemma}

\begin{proof}
Fix $t\ge 1$, $m\ge 0$, and $P\in\mathcal{T}_t$.
A word $T\in\Ext_m(P)$ is determined uniquely by its suffix $(T_{t+1},\dots,T_{t+m})$, and the condition that
$T=(P_1,\dots,P_t,T_{t+1},\dots,T_{t+m})$ lie in $\mathcal{T}_{t+m}$. Since the defining condition for block-tag words is local, involving
only adjacent positions, the allowable choices for the suffix depend on the prefix $P$ only through its final symbol $P_t$, and in fact only
through the two cases $P_t=0$ and $P_t\in\{S,C\}$. This proves the first assertion, so $U_m(0)$ and $U_m(1)$ are well defined.

The case $m=0$ is immediate: for any prefix $P$, there is exactly one extension of length $0$, namely $P$ itself. Hence
$U_0(0)=U_0(1)=1$.

Now fix $m\ge 0$.

If $P_t=0$, then an extension in $\Ext_{m+1}(P)$ is obtained by choosing the next symbol $T_{t+1}$. Since the symbol immediately to the left
is $0$, the next symbol may be either $0$ or $S$. If $T_{t+1}=0$, the remaining $m$ positions can be filled in $U_m(0)$ ways. If
$T_{t+1}=S$, the remaining $m$ positions can be filled in $U_m(1)$ ways. Therefore
\[
U_{m+1}(0)=U_m(0)+U_m(1).
\]

If instead $P_t\in\{S,C\}$, then the next symbol may be $0$, $S$, or $C$. Choosing $0$ leaves $U_m(0)$ possibilities for the remaining $m$
positions, while choosing either $S$ or $C$ leaves $U_m(1)$ possibilities. Hence
\[
U_{m+1}(1)=U_m(0)+2U_m(1).
\]

These two recurrences are exactly the displayed matrix formula.

Finally, the explicit formulas follow from the initial values $U_0(0)=1=F_1$ and $U_0(1)=1=F_2$, together with the identities
$F_{2m+3}=F_{2m+1}+F_{2m+2}$ and $F_{2m+4}=F_{2m+1}+2F_{2m+2}$.
\end{proof}

Let $T=(T_1,\dots,T_{n-1})\sim \Unif{\mathcal{T}_{n-1}}$. For $1\le t\le n-1$, define
\[
\mathsf{Pres}_t(T):=\mathbf{1}\{T_t\in\{S,C\}\},
\]
and set $\mathsf{Pres}_0(T):=0$.

\begin{lemma}\label{lem:state-markov}
The process $(\mathsf{Pres}_t(T))_{t=0}^{n-1}$ is a time-inhomogeneous Markov chain on $\{0,1\}$:
for each $0\le t\le n-2$, each $b\in\{0,1\}$, and every history $y_0,\dots,y_t\in\{0,1\}$ having positive probability,
\[
\Prob{\mathsf{Pres}_{t+1}(T)=b \mid \mathsf{Pres}_0(T)=y_0,\dots,\mathsf{Pres}_t(T)=y_t}
= K_t(y_t,b),
\]
where the transition kernel $K_t$ on $\{0,1\}$ has entries
\[
K_t(0,0)=\frac{U_{n-t-2}(0)}{U_{n-t-2}(0)+U_{n-t-2}(1)},
\qquad
K_t(0,1)=\frac{U_{n-t-2}(1)}{U_{n-t-2}(0)+U_{n-t-2}(1)},
\]
\[
K_t(1,0)=\frac{U_{n-t-2}(0)}{U_{n-t-2}(0)+2U_{n-t-2}(1)},
\qquad
K_t(1,1)=\frac{2U_{n-t-2}(1)}{U_{n-t-2}(0)+2U_{n-t-2}(1)}.
\]

Moreover, conditional on any realization of the full presence path $(\mathsf{Pres}_t(T))_{t=0}^{n-1}$ having positive probability, the symbols $T_t$ are determined at every position except
those $t$ with $\mathsf{Pres}_{t-1}(T)=\mathsf{Pres}_t(T)=1$. At each such position, the symbol $T_t$ is chosen from $\{S,C\}$, and these
choices are independent and each uniform on $\{S,C\}$.
\end{lemma}

\begin{proof}
Fix $0\le t\le n-2$. Assume first that $1\le t\le n-2$, and fix a prefix $P=(P_1,\dots,P_t)\in\mathcal{T}_t$. By Definition~\ref{def:block-tag-word}, the number of completions of positions $t+2,\dots,n-1$
depends only on whether $T_{t+1}=0$ or $T_{t+1}\in\{S,C\}$.

Suppose $P_t=0$. Then the next symbol is either $0$ or $S$.
If $T_{t+1}=0$, then the remaining $n-t-2$ positions can be filled in $U_{n-t-2}(0)$ ways.
If $T_{t+1}=S$, then the remaining $n-t-2$ positions can be filled in $U_{n-t-2}(1)$ ways.
Since $T$ is uniform on $\mathcal{T}_{n-1}$, conditioned on $T|_{[t]}=P$ all such completions are equally likely. Therefore,
\[
\Prob{\mathsf{Pres}_{t+1}(T)=x \mid T|_{[t]}=P}
=
\begin{cases}
\dfrac{U_{n-t-2}(0)}{U_{n-t-2}(0)+U_{n-t-2}(1)}, & x=0,\\[3mm]
\dfrac{U_{n-t-2}(1)}{U_{n-t-2}(0)+U_{n-t-2}(1)}, & x=1.
\end{cases}
\]
Suppose that $P_t\in\{S,C\}$. Then the next symbol may be $0$, $S$, or $C$.
If $T_{t+1}=0$, then the remaining $n-t-2$ positions can be filled in $U_{n-t-2}(0)$ ways.
If $T_{t+1}=S$ or $T_{t+1}=C$, then the remaining $n-t-2$ positions can be filled in $U_{n-t-2}(1)$ ways.
Therefore,
\[
\Prob{\mathsf{Pres}_{t+1}(T)=x \mid T|_{[t]}=P}
=
\begin{cases}
\dfrac{U_{n-t-2}(0)}{U_{n-t-2}(0)+2U_{n-t-2}(1)}, & x=0,\\[3mm]
\dfrac{2U_{n-t-2}(1)}{U_{n-t-2}(0)+2U_{n-t-2}(1)}, & x=1.
\end{cases}
\]
In each case, for $1\le t\le n-2$, the conditional law of $\mathsf{Pres}_{t+1}(T)$ depends on the prefix $P$ only through whether
$P_t=0$ or $P_t\in\{S,C\}$, that is, only through the value of $\mathsf{Pres}_t(T)$.

The case $t=0$ is identical to the case $P_t=0$ above, with the role of the final prefix symbol played by the convention
$\mathsf{Pres}_0(T)=0$. Therefore $(\mathsf{Pres}_t(T))_{t=0}^{n-1}$ is a time-inhomogeneous Markov chain, and the displayed formulas for
the transition kernel $K_t$ follow.

For the final statement, condition on a realization of the full presence path $(\mathsf{Pres}_t(T))_{t=0}^{n-1}$ having positive probability. If $\mathsf{Pres}_t(T)=0$, then necessarily
$T_t=0$. If $\mathsf{Pres}_{t-1}(T)=0$ and $\mathsf{Pres}_t(T)=1$, then necessarily $T_t=S$, since a symbol $C$ cannot follow $0$.
If $\mathsf{Pres}_{t-1}(T)=1$ and $\mathsf{Pres}_t(T)=1$, then $T_t$ may be either $S$ or $C$. Thus the only remaining freedom is at those
indices $t$ for which $\mathsf{Pres}_{t-1}(T)=\mathsf{Pres}_t(T)=1$, and at each such index there are exactly two choices.

Moreover, changing $T_t$ between $S$ and $C$ at such an index does not affect which symbols may appear in later positions, so the compatible
block-tag words are obtained by making these choices independently over all indices with consecutive $1$'s in the presence path. Since $T$ is
uniform on $\mathcal{T}_{n-1}$, all compatible block-tag words are equally likely. Therefore, conditional on the full presence path, the
symbols at those indices are independent and each uniform on $\{S,C\}$.
\end{proof}

\begin{remark}\label{rem:boolean-sampling}
The block-tag encoding also yields an exact uniform sampler for $\mathcal{B}_n$ running in $O(n)$ time. Indeed, by
Lemma~\ref{lem:block-boundary-encoding}, it suffices to sample a block-tag word $T\in\mathcal{T}_{n-1}$ uniformly. This can be done
sequentially from left to right using the completion counts from Lemma~\ref{lem:transfer-matrix}. If $x_t\in\{0,1\}$ records whether the
current tag is present, and there remain $m$ positions to fill after choosing the next symbol, then:
\begin{itemize}
\item[(1)] if $x_t=0$, the next tag is chosen from $\{0,S\}$ with probabilities proportional to $U_m(0)$ and $U_m(1)$;
\item[(2)] if $x_t=1$, the next tag is chosen from $\{0,S,C\}$ with probabilities proportional to $U_m(0),U_m(1)$, and $U_m(1)$.
\end{itemize}
Since $U_m(0)=F_{2m+1}$ and $U_m(1)=F_{2m+2}$, these probabilities are explicit. Choosing each next symbol with probability proportional to
the number of admissible completions produces a uniform random tag word in $\mathcal{T}_{n-1}$, and hence, via the bijection
$T\mapsto w_T$, a uniform random Boolean permutation in $\mathcal{B}_n$.
\end{remark}

\begin{lemma}\label{lem:dobrushin}
Let $K_t$ be the transition kernel from Lemma~\ref{lem:state-markov}. Then
\[
\delta:=\sup_{0\le t\le n-2}\left\|K_t(0,\cdot)-K_t(1,\cdot)\right\|_{\mathrm{TV}}\le \frac{1}{6}.
\]
Consequently, for every $r\ge 0$ and $h\ge 1$ with $r+h\le n-1$, and every two probability measures $\pi,\pi'$ on $\{0,1\}$,
\[
\left\|\pi K_rK_{r+1}\cdots K_{r+h-1}-\pi'K_rK_{r+1}\cdots K_{r+h-1}\right\|_{\mathrm{TV}}
\le
\delta^h\,\|\pi-\pi'\|_{\mathrm{TV}}.
\]
\end{lemma}

\begin{proof}
Since the state space is $\{0,1\}$, the total variation distance between two probability measures is the absolute
difference of their masses at $0$. Thus, for $0\le t\le n-2$, $\left\|K_t(0,\cdot)-K_t(1,\cdot)\right\|_{\mathrm{TV}}=\left|K_t(0,0)-K_t(1,0)\right|$.
By Lemma~\ref{lem:state-markov},
\begin{align*}
\left|K_t(0,0)-K_t(1,0)\right|
&=
\left|
\frac{U_{n-t-2}(0)}{U_{n-t-2}(0)+U_{n-t-2}(1)}
-
\frac{U_{n-t-2}(0)}{U_{n-t-2}(0)+2U_{n-t-2}(1)}
\right|
\\
&=
\frac{U_{n-t-2}(0)\,U_{n-t-2}(1)}{\left(U_{n-t-2}(0)+U_{n-t-2}(1)\right)\left(U_{n-t-2}(0)+2U_{n-t-2}(1)\right)}.
\end{align*}
Writing $\rho:=U_{n-t-2}(0)/U_{n-t-2}(1)$ gives $\left|K_t(0,0)-K_t(1,0)\right|=\frac{\rho}{(\rho+1)(\rho+2)}$. By Lemma~\ref{lem:transfer-matrix}, $0<U_m(0)\le U_m(1)$ for every $m\ge0$, with $U_0(0)=U_0(1)=1$. Thus $0<\rho\le1$, and
\[
  \left\|K_t(0,\cdot)-K_t(1,\cdot)\right\|_{\mathrm{TV}} = \frac{\rho}{(\rho+1)(\rho+2)}\le \frac{1}{6}.
\]
As $t$ was arbitrary, taking the supremum over all $t$ yields $\delta\le 1/6$.

For the contraction estimate, view a probability measure $\mu$ on $\{0,1\}$ as a row vector $\mathbf{p}=(p_0,p_1)$ with
$p_a:=\mu(a)$ for $a\in\{0,1\}$, and write $\mathbf{p}K_t$ for the distribution obtained after one step using the kernel $K_t$, that is,
\[
(\mathbf{p}K_t)_b=\sum_{a\in\{0,1\}}p_a\,K_t(a,b).
\]
Let $e_1:=(1,0)$ and $e_2:=(0,1)$ be the standard basis vectors, corresponding to point masses at $0$ and $1$.
Now fix two probability measures $\pi,\pi'$ on $\{0,1\}$ and set $\lambda:=\pi(0)-\pi'(0)$. Then
$\pi - \pi' = (\pi(0) - \pi'(0), \pi(1) - \pi'(1)) = (\lambda, \pi(1) - \pi'(1))$. Since $\pi(0) + \pi(1) = 1$ and $\pi'(0) + \pi'(1) = 1$, we obtain $\pi(1) - \pi'(1) = -(\pi(0) - \pi'(0)) = -\lambda$.
Thus, $\pi-\pi'=(\lambda,-\lambda)=\lambda(e_1-e_2)$, so by linearity, $(\pi-\pi')K_t=\lambda\left(K_t(0,\cdot)-K_t(1,\cdot)\right)$. Since $\|\pi-\pi'\|_{\mathrm{TV}}=|\pi(0)-\pi'(0)|=|\lambda|$ on $\{0,1\}$, taking total variation norms gives
\[
\left\|\pi K_t-\pi'K_t\right\|_{\mathrm{TV}} = \left\|(\pi-\pi')K_t\right\|_{\mathrm{TV}} \le |\lambda|\,\delta = \delta\,\|\pi-\pi'\|_{\mathrm{TV}}.
\]
Applying this a second time we get
\[
\left\|\pi K_rK_{r+1}-\pi'K_rK_{r+1}\right\|_{\mathrm{TV}}
\le
\delta\,\left\|\pi K_r-\pi'K_r\right\|_{\mathrm{TV}}
\le
\delta^2\,\|\pi-\pi'\|_{\mathrm{TV}},
\]
and iterating in the same way over $h$ kernels yields
\[
\left\|\pi K_rK_{r+1}\cdots K_{r+h-1}-\pi'K_rK_{r+1}\cdots K_{r+h-1}\right\|_{\mathrm{TV}}
\le
\delta^h\,\|\pi-\pi'\|_{\mathrm{TV}}. \qedhere
\]
\end{proof}

\begin{proposition}\label{prop:boolean-strong-mixing}
Fix $n\ge 2$ and let $w\sim\Unif{\mathcal{B}_n}$. Then the finite sequence $\chi(w) := (\chi_i(w))_{i=1}^n$ satisfies
\[
\alpha_{\chi(w)}(h)\le 6^{-h/3},
\]
for every $h\ge 1$. Consequently, if $u,v\stackrel{\mathrm{iid}}{\sim}\Unif{\mathcal{B}_n}$ and $Z_i(u,v):=\Delta_i(u,v)-\E{\Delta_i(u,v)}$, then
\[
\alpha_{Z(u,v)}(h)\ \le\ 6^{-h/6}
\ =\ \exp(-2ch),
\qquad c:=\tfrac1{12}\log 6,
\]
for every $h\ge 1$.
\end{proposition}

\begin{proof}
Let $T=\mathsf{Tag}(w)\sim\Unif{\mathcal{T}_{n-1}}$. We first prove that the finite sequence $(\mathsf{Pres}_t(T))_{t=0}^{n-1}$ satisfies
\[
\alpha_{\mathsf{Pres}(T)}(h)\le 6^{-h}
\]
for every $h\ge 1$.

Fix $h\ge 1$. Since $(\mathsf{Pres}_t(T))_{t=0}^{n-1}$ is a finite sequence extended by constants (Definition~\ref{def:rosenblatt-alpha}), the past and future $\sigma$-algebras are trivial when the relevant index sets lie entirely outside $\{0,\dots,n-1\}$. Since $\mathsf{Pres}_0(T)=0$, the past $\sigma$-algebra is also trivial when $r=0$. These cases contribute zero to the mixing coefficient, so it suffices to consider integers $r$ with $1\le r\le n-1-h$. In this range, every event $\{\mathsf{Pres}_r(T)=x,\mathsf{Pres}_{r+h}(T)=y\}$ has positive probability for $x,y\in\{0,1\}$, since all entries of the transition kernels in Lemma~\ref{lem:state-markov} are positive. Set
\[
\mathcal{M}_r:=\sigma\left(\mathsf{Pres}_s(T): s\le r\right),
\qquad
\mathcal{G}_{r+h}:=\sigma\left(\mathsf{Pres}_s(T): s\ge r+h\right).
\]
Fix $A\in\mathcal{M}_r$ and $B\in\mathcal{G}_{r+h}$.
By Lemma~\ref{lem:state-markov}, the sequence $(\mathsf{Pres}_t(T))_{t=0}^{n-1}$ is a time-inhomogeneous Markov chain. Hence, by the
Markov property at time $r$, the conditional distribution of the future presence path
\[
(\mathsf{Pres}_{r+h}(T),\mathsf{Pres}_{r+h+1}(T),\dots,\mathsf{Pres}_{n-1}(T))
\]
given $\mathcal{M}_r$ depends only on the value of $\mathsf{Pres}_r(T)$. Therefore there is a function
$\phi:\{0,1\}\to[0,1]$ such that
\[
\E{\mathbf{1}_B\mid \mathcal{M}_r}=\phi\left(\mathsf{Pres}_r(T)\right).
\]
For each $x\in\{0,1\}$, the law of total probability conditioned on $\mathsf{Pres}_{r+h}(T)$ gives
\[
\phi(x)
=
\sum_{y\in\{0,1\}}
\Prob{B\mid \mathsf{Pres}_{r+h}(T)=y,\ \mathsf{Pres}_r(T)=x}\,
\Prob{\mathsf{Pres}_{r+h}(T)=y\mid \mathsf{Pres}_r(T)=x}.
\]
By the Markov property at time $r+h$, we have $\Prob{B\mid \mathsf{Pres}_{r+h}(T)=y,\ \mathsf{Pres}_r(T)=x}=\Prob{B\mid \mathsf{Pres}_{r+h}(T)=y}$, so the law of total probability gives
\[
\phi(x)
=
\sum_{y\in\{0,1\}}
\Prob{B\mid \mathsf{Pres}_{r+h}(T)=y}\,
\Prob{\mathsf{Pres}_{r+h}(T)=y\mid \mathsf{Pres}_r(T)=x}.
\]
Since $\Prob{\mathsf{Pres}_{r+h}(T)=y\mid \mathsf{Pres}_r(T)=x}$ is the $h$-step transition probability, it equals the $(x,y)$-entry of $K_r\cdots K_{r+h-1}$. Writing $\pi:=(1,0)$ and $\pi':=(0,1)$ for the point masses at $0$ and $1$, we have $\Prob{\mathsf{Pres}_{r+h}(T)=y\mid \mathsf{Pres}_r(T)=x}=(\pi_x K_r\cdots K_{r+h-1})(y)$ with $\pi_0=\pi$, $\pi_1=\pi'$. Taking the difference $\phi(0)-\phi(1)$ yields
\[
\phi(0)-\phi(1)
=
\sum_{y\in\{0,1\}}
\Prob{B\mid \mathsf{Pres}_{r+h}(T)=y}
\left(
(\pi K_r\cdots K_{r+h-1})(y)
-
(\pi' K_r\cdots K_{r+h-1})(y)
\right).
\]
Since $\sum_y\bigl((\pi K_r\cdots K_{r+h-1})(y)-(\pi' K_r\cdots K_{r+h-1})(y)\bigr)=0$ and $0\le\Prob{B\mid\mathsf{Pres}_{r+h}(T)=y}\le 1$, the sum $\phi(0)-\phi(1)$ collapses to
\[
\left(
\Prob{B\mid \mathsf{Pres}_{r+h}(T)=0}
-
\Prob{B\mid \mathsf{Pres}_{r+h}(T)=1}
\right)
\left(
(\pi K_r\cdots K_{r+h-1})(0)
-
(\pi' K_r\cdots K_{r+h-1})(0)
\right).
\]
The first factor has absolute value at most \(1\). Therefore
\begin{align*}
|\phi(0)-\phi(1)|
&\le
\left|
(\pi K_r\cdots K_{r+h-1})(0)
-
(\pi' K_r\cdots K_{r+h-1})(0)
\right|
\\&=
\left\|\pi K_r\cdots K_{r+h-1}
-
\pi' K_r\cdots K_{r+h-1}\right\|_{\mathrm{TV}}.
\end{align*}
Applying Lemma~\ref{lem:dobrushin} gives
\[
|\phi(0)-\phi(1)|\le\left\|\pi K_r\cdots K_{r+h-1}
-
\pi' K_r\cdots K_{r+h-1}\right\|_{\mathrm{TV}}\le
6^{-h}\,\|\pi-\pi'\|_{\mathrm{TV}} = 6^{-h}.
\]
Since $A\in\mathcal{M}_r$, the indicator $\mathbf{1}_A$ is $\mathcal{M}_r$-measurable. Hence, by the tower property,
\[
\Prob{A\cap B}
=\E{\mathbf{1}_A\mathbf{1}_B}
=\E{\E{\mathbf{1}_A\mathbf{1}_B\mid\mathcal{M}_r}}
=\E{\mathbf{1}_A\E{\mathbf{1}_B\mid\mathcal{M}_r}}
=\E{\mathbf{1}_A\,\phi(\mathsf{Pres}_r(T))}.
\]
Subtracting $\Prob{A}\Prob{B}=\E{\mathbf{1}_A}\Prob{B}=\E{\mathbf{1}_A\Prob{B}}$ gives
\[
\Prob{A\cap B}-\Prob{A}\Prob{B}
=
\E{\,\mathbf{1}_A\left(\phi(\mathsf{Pres}_r(T))-\Prob{B}\right)\,}.
\]
Since $\Prob{B}=\phi(0)\Prob{\mathsf{Pres}_r(T)=0}+\phi(1)\Prob{\mathsf{Pres}_r(T)=1}$ lies between $\phi(0)$ and $\phi(1)$, we have $|\phi(x)-\Prob{B}|\le|\phi(0)-\phi(1)|$ for $x\in\{0,1\}$, and hence
\[
\left|\Prob{A\cap B}-\Prob{A}\Prob{B}\right|
\le
\Prob{A}\,|\phi(0)-\phi(1)|
\le
|\phi(0)-\phi(1)|
\le
6^{-h}.
\]
Taking suprema over $r$, $A$, and $B$ gives $\alpha_{\mathsf{Pres}(T)}(h)\le 6^{-h}$ for every $h\ge 1$.

Let $\varepsilon_0,\varepsilon_1,\dots,\varepsilon_{n-1}$ be independent $\{S,C\}$-valued random variables, independent of the finite
sequence $(\mathsf{Pres}_t(T))_{t=0}^{n-1}$, with $\varepsilon_0:=S$ deterministic and $\varepsilon_t$ fair for $1\le t\le n-1$.
These variables encode the residual randomness in the tag process after conditioning on the presence path, namely the $S/C$ choices at positions where consecutive presence indicators are both $1$.
By the final statement of Lemma~\ref{lem:state-markov}, after enlarging the probability space if necessary we may and do realize
$T$, $(\mathsf{Pres}_t(T))_{t=0}^{n-1}$, and $(\varepsilon_t)_{t=0}^{n-1}$ on a common finite probability space so that, for each
$1\le t\le n-1$,
\[
T_t=
\begin{cases}
0, & \mathsf{Pres}_t(T)=0,\\
S, & \mathsf{Pres}_{t-1}(T)=0\ \text{and}\ \mathsf{Pres}_t(T)=1,\\
\varepsilon_t, & \mathsf{Pres}_{t-1}(T)=1\ \text{and}\ \mathsf{Pres}_t(T)=1.
\end{cases}
\]
Applying Lemma~\ref{lem:alpha-factor}(iii) to the finite sequences $(\mathsf{Pres}_t(T))_{t=0}^{n-1}$ and $(\varepsilon_t)_{t=0}^{n-1}$ gives
\[
\alpha_{(\mathsf{Pres}(T),\varepsilon)}(h)=\alpha_{\mathsf{Pres}(T)}(h)
\]
for every $h\ge 1$, where $(\mathsf{Pres}(T),\varepsilon)$ denotes the finite sequence
\[
\left((\mathsf{Pres}_t(T),\varepsilon_t)\right)_{t=0}^{n-1},
\]
interpreted via the constant-extension convention.

The displayed formula shows that passing from $\mathsf{Pres}$ to $T$ loses one unit of lag, since $T_t$ is a function of $\mathsf{Pres}_{t-1}(T)$, $\mathsf{Pres}_t(T)$, and $\varepsilon_t$. Hence, for every $r$,
\[
\sigma(T_1,\dots,T_r)\subseteq \sigma\left((\mathsf{Pres}_s(T),\varepsilon_s): s\le r\right),
\quad\;
\sigma(T_{r+h},\dots,T_{n-1})\subseteq \sigma\left((\mathsf{Pres}_s(T),\varepsilon_s): s\ge r+h-1\right).
\]
It follows that
\[
\alpha_T(h)\le \alpha_{(\mathsf{Pres}(T),\varepsilon)}(h-1)\le 6^{-(h-1)}
\]
for every $h\ge 2$, where $\alpha_T(h)$ denotes the $\alpha$-mixing coefficient of the finite sequence $(T_1,\dots,T_{n-1})$.

Now apply Lemma~\ref{lem:chi-as-local}. Passing from $T$ to $\chi(w)$ loses one more unit of lag, since for $j\ge 2$ the variable $\chi_j(w)$ is a function of $(T_{j-1},T_j)$, while $\chi_1(w)=0$ is
constant. Hence
\[
\alpha_{\chi(w)}(h)\le \alpha_T(h-1)\le \alpha_{\mathsf{Pres}(T)}(h-2)\le 6^{-(h-2)}
\]
for every $h\ge 3$. For $h=1,2$ we use the universal bound $\alpha(h)\le \tfrac14$, valid for every process and every $h\ge 1$. Combining these bounds, and noting that
$6^{-(h-2)}\le 6^{-h/3}$ for $h\ge 3$, yields
\[
\alpha_{\chi(w)}(h)\le 6^{-h/3}
\]
for every $h\ge 1$.

Finally, let $u,v\stackrel{\mathrm{iid}}{\sim}\Unif{\mathcal{B}_n}$. Since the sequences $\chi(u)$ and $\chi(v)$ are finite and independent,
Lemma~\ref{lem:alpha-factor}(ii) gives
\[
\alpha_{(\chi(u),\chi(v))}(h)
\le
\alpha_{\chi(u)}(h)+\alpha_{\chi(v)}(h)
=
2\,\alpha_{\chi(w)}(h),
\]
for every $h\ge 1$, where $(\chi(u),\chi(v))$ denotes the finite sequence
\[
\left((\chi_i(u),\chi_i(v))\right)_{i=1}^{n},
\]
interpreted via the constant-extension convention.
Since $Z_i(u,v)$ is a coordinatewise function of the pair $(\chi_i(u),\chi_i(v))$, Lemma~\ref{lem:alpha-factor}(i) gives
\[
\alpha_{Z(u,v)}(h)\le 2\,\alpha_{\chi(w)}(h)\le 2\cdot 6^{-h/3}
\]
for every $h\ge 1$. By the universal bound,
we also have $\alpha_{Z(u,v)}(h)\le \tfrac14$. Therefore
\[
\alpha_{Z(u,v)}(h)\le \min\!\left(\tfrac14,\ 2\cdot 6^{-h/3}\right)\le 6^{-h/6}
\]
for every $h\ge 1$, because $\tfrac14\le 6^{-h/6}$ for $h=1,2,3,4$ and $2\cdot 6^{-h/3}\le 6^{-h/6}$ for $h\ge 5$. This proves the proposition.
\end{proof}

\subsubsection{Block concentration via Merlev\`ede--Peligrad--Rio}\label{subsec:boolean-mpr}

We will apply the Bernstein inequality of Merlev\`ede--Peligrad--Rio \cite[Theorem 1]{MPR09} to centered block sums of the increment process
$\Delta_i(u,v)=1-\chi_i(u)-\chi_i(v)$.

\begin{lemma}\label{lem:mpr-block}
Let $M,c>0$, and let $(X_t)_{t\ge1}$ be a sequence of centered real-valued random variables such that $\|X_t\|_\infty\le M$ for all $t$ and whose Rosenblatt strong-mixing coefficients satisfy $\alpha(h)\le e^{-2ch}$ for every integer $h\ge1$. Then there exists $C_3>0$, depending only on $c$, such that for all integers $k\ge4$ and all $x\ge0$,
\begin{equation}\label{eq:mpr-block-tail}
\Prob*{\sum_{t=1}^{k}X_t\le -x}
\ \le\
\exp\!\left(
-\frac{C_3\,x^2}{kM^2 + Mx(\log k)(\log\log k)}
\right).
\end{equation}
\end{lemma}

\begin{proof}
Set $S_k:=\sum_{t=1}^{k}X_t$. Applying the Laplace-transform estimate in \cite[Theorem~1]{MPR09} to $(-X_t)$ gives constants $C_1,C_2>0$, depending only on $c$, such that, for every integer $k\ge4$ and every $t$ satisfying $0<t<1/(C_1M(\log k)(\log\log k))$,
\[
\log\E{e^{-tS_k}}\le\frac{C_2t^2kM^2}{1-C_1tM(\log k)(\log\log k)}.
\]
For $x>0$, choose $t=x/(2C_2kM^2+C_1Mx(\log k)(\log\log k))$. This lies in the permitted range and gives $\log\E{e^{-tS_k}}\le tx/2$. Chernoff's inequality yields $\Prob{S_k\le -x}\le e^{-tx/2}$, so $C_3=1/\max\{4C_2,2C_1\}$ gives \eqref{eq:mpr-block-tail}. The case $x=0$ is immediate.
\end{proof}

\begin{lemma}\label{lem:bool-uniform-drift}
Let $u,v\stackrel{\mathrm{iid}}{\sim}\Unif{\mathcal{B}_n}$. There exists $\eta>0$ such that for all $n\ge 3$ and all $1\le i\le n-1$,
\[
\E{\Delta_i(u,v)}\ge \eta.
\]
\end{lemma}

\begin{proof}
Let $w\sim\Unif{\mathcal{B}_n}$. Since $\Delta_i(u,v)=1-\chi_i(u)-\chi_i(v)$ and $u,v\iidsim\Unif{\mathcal{B}_n}$, we have
\[
\E{\Delta_i(u,v)}=1-2\,\E{\chi_i(w)}=2\Prob{\rec_i(w)=1}-1.
\]
Using \eqref{eq:bool-cardinality} and Theorem~\ref{thm:Nnj-recurrence}, we obtain $\Prob{\rec_i(w)=1}=\frac{F_{2n-2}-F_{2i-4}F_{2(n-i)-1}}{F_{2n-1}}$,
so
\begin{equation}\label{eq:bool-uniform-drift-expr}
\E{\Delta_i(u,v)}
=
\frac{2F_{2n-2}-F_{2n-1}-2F_{2i-4}F_{2(n-i)-1}}{F_{2n-1}}
=
\frac{F_{2n-4}-2F_{2i-4}F_{2(n-i)-1}}{F_{2n-1}},
\end{equation}
using $2F_{2n-2}-F_{2n-1}=F_{2n-4}$.
By Lemma~\ref{lem:fib-identity} with $(m,k)=\left((2i-4)-1,\ (2(n-i)-1)-1\right)$, we get $F_{2n-6}=F_{2i-4}F_{2(n-i)-1}+F_{2i-5}F_{2(n-i)-2}$.
Since $2i-5$ is odd, the extension rule $F_{-r}=(-1)^{r+1}F_r$ gives $F_{2i-5}\ge 0$, and also $F_{2(n-i)-2}\ge 0$.
Thus, $F_{2n-6}\ge F_{2i-4}F_{2(n-i)-1}$, and
\[
F_{2n-4}-2F_{2i-4}F_{2(n-i)-1}\ge F_{2n-4}-2F_{2n-6}=F_{2n-7}.
\]
Applying the above inequality to \eqref{eq:bool-uniform-drift-expr} gives
\[
\E{\Delta_i(u,v)}\ge \frac{F_{2n-7}}{F_{2n-1}},
\]
for all $n\ge 3$ and $1\le i\le n-1$.
Since $F_{t+2}=F_{t+1}+F_t\le 3F_t$ for all $t\ge 1$, iterating gives $F_{t+6}\le 27F_t$. Setting $t=2n-7\ge 1$ yields $F_{2n-7}/F_{2n-1}\ge 1/27$ for $n\ge 4$. The case $n=3$ gives $F_{-1}/F_5=1/5\ge 1/27$. Hence $\eta:=1/27$ works.
\end{proof}

\begin{lemma}\label{lem:bool-terminal}
Let $n\ge 3$ and let $u,v\stackrel{\mathrm{iid}}{\sim}\Unif{\mathcal{B}_n}$. Then
\[
\max_{1\le i\le n+1}\E{Y_i(u,v)} \;=\; \E{Y_n(u,v)} \;=\; n+1-2\cdot\frac{F_{2n-3}}{F_{2n-1}}.
\]
In particular, $\max_{1\le i\le n+1}\E{Y_i(u,v)}=n+O(1)$.
\end{lemma}

\begin{proof}
By Lemma~\ref{lem:bool-uniform-drift}, $\E{\Delta_i(u,v)}=\E{Y_{i+1}(u,v)}-\E{Y_i(u,v)}\ge \eta>0$ for all $1\le i\le n-1$. Hence, the sequence $\E{Y_i(u,v)}$ is strictly increasing for $1\le i\le n$.

It remains to check that $\E{Y_{n+1}(u,v)}<\E{Y_n(u,v)}$, i.e.\ $\E{\Delta_n(u,v)}<0$ for $n\ge 3$.
Let $w\sim\Unif{\mathcal{B}_n}$. Since $u,v$ are i.i.d. and $\chi_n(w)=\mathbf{1}\{\rec_n(w)=0\}$,
\[
\E{\Delta_n(u,v)}=1-2\,\E{\chi_n(w)}=2\Prob{\rec_n(w)=1}-1.
\]
But $\rec_n(w)=1$ if and only if $w(n)=n$, and deleting the fixed last position gives a bijection $\{w\in\mathcal{B}_n:\ w(n)=n\}\ \cong\ \mathcal{B}_{n-1}$,
so $\Prob{\rec_n(w)=1}=|\mathcal{B}_{n-1}|/|\mathcal{B}_n|=F_{2n-3}/F_{2n-1}$.
Therefore for $n\ge 3$,
\[
\E{\Delta_n(u,v)}=2\cdot\frac{F_{2n-3}}{F_{2n-1}}-1<0,
\]
since $F_{2n-1}=F_{2n-2}+F_{2n-3}>2F_{2n-3}$ for $n\ge 3$.
Thus $\E{Y_i(u,v)}$ increases up to $i=n$ and decreases at $i=n+1$, so $\E{Y_i(u,v)}$ attains its unique maximum at $i=n$.

Finally, $\Lambda_n(w)=\chi_n(w)$ and $\Lambda_{n+1}(w)=0$, hence
\[
\E{Y_n(u,v)}=2\E{\Lambda_n(w)}+(n-1)=2\left(1-\Prob{\rec_n(w)=1}\right)+(n-1)=n+1-2\cdot\frac{F_{2n-3}}{F_{2n-1}}. \qedhere
\]
\end{proof}

\begin{lemma}\label{lem:boolean-fs-lower}
Let $u,v\stackrel{\mathrm{iid}}{\sim}\Unif{\mathcal{B}_n}$. Then $\E{\FS(u,v)}\ge n-\frac13$.
\end{lemma}

\begin{proof}
Evaluating \eqref{eq:mainFSEq} at $i=1+\max(\FS(u),\FS(v))$ gives $\FS(u,v)\ge \max(\FS(u),\FS(v))$, so $\E{\FS(u,v)}\ge \E{\max(\FS(u),\FS(v))}$.

Let $w\sim\Unif{\mathcal{B}_n}$. For $1\le t\le n-1$, the event $\{\FS(w)\le n-t\}$ means $w$ fixes positions $n-t+1,\dots,n$, giving a bijection with $\mathcal{B}_{n-t}$. By \eqref{eq:bool-cardinality},
\[
\Prob{n-\FS(w)\ge t}=\frac{F_{2n-2t-1}}{F_{2n-1}}.
\]
By the tail-sum formula and independence,
\[
\E{\max(\FS(u),\FS(v))}
=n-\sum_{t=1}^{n-1}\Prob{n-\FS(u)\ge t,\ n-\FS(v)\ge t}
=n-\sum_{t=1}^{n-1}\left(\frac{F_{2n-2t-1}}{F_{2n-1}}\right)^2.
\]
Since $F_{k+2}\ge 2F_k$ for $k\ge 1$, iterating gives $F_{2n-1}\ge 2^t F_{2n-2t-1}$, hence $(F_{2n-2t-1}/F_{2n-1})^2\le 4^{-t}$. Therefore
\[
\E{\max(\FS(u),\FS(v))}
\ge n-\sum_{t\ge 1}4^{-t}
= n-\frac13.
\qedhere
\]
\end{proof}

\theoremBooleanExpectation*

\begin{proof}
For the upper bound, $\max(\FS(u),\FS(v))\le n$ and \eqref{eq:mainFSEq} give $\FS(u,v)\le \max_{1\le i\le n+1}Y_i(u,v)$. By Lemma~\ref{lem:bool-terminal}, $\E{Y_n(u,v)}=n+O(1)$, so it suffices to show $\E*{\max_{1\le i\le n+1}Y_i(u,v)-Y_n(u,v)}=O(1)$. Since $Y_{n+1}=n$, we have
\[
\max_{1\le i\le n+1}Y_i-Y_n
\le
\sum_{k=1}^{n-1}\max\{Y_{n-k}-Y_n,0\}+1,
\]
so it is enough to prove $\sum_{k=1}^{n-1}\E{\max\{Y_{n-k}-Y_n,0\}}=O(1)$.

Fix $k\in\{1,\dots,n-1\}$. Then $Y_n-Y_{n-k}=\sum_{t=n-k}^{n-1}\Delta_t=\sum_{t=n-k}^{n-1}\E{\Delta_t}+\sum_{t=n-k}^{n-1}Z_t$, and Lemma~\ref{lem:bool-uniform-drift} gives $\sum_{t=n-k}^{n-1}\E{\Delta_t}\ge \eta k$. Hence
\[
\{Y_{n-k}\ge Y_n\}\subseteq \left\{\sum_{t=n-k}^{n-1}Z_t\le -\eta k\right\}.
\]

By Proposition~\ref{prop:boolean-strong-mixing}, $(Z_t)$ satisfies $\alpha(h)\le \exp(-2ch)$ for all $h\ge 1$, where $c=\tfrac1{12}\log 6$, and $|Z_t|\le 2$. Reindexing the block $(Z_{n-k},\dots,Z_{n-1})$ as $X_1,\dots,X_k$ and extending by zeros preserves centering and boundedness and does not increase the mixing coefficients, by the convention following Definition~\ref{def:rosenblatt-alpha}. Applying \eqref{eq:mpr-block-tail} to this sequence with $M=2$ and $x=\eta k$ gives, for $k\ge 4$,
\[
\Prob{Y_{n-k}\ge Y_n}
\le
\exp\!\left(-\frac{C_3\eta^2\,k}{4+2\eta(\log k)(\log\log k)}\right).
\]
Since $(\log k)(\log\log k)\ge (\log 4)(\log\log 4)>0$ for $k\ge 4$, there exists $c'>0$ such that
\[
\Prob{Y_{n-k}\ge Y_n}
\le
\exp\!\left(-c'\,\frac{k}{(\log k)(\log\log k)}\right).
\]

Since $|Y_{n-k}-Y_n|\le k$, we have $\E{\max\{Y_{n-k}-Y_n,0\}}\le k\,\Prob{Y_{n-k}\ge Y_n}$. Absorbing the finitely many cases $k\le 3$ into the constant,
\[
\sum_{k=1}^{n-1}\E{\max\{Y_{n-k}-Y_n,0\}}
\le
O(1)+\sum_{k\ge 4}k\,\exp\!\left(-c'\,\frac{k}{(\log k)(\log\log k)}\right)
=
O(1),
\]
since the series converges. Therefore, $\E{\FS(u,v)}\le \E{Y_n(u,v)}+O(1)=n+O(1)$. Together with the lower bound from Lemma~\ref{lem:boolean-fs-lower}, this proves the theorem.
\end{proof}

\section{Backward Stability}\label{sec:backward-stability}

In this section, we consider the nonnegative back-stability number $\BS(u,v)$ defined in the introduction. For $n\ge1$ and $u,v\in S_n$, we have
\begin{equation}\label{eq:bs-fs}
\BS(u,v)+n=\max\{n,\FS(w_0uw_0,w_0vw_0)\},
\end{equation}
where $w_0$ is the longest element in $S_n$.

To prove \eqref{eq:bs-fs}, let $\lambda_i(w)$ count the positions $1\le j\le i$ for which $w(j)>w(j')$ for some $j'>j$, with $\lambda_0(w)=0$. By \cite[Theorem~1.3]{HardtWallach2024Schubert}, $\BS(u,v)$ is the maximum of $\lambda_i(u)+\lambda_i(v)-i$ over $i\ge0$. Since $\lambda_i(w)=\lambda_n(w)$ for $w\in S_n$ and $i\ge n$, the terms with $i>n$ are strictly smaller than the term with $i=n$, so the maximum may be restricted to $0\le i\le n$. Conjugation by $w_0$ reverses the positions in each inversion, giving $\Lambda_h(w_0ww_0)=\lambda_{n+1-h}(w)$ for $1\le h\le n+1$. Reindexing therefore gives $n+\BS(u,v)=\widetilde{\FS}_n(w_0uw_0,w_0vw_0)$, with $\widetilde{\FS}_n$ as in \eqref{eq:upper_inequality}. Extending the maximum in \eqref{eq:mainFSEq} through $h=n+1$ adds only indices at which both $\Lambda_h$ terms vanish. At such indices $Y_h=h-1$, and $Y_{n+1}=n$. This proves \eqref{eq:bs-fs}.

\subsection{Distributional Equivalence}

Conjugation by $w_0$ preserves the key permutation classes we study. Let $\varphi(w) = w_0 w w_0$. By \cite[Corollary~2.3.3(ii), Proposition~2.3.4]{BjornerBrenti2005}, the map $\varphi$ preserves length and induces a Coxeter-diagram automorphism. Since $\varphi$ sends reduced expressions to reduced expressions by relabeling letters, it preserves the Boolean property. In other words, if $b \in S_n$ is Boolean, then so is $w_0 b w_0$ (see \cite[Proposition~7.3]{Tenner2007}). Similarly, $w$ is Grassmannian if and only if $w_0 w w_0$ is Grassmannian (see \cite[(1.31)]{Macdonald1991Schubert}). Let $\mathcal{W}\in \{S_n, \Grass_n, \mathcal{B}_n\}$. For $u, v \iidsim \Unif{\mathcal{W}}$, the pair $(w_0 u w_0, w_0 v w_0)$ is distributed identically to $(u, v)$.

\begin{corollary}
Let $n\ge1$. For any of the three families $\mathcal{W} \in \{S_n,\Grass_n,\mathcal{B}_n\}$ and $u,v\iidsim\Unif{\mathcal{W}}$, \eqref{eq:bs-fs} gives
\begin{equation}\label{eq:bs-distribution}
\BS(u,v)+n\;\stackrel{d}{=}\;\widetilde{\FS}_n(u,v)=\max\{n,\FS(u,v)\}.
\end{equation}
\end{corollary}

To transfer the forward-stability asymptotics using \eqref{eq:bs-distribution}, it remains to compare $\widetilde{\FS}_n(u,v)$ and $\FS(u,v)$. These variables agree whenever $\FS(u,v)\ge n$, and their difference lies in $[0,n]$.

For uniform permutations and $n\ge2$, the event $\FS(u,v)<n$ implies that both $u$ and $v$ fix $n$, an event of probability $1/n^2$. Hence $\E{\widetilde{\FS}_n(u,v)-\FS(u,v)}\le1/n=o(1)$, and the two variables differ only on an event of probability tending to zero.

For Grassmannian permutations under $\mathbb{G}_n$, Proposition~\ref{prop:grass_fs_exact_prob} gives $\FS(u,v)=M_u+M_v-k_{\max}+\Xi$, with $\Prob{\Xi\ne0}=O(e^{-cn})$ for some $c>0$. The binomial and subset-maximum estimates used in the proof of Lemma~\ref{lem:grass_tail_bounds}, with thresholds $5n/8$ and $7n/8$, show that $k_{\max}\le5n/8$ and $M_u,M_v\ge7n/8$ except on an event of exponentially small probability. When these inequalities hold and $\Xi=0$, we have $\FS(u,v)\ge9n/8>n$. The total variation bound in Lemma~\ref{lem:grass_TV_single}, applied to the independent pair, transfers the exponentially small bound on $\Prob{\FS(u,v)<n}$ to the uniform model. Consequently, $\E{\widetilde{\FS}_n(u,v)-\FS(u,v)}=o(1)$.

These estimates preserve the expectation expansions and normalized limiting distributions from Sections~\ref{sec:uniform-expectation} and~\ref{sec:grassmannian} when $\FS(u,v)$ is replaced by $\BS(u,v)+n$. For Boolean permutations, the difference $\widetilde{\FS}_n(u,v)-\FS(u,v)$ is at most $n-\max\{\FS(u),\FS(v)\}$. The proof of Lemma~\ref{lem:boolean-fs-lower} gives $\E{\max\{\FS(u),\FS(v)\}}\ge n-\frac13$, so the expected difference is at most $\frac13$. Together with $\E{\FS(u,v)}=n+O(1)$ from Section~\ref{sec:boolean-expectation} and \eqref{eq:bs-distribution}, this yields $\E{\BS(u,v)}=O(1)$.

\begin{corollary}\label{cor:backward-asymptotics}
As $n\to\infty$, the following hold.
\begin{enumerate}[label=(\arabic*)]
\item If $u,v\iidsim\Unif{S_n}$, then
\[
\begin{aligned}
\E{\BS(u,v)} &= n-2\ln n-2\gamma+c+o(1),\\
\frac{\BS(u,v)-(n-2\ln n)}{\sqrt{2\ln n}} &\xrightarrow{d} \mathcal{N}(0,1),
\end{aligned}
\]
where $\gamma$ is the Euler--Mascheroni constant and $c\in[1,\pi^2/6]$ is the same constant as in the uniform forward-stability expectation expansion.
\item If $u,v\iidsim\Unif{\Grass_n}$, then
\[
\begin{aligned}
\E{\BS(u,v)} &= \frac{n}{2}-\frac{\sqrt{n}}{2\sqrt{\pi}}-2+o(1),\\
\frac{\BS(u,v)-n/2}{\sqrt{n}} &\xrightarrow{d} -\frac12\max(Z_1,Z_2),
\end{aligned}
\]
where $Z_1,Z_2$ are independent standard normal random variables.
\item If $u,v\iidsim\Unif{\mathcal{B}_n}$, then $\E{\BS(u,v)}=O(1)$.
\end{enumerate}
\end{corollary}

\section{Conjectures on \texorpdfstring{$\E{\FS(u,v)}$}{E(FS(u,v))}}\label{sec:conjectures}

Beyond the three primary families treated in this paper, it is natural to ask for the asymptotic behavior of
$\E{\FS(u,v)}$ for other well-known permutation classes.  Here we record several conjectures suggested by Monte Carlo
experiments. For each class we estimated $\E{\FS(u,v)}$ from random samples.  When an efficient uniform sampler was available we used it.
Otherwise we employed a Metropolis--Hastings algorithm to construct an ergodic Markov chain on the class whose stationary
distribution is uniform.

\begin{definition}[Pattern avoidance]
Let $\pi\in S_k$ and let $w\in S_n$. We say that $w$ \emph{contains} the pattern $\pi$ if there exist indices
$1\le i_1<i_2<\cdots<i_k\le n$ such that the subsequence $(w(i_1),\dots,w(i_k))$ has the same relative order as
$(\pi(1),\dots,\pi(k))$. If no such indices exist, then $w$ \emph{avoids} $\pi$.
For patterns $\pi^{(1)},\dots,\pi^{(r)}$ we write
\[
\Av_n(\pi^{(1)},\dots,\pi^{(r)})\;:=\;\{\,w\in S_n:\ w\text{ avoids each }\pi^{(t)}\,\}.
\]
\end{definition}

\subsection{Permutation classes}

\begin{definition}[CoGrassmannian permutations]
A permutation $w\in S_n$ is \emph{coGrassmannian} if $w^{-1}$ is Grassmannian.
We write $\CoGr_n$ for the set of coGrassmannian permutations in $S_n$.
\end{definition}

\begin{definition}[Fireworks permutations {\cite{PechenikSpeyerWeigandt2022}}]
Let $w\in S_n$. Write $w$ as a concatenation of its maximal consecutive decreasing runs,
\[
w \;=\; R_1\,R_2\,\cdots\,R_r,
\]
so that each $R_a$ is strictly decreasing and each boundary between $R_a$ and $R_{a+1}$ is an ascent.
We say that $w$ is \emph{fireworks} if the initial elements of these decreasing runs are increasing,
\[
R_1(1) \;<\; R_2(1) \;<\; \cdots \;<\; R_r(1).
\]
Let $\mathcal F_n$ denote the set of fireworks permutations in $S_n$.
\end{definition}

\begin{definition}[Smooth permutations]
Let $G/B$ be the complete flag variety of type~$A_{n-1}$ and let $X_w\subseteq G/B$ be the Schubert variety indexed by
$w\in S_n$. We say that $w$ is \emph{smooth} if $X_w$ is smooth.  In type~$A$, Lakshmibai and Sandhya proved that this
geometric condition is equivalent to avoidance of the patterns $3412$ and $4231$ \cite[Theorem~1]{LakshmibaiSandhya1990}; we write
\[
\Sm_n \;:=\; \Av_n(3412,4231)
\]
for the set of smooth permutations in $S_n$.
\end{definition}

\begin{definition}[Vexillary and covexillary permutations]
A permutation $w\in S_n$ is \emph{vexillary} if it avoids $2143$, and \emph{covexillary} if it avoids $3412$.
We write $\Vex_n:=\Av_n(2143)$ and $\CoVex_n:=\Av_n(3412)$.
Vexillary permutations have Schubert polynomials with determinantal and flagged-Schur descriptions \cite[(3.1), (4.9)]{Macdonald1991Schubert}.
Covexillary permutations form a tractable and widely studied class for the \emph{local} geometry of Schubert varieties:
in the covexillary case, important singularity invariants such as multiplicities, Hilbert series, and Kazhdan--Lusztig-type
polynomials admit positive combinatorial formulas via Gröbner methods \cite{LiYong2012Degenerations,LiLiYong2011Drift}.
\end{definition}

\begin{definition}[Maximal Levi-spherical permutations]
Let $G/B$ be the complete flag variety of type~$A_{n-1}$, and let $X_w\subseteq G/B$ be the Schubert variety indexed by
$w\in S_n$.  We say that $w$ is \emph{maximal Levi-spherical} if the maximal parabolic subgroup $P\le G$ that stabilizes
$X_w$ has the property that a Borel subgroup of its Levi factor acts on $X_w$ with a dense orbit.

The second author, along with Gao and Yong, gave a root-theoretic classification which yields the following
combinatorial criterion \cite{GaoHodgesYong2024LeviSpherical}.  Define
\[
I(w)\;:=\;\{\, i\in\{1,\dots,n-1\} : w^{-1}(i)>w^{-1}(i+1)\,\},
\]
let $W_{I(w)}\le S_n$ be the corresponding parabolic subgroup, and let $w_0(I(w))$ be its longest element.  Then $w$ is
maximal Levi-spherical if and only if the permutation $w_0(I(w))\,w$ is Boolean.

Let $\Lev_n$ denote the set of maximal Levi-spherical permutations in $S_n$.
\end{definition}

\subsection{Three asymptotic regimes}

The expected value of the forward stability falls into three regimes that correlate strongly with the abundance of
left-to-right maxima (records) in the underlying class.
\begin{itemize}
\item[(1)] the record-dense regime $\E{\FS(u,v)} = n+o(n)$,
\item[(2)] the intermediate record regime $\E{\FS(u,v)} = c n+o(n)$ for some $1<c<2$,
\item[(3)] the record-sparse regime $\E{\FS(u,v)} = 2n-o(n)$.
\end{itemize}

\begin{conjecture}[Record-dense regime]\label{conj:fs-rdense}
As $n\to\infty$, the following hold.
\begin{enumerate}[label=(\alph*)]
\item (Boolean, refined constant term.) If $u,v\sim \Unif{\mathcal B_n}$, then
\[
\E{\FS(u,v)}\;=\;n+3+o(1).
\]

\item (CoGrassmannian.) If $u,v\sim \Unif{\CoGr_n}$, then
\[
\E{\FS(u,v)}\;=\;n+\frac{1}{\sqrt{\pi}}\sqrt{n}+O(1).
\]

\item ($321$-avoiding.) If $u,v\sim \Unif{\Av_n(321)}$, then there exists a constant $c_{321}\in[0.92,0.94]$ such that
\[
\E{\FS(u,v)}\;=\;n+c_{321}\sqrt{n}+O(1).
\]

\end{enumerate}
\end{conjecture}

\begin{conjecture}[Intermediate record regime]\label{conj:fs-rint}
As $n\to\infty$, the following hold.
\begin{enumerate}[label=(\alph*)]
\item (Smooth.) If $u,v\sim \Unif{\Sm_n}$, then there exists a constant $c_{\mathrm{sm}}\in[1.35,1.40]$ such that
\[
\E{\FS(u,v)} \;=\; c_{\mathrm{sm}}\,n + o(n).
\]

\item (Maximal Levi-spherical.) If $u,v\sim \Unif{\Lev_n}$, then
\[
\E{\FS(u,v)} \;=\; \frac{5}{4}\,n + O(1).
\]
\end{enumerate}
\end{conjecture}

\begin{conjecture}[Record-sparse regime]\label{conj:fs-rsparse}
As $n\to\infty$, the following hold.
\begin{enumerate}[label=(\alph*)]
\item ($231$-avoiding.) If $u,v\sim \Unif{\Av_n(231)}$, then
\[
\E{\FS(u,v)}\;=\;2n-5+o(1).
\]

\item ($132$-avoiding.) If $u,v\sim \Unif{\Av_n(132)}$, then
\[
\E{\FS(u,v)}\;=\;2n-5+o(1).
\]

\item (Fireworks.) If $u,v\sim \Unif{\mathcal F_n}$, then
\[
\E{\FS(u,v)}
\;=\;
2n-\frac{2n}{\log n-\log\log n}
+o\!\left(\frac{n}{\log n}\right).
\]

\item (Vexillary.) If $u,v\sim \Unif{\Vex_n}$, then
\[
\E{\FS(u,v)} \;=\; 2n - O(\sqrt{n}).
\]

\item (Covexillary.) If $u,v\sim \Unif{\CoVex_n}$, then there exists a constant $c_{\mathrm{cov}}\in[6,9]$ such that
\[
\E{\FS(u,v)} \;=\; 2n - c_{\mathrm{cov}} - o(1).
\]
\end{enumerate}
\end{conjecture}

\section{Record equivalence for pattern avoidance classes}\label{sec:record-equivalence}
 
In the Monte Carlo experiments above, we observed that the distribution of
$\FS(u,v)$ for $u,v \sim \Unif{\Av_n(132)}$ seemed to agree exactly with the
distribution of
$\FS(u,v)$
for $u,v \sim \Unif{\Av_n(231)}$.  This phenomenon is explained by an exact equidistribution
statement for record sets in these two Catalan classes, which we explain now.

\subsection{Record equidistribution for 132- and 231-avoiding permutations}\label{sec:132-vs-231}

The distribution of $\FS(u,v)$ for $u,v \sim \Unif{\Av_n(132)}$ turns out to agree exactly with the corresponding distribution
for $u,v \sim \Unif{\Av_n(231)}$.  This phenomenon is explained by an exact equidistribution
statement for record sets in these two Catalan classes, which we now prove.

Write $\Cat_m =\frac{1}{m+1}\binom{2m}{m}$ for the $m$th Catalan number. For $w\in S_n$, let $\Rec(w) = \{ j\in[n]: \rec_j(w)=1 \}$. Then $\Rec(w)$ always contains $1$, and the largest value of $\Rec(w)$ is the position of $n$ in $w$.

\begin{proposition}\label{prop:132-231-record-set-count}
Fix $n\ge 1$, a pattern $\pi\in\{132,231\}$, and a subset
\[
R=\{r_1<r_2<\cdots<r_k\}\subseteq [n]
\qquad\text{with }r_1=1.
\]
Set $r_{k+1}:=n+1$ and define the gaps $g_t:=r_{t+1}-r_t-1$  for $1 \leq t \leq k$. Then
\begin{equation}\label{eq:record-fiber-catalan-product}
\left|\{\,w\in \Av_n(\pi): \Rec(w)=R\,\}\right|
=
\prod_{t=1}^k \Cat_{g_t}.
\end{equation}
In particular, for every $R$ as above,
\[
\left|\{\,w\in \Av_n(132): \Rec(w)=R\,\}\right|
=
\left|\{\,w\in \Av_n(231): \Rec(w)=R\,\}\right|.
\]
\end{proposition}

\begin{proof}
We prove \eqref{eq:record-fiber-catalan-product} by strong induction on $n$, including the case $n=0$, $R=\varnothing$, and $k=0$. In this case both avoidance classes contain exactly one permutation, the empty permutation, whose record set is empty; the product on the right is an empty product and equals $1$.

Now let $n\ge1$, assume the formula for all smaller sizes, and write $R=\{r_1<\cdots<r_k\}$ with $r_1=1$.
Since $n$ is the unique global maximum, the final record position is the position of $n$, so for every
$w$ with $\Rec(w)=R$ we have $w(r_k)=n$.

First suppose $w\in \Av_n(231)$ and $w(r_k)=n$.
If there exist indices $i<r_k<j$ with $w(i)>w(j)$, then $(w(i),w(r_k),w(j))=(w(i),n,w(j))$ is order-isomorphic to $231$,
a contradiction. Hence
\begin{equation}\label{eq:231-split-rev}
w(1),\dots,w(r_k-1)< w(r_k+1),\dots,w(n).
\end{equation}
Since there are $r_k-1$ entries on the left and $n-r_k$ on the right, \eqref{eq:231-split-rev} forces
\[
\{w(1),\dots,w(r_k-1)\}=\{1,2,\dots,r_k-1\},
\qquad
\{w(r_k+1),\dots,w(n)\}=\{r_k,r_k+1,\dots,n-1\}.
\]
Thus the left block $u:=(w(1),\dots,w(r_k-1))\in S_{r_k-1}$ is already a permutation, while the right block is encoded by $v:=\std\left(w(r_k+1),\dots,w(n)\right)\in S_{n-r_k}$.
This construction is reversible: given $u\in S_{r_k-1}$ and $v\in S_{n-r_k}$, there is a unique word on the alphabet
$\{r_k,\dots,n-1\}$ whose standardization is $v$, and concatenating it after $n$ produces a permutation $w\in S_n$
satisfying \eqref{eq:231-split-rev}. Moreover, we claim $w\in \Av_n(231)$ if and only if $u\in \Av_{r_k-1}(231)$ and $v\in \Av_{n-r_k}(231)$.
Indeed, the forward implication is immediate since any pattern occurring in $u$ or in the right block of $w$
also occurs in $w$. For the reverse implication, suppose $u$ and $v$ both avoid $231$.
By \eqref{eq:231-split-rev}, every entry in positions $<r_k$ is smaller than every entry in positions $>r_k$,
and $w(r_k)=n$ is the global maximum. Hence a $231$-pattern in $w$ cannot use indices from both sides of $r_k$.
Therefore any $231$-pattern in $w$ would lie entirely in the left block or entirely in the right block, and thus cannot occur.

Similarly, suppose $w\in \Av_n(132)$ and $w(r_k)=n$.
If there exist indices $i<r_k<j$ with $w(i)<w(j)$, then $(w(i),n,w(j))$ is order-isomorphic to $132$, a contradiction.
Hence
\[
w(1),\dots,w(r_k-1)> w(r_k+1),\dots,w(n),
\]
so the left values are the $(r_k-1)$ largest elements of $\{1,\dots,n-1\}$ and the right values are the $(n-r_k)$ smallest.
In this case we encode both sides by standardization:
\[
u:=\std\left(w(1),\dots,w(r_k-1)\right)\in S_{r_k-1},
\qquad
v:=\std\left(w(r_k+1),\dots,w(n)\right)\in S_{n-r_k}.
\]
Again the construction is reversible, and $w\in \Av_n(132)$ if and only if
$u\in \Av_{r_k-1}(132)$ and $v\in \Av_{n-r_k}(132)$.

Because $n$ is the global maximum, no position strictly after $r_k$ can be a record, so $\Rec(w)\subseteq[r_k]$ and $r_k\in\Rec(w)$.
Moreover, for $j<r_k$, the event that $j$ is a record depends only on the relative order of the prefix
$(w(1),\dots,w(r_k-1))$, hence coincides with the record event at position $j$ for the standardized left block.
Consequently, for $\pi\in\{132,231\}$, $\Rec(w)=R$ if and only if $\Rec(u)=R\setminus\{r_k\}$.
The right factor $v$ contributes no records and is unconstrained by $R$.

Combining the above for each $\pi\in\{132,231\}$ we obtain a bijection
\[
\{\,w\in \Av_n(\pi): \Rec(w)=R\,\}
\longleftrightarrow
\{\,u\in \Av_{r_k-1}(\pi): \Rec(u)=R\setminus\{r_k\}\,\}\times \Av_{n-r_k}(\pi),
\]
and hence
\begin{equation}\label{eq:record-recursion-rev}
\left|\{\,w\in \Av_n(\pi): \Rec(w)=R\,\}\right|
=
\left|\{\,u\in \Av_{r_k-1}(\pi): \Rec(u)=R\setminus\{r_k\}\,\}\right|\cdot |\Av_{n-r_k}(\pi)|.
\end{equation}

It is well-known, going back to Knuth's analysis of one-stack sortable permutations~\cite[Section 2.2.1, Exercise 4]{KnuthTAOCP1}, that $|\Av_m(231)|=\Cat_m$ for all $m\ge 0$.
Since reversal of one-line notation, that is, the bijection $w\mapsto w w_0$ on $S_m$, sends the pattern $231$ to $132$,
it restricts to a bijection $\Av_m(231)\to \Av_m(132)$, and hence $|\Av_m(132)|=\Cat_m$ as well.
Thus \eqref{eq:record-recursion-rev} becomes
\[
\left|\{\,w\in \Av_n(\pi): \Rec(w)=R\,\}\right|
=
\left|\{\,u\in \Av_{r_k-1}(\pi): \Rec(u)=R\setminus\{r_k\}\,\}\right|\cdot \Cat_{n-r_k}.
\]
Now note that $n-r_k=g_k$, and the gap data for $R\setminus\{r_k\}\subseteq[r_k-1]$
is exactly $(g_1,\dots,g_{k-1})$. Applying the induction hypothesis to the left factor gives
\[
\left|\{\,u\in \Av_{r_k-1}(\pi): \Rec(u)=R\setminus\{r_k\}\,\}\right|
=
\prod_{t=1}^{k-1}\Cat_{g_t},
\]
so multiplying by $\Cat_{g_k}$ yields \eqref{eq:record-fiber-catalan-product}.  This completes the induction.
\end{proof}

\begin{corollary}\label{cor:132-231-record-eqdist}
Fix $n\ge 1$. The record-set statistic $\Rec$ is equidistributed on $\Av_n(132)$ and $\Av_n(231)$, in the sense that for
every subset $R\subseteq[n]$ with $1\in R$,
\[
\left|\{\,w\in \Av_n(132): \Rec(w)=R\,\}\right|
\;=\;
\left|\{\,w\in \Av_n(231): \Rec(w)=R\,\}\right|.
\]
Equivalently, the multisets $\{\!\{\Rec(w):w\in \Av_n(132)\}\!\}$ and $\{\!\{\Rec(w):w\in \Av_n(231)\}\!\}$ coincide.
\end{corollary}

\begin{corollary}\label{cor:132-231-record-eqdist-prob}
Let $w_{132}\sim \Unif{\Av_n(132)}$ and $w_{231}\sim \Unif{\Av_n(231)}$. Then for every $R\subseteq[n]$ with $1\in R$,
\[
\Prob{\Rec(w_{132})=R}=\Prob{\Rec(w_{231})=R}.
\]
Equivalently, the record indicator vectors $(\rec_1(w),\dots,\rec_n(w))$ (or $(\chi_1(w),\dots,\chi_n(w))$) have the same
distribution under $w=w_{132}$ and under $w=w_{231}$.
\end{corollary}

\begin{proof}
Divide the identity of Corollary~\ref{cor:132-231-record-eqdist} by
$|\Av_n(132)|=|\Av_n(231)|=\Cat_n$.
\end{proof}

\begin{corollary}\label{cor:132-231-fs-eqdist}
Let $(u_{132},v_{132})$ be independent with $u_{132},v_{132}\iidsim \Unif{\Av_n(132)}$, and let
$(u_{231},v_{231})$ be independent with $u_{231},v_{231}\iidsim \Unif{\Av_n(231)}$.
Then
\[
\FS(u_{132},v_{132})\ \stackrel{d}{=}\ \FS(u_{231},v_{231}).
\]
In particular, for every integer $t$,
\[
\Prob{\FS(u_{132},v_{132})=t}=\Prob{\FS(u_{231},v_{231})=t},
\qquad\text{and}\qquad
\E{\FS(u_{132},v_{132})} =\E{\FS(u_{231},v_{231})}.
\]
\end{corollary}

\begin{proof}
For $w\in S_{\mathbb Z_+}$, set
\[
m_\chi(w):=
\begin{cases}
1, & \text{if }\chi_j(w)=0\text{ for all }j,\\
\max\{j\ge 1:\chi_j(w)=1\}, & \text{otherwise}.
\end{cases}
\]
We claim that $m_\chi(w)=\FS(w)$.  Indeed, if $\chi_j(w)=0$ for all $j$, then every position is a record, hence $w$ is increasing,
so $w=id$ and $\FS(w)=1$.  Otherwise let $m=\FS(w)\ge 2$.  Since $w\in S_m$, we have $w(j)=j$ for all $j>m$, and because
$w(1),\dots,w(m)$ is a permutation of $[m]$, it follows that $w(j)=j>\max\{w(1),\dots,w(j-1)\}$ for every $j>m$; hence
$\chi_j(w)=0$ for all $j>m$.  On the other hand, $w(m)\neq m$ forces $w(m)<m$, so there exists $i<m$ with $w(i)>w(m)$,
and therefore $\chi_m(w)=1$.  Thus $\max\{j:\chi_j(w)=1\}=m$, proving the claim.

Now for $u,v\in S_{\mathbb Z_+}$, the Hardt--Wallach max formula reads
\[
\FS(u,v)=\max_{1\le i\le 1+\max(\FS(u),\FS(v))}\left(\Lambda_i(u)+\Lambda_i(v)+i-1\right),
\]
and each $\Lambda_i(w)$ is determined by $\chi(w)$ via $\Lambda_i(w)=\sum_{j\ge i}\chi_j(w)$.
By the claim above, the cutoff $1+\max(\FS(u),\FS(v))$ is also determined by the pair $(\chi(u),\chi(v))$.
Consequently, $\FS(u,v)$ is a deterministic function of the pair of record-complement vectors $(\chi(u),\chi(v))$.

By Corollary~\ref{cor:132-231-record-eqdist-prob}, we have $\chi(u_{132})\stackrel{d}{=}\chi(u_{231})$ and
$\chi(v_{132})\stackrel{d}{=}\chi(v_{231})$, and independence implies
\[
(\chi(u_{132}),\chi(v_{132}))\ \stackrel{d}{=}\ (\chi(u_{231}),\chi(v_{231})).
\]
Applying the map $(\chi(u),\chi(v))\mapsto \FS(u,v)$ gives
$\FS(u_{132},v_{132})\stackrel{d}{=}\FS(u_{231},v_{231})$, and the expectation identity follows immediately.
\end{proof}

\subsection{Toward a classification of record-equivalent patterns}

The exact coincidence between $\Av_n(132)$ and $\Av_n(231)$ proved in Section~\ref{sec:132-vs-231} suggests a broader classification problem: when do two avoidance classes have the same record-set distribution?

\begin{definition}
Let $\pi,\sigma\in S_k$. We say that $\pi$ and $\sigma$ are \emph{record-equivalent} if for every $n\ge k$ and every subset $R\subseteq[n]$,
\[
\left|\{\,w\in \Av_n(\pi): \Rec(w)=R\,\}\right|
=
\left|\{\,w\in \Av_n(\sigma): \Rec(w)=R\,\}\right|.
\]
\end{definition}

\begin{definition}
Let $\pi\in S_k$ and let $i\in[k]$. Let $\del_i(\pi)$ be the length-$(k-1)$ word obtained by deleting the entry of $\pi$ in position $i$,
and define
\[
\pi^{(i)}:=\std(\del_i(\pi))\in S_{k-1}.
\]
Also write $m_\pi:=\pi^{-1}(k)$ for the position of the maximum entry of $\pi$.
\end{definition}

\begin{conjecture}[Recursive criterion for record-equivalence]\label{conj:record-equivalence-recursive}
Fix $k\ge 2$ and let $\pi,\sigma\in S_k$. Then $\pi$ and $\sigma$ are record-equivalent if and only if all of the following hold:
\begin{enumerate}[label=\textnormal{(\arabic*)}]
\item \textbf{Record-set agreement at rank $k$:}
\[
\Rec(\pi)=\Rec(\sigma).
\]

\item \textbf{Wilf-equivalence at rank $k$:}
\[
|\Av_n(\pi)|=|\Av_n(\sigma)|\qquad\text{for all }n\ge k.
\]

\item \textbf{Non-max deletion minors are record-equivalent:} if $m_\pi=m_\sigma$, then for every $i\in[k]$ with $i\neq m_\pi$,
\[
\pi^{(i)}\ \text{and}\ \sigma^{(i)}\ \text{are record-equivalent in }S_{k-1}.
\]

\item \textbf{Terminal step when the maximum sits at $k-1$:} if $k\ge 4$ and
\[
\Rec(\pi)=\Rec(\sigma)=[k-1]
\qquad\text{and}\qquad
m_\pi=m_\sigma=k-1,
\]
then
\[
\pi(k)=\sigma(k).
\]
\end{enumerate}
\end{conjecture}

This conjecture is supported by exhaustive computations for all pairs of patterns of size $k\le 8$, with the record-set distributions compared through avoidance classes of lengths $n\le 15$. It suggests that record-equivalence is controlled recursively by the position of the maximum together with the record-equivalence classes of the deletion minors away from that position.

\printbibliography

@book {DubhashiPanconesi2009,
    AUTHOR = {Dubhashi, Devdatt P. and Panconesi, Alessandro},
     TITLE = {Concentration of measure for the analysis of randomized
              algorithms},
 PUBLISHER = {Cambridge University Press, Cambridge},
      YEAR = {2009},
     PAGES = {xvi+196},
      ISBN = {978-0-521-88427-3},
   MRCLASS = {68W20 (60C05 60F10 68-02 68W40)},
  MRNUMBER = {2547432},
MRREVIEWER = {Yannis\ C.\ Stamatiou},
       DOI = {10.1017/CBO9780511581274},
       URL = {https://doi.org/10.1017/CBO9780511581274},
}

@book {V89,
    AUTHOR = {Vajda, S.},
     TITLE = {Fibonacci \& {L}ucas numbers, and the golden section},
    SERIES = {Ellis Horwood Series: Mathematics and its Applications},
      NOTE = {Theory and applications,
              With chapter XII by B. W. Conolly},
 PUBLISHER = {Ellis Horwood Ltd., Chichester; Halsted Press [John Wiley \&
              Sons, Inc.], New York},
      YEAR = {1989},
     PAGES = {190},
      ISBN = {0-470-21508-9},
   MRCLASS = {11B39 (05A15)},
  MRNUMBER = {1015938},
MRREVIEWER = {Neville\ Robbins},
}

@misc{HardtWallach2024Schubert,
      title={When do Schubert polynomial products stabilize?}, 
      author={Andrew Hardt and David Wallach},
      year={2025},
      eprint={2412.06976},
      archivePrefix={arXiv},
      primaryClass={math.CO},
      url={https://arxiv.org/abs/2412.06976}, 
}

@article {HardtWallach-Grothendieck,
    AUTHOR = {Hardt, Andrew and Wallach, David},
     TITLE = {Stability of products of double {G}rothendieck polynomials},
   JOURNAL = {Int. Math. Res. Not. IMRN},
  FJOURNAL = {International Mathematics Research Notices. IMRN},
      YEAR = {2025},
    NUMBER = {22},
     PAGES = {Paper No. rnaf338, 15},
      ISSN = {1073-7928,1687-0247},
   MRCLASS = {05E05 (05E14 14N15)},
  MRNUMBER = {4987463},
MRREVIEWER = {Esther\ Banaian},
       DOI = {10.1093/imrn/rnaf338},
       URL = {https://doi.org/10.1093/imrn/rnaf338},
}

@article {Renyi1962,
    AUTHOR = {R\'{e}nyi, Alfr\'{e}d},
     TITLE = {Th\'{e}orie des \'{e}l\'{e}ments saillants d'une suite
              d'observations},
   JOURNAL = {Ann. Fac. Sci. Univ. Clermont-Ferrand},
  FJOURNAL = {Annales de la Facult\'{e} des Sciences de l'Universit\'{e} de
              Clermont-Ferrand},
    VOLUME = {8},
      YEAR = {1962},
     PAGES = {7--13},
      ISSN = {0069-472X,2400-4839},
   MRCLASS = {60.30},
  MRNUMBER = {286162},
MRREVIEWER = {R.\ Fischler},
}

@book {BjornerBrenti2005,
    AUTHOR = {Bj\"{o}rner, Anders and Brenti, Francesco},
     TITLE = {Combinatorics of {C}oxeter groups},
    SERIES = {Graduate Texts in Mathematics},
    VOLUME = {231},
 PUBLISHER = {Springer, New York},
      YEAR = {2005},
     PAGES = {xiv+363},
      ISBN = {978-3-540-44238-7}, % MR also lists ISBN 3-540-44238-3.
   MRCLASS = {05-01 (05E15 20F55)},
  MRNUMBER = {2133266},
MRREVIEWER = {Jian-yi\ Shi},
}

@article {Tenner2007,
    AUTHOR = {Tenner, Bridget Eileen},
     TITLE = {Pattern avoidance and the {B}ruhat order},
   JOURNAL = {J. Combin. Theory Ser. A},
  FJOURNAL = {Journal of Combinatorial Theory. Series A},
    VOLUME = {114},
      YEAR = {2007},
    NUMBER = {5},
     PAGES = {888--905},
      ISSN = {0097-3165,1096-0899},
   MRCLASS = {05E15 (06A07 20F55)},
  MRNUMBER = {2333139},
MRREVIEWER = {Alessandro\ Conflitti},
       DOI = {10.1016/j.jcta.2006.10.003},
       URL = {https://doi.org/10.1016/j.jcta.2006.10.003},
}

@book{Macdonald1991Schubert,
    author = {I. G. Macdonald},
    title = {Notes on Schubert Polynomials},
    publisher = {Publications du LaCIM},
    address = {Montr{\'e}al},
    year = 1991,
    volume = 6,
    ISBN = {2-89276-086-0 LACIM Montr\'eal},
}

@book {vanDerVaart1998AsymptoticStatistics,
    AUTHOR = {van der Vaart, A. W.},
     TITLE = {Asymptotic statistics},
    SERIES = {Cambridge Series in Statistical and Probabilistic Mathematics},
    VOLUME = {3},
 PUBLISHER = {Cambridge University Press, Cambridge},
      YEAR = {1998},
     PAGES = {xvi+443},
      ISBN = {0-521-49603-9}, % MR also lists ISBN 0-521-78450-6.
   MRCLASS = {62-02 (62E20 62F05 62F12 62G07 62G09 62G20)},
  MRNUMBER = {1652247},
MRREVIEWER = {Nancy\ Reid},
       DOI = {10.1017/CBO9780511802256},
       URL = {https://doi.org/10.1017/CBO9780511802256},
}

@article {PechenikSpeyerWeigandt2022,
    AUTHOR = {Pechenik, Oliver and Speyer, David E. and Weigandt, Anna},
     TITLE = {Castelnuovo-{M}umford regularity of matrix {S}chubert
              varieties},
   JOURNAL = {Selecta Math. (N.S.)},
  FJOURNAL = {Selecta Mathematica. New Series},
    VOLUME = {30},
      YEAR = {2024},
    NUMBER = {4},
     PAGES = {Paper No. 66, 44},
      ISSN = {1022-1824,1420-9020},
   MRCLASS = {05E05 (05E40 13C40 14M15)},
  MRNUMBER = {4768772},
MRREVIEWER = {Giuseppe\ Favacchio},
       DOI = {10.1007/s00029-024-00959-x},
       URL = {https://doi.org/10.1007/s00029-024-00959-x},
}

@article {LakshmibaiSandhya1990,
    AUTHOR = {Lakshmibai, V. and Sandhya, B.},
     TITLE = {Criterion for smoothness of {S}chubert varieties in {${\rm
              Sl}(n)/B$}},
   JOURNAL = {Proc. Indian Acad. Sci. Math. Sci.},
  FJOURNAL = {Indian Academy of Sciences. Proceedings. Mathematical
              Sciences},
    VOLUME = {100},
      YEAR = {1990},
    NUMBER = {1},
     PAGES = {45--52},
      ISSN = {0253-4142,0973-7685},
   MRCLASS = {14M15 (14L35)},
  MRNUMBER = {1051089},
MRREVIEWER = {H.\ H.\ Andersen},
       DOI = {10.1007/BF02881113},
       URL = {https://doi.org/10.1007/BF02881113},
}

@article {LiLiYong2011Drift,
    AUTHOR = {Li, Li and Yong, Alexander},
     TITLE = {Kazhdan-{L}usztig polynomials and drift configurations},
   JOURNAL = {Algebra Number Theory},
  FJOURNAL = {Algebra \& Number Theory},
    VOLUME = {5},
      YEAR = {2011},
    NUMBER = {5},
     PAGES = {595--626},
      ISSN = {1937-0652,1944-7833},
   MRCLASS = {14M15 (05E15 13D40 20C08)},
  MRNUMBER = {2889748},
MRREVIEWER = {Anthony\ Henderson},
       DOI = {10.2140/ant.2011.5.595},
       URL = {https://doi.org/10.2140/ant.2011.5.595},
}

@article {LiYong2012Degenerations,
    AUTHOR = {Li, Li and Yong, Alexander},
     TITLE = {Some degenerations of {K}azhdan-{L}usztig ideals and
              multiplicities of {S}chubert varieties},
   JOURNAL = {Adv. Math.},
  FJOURNAL = {Advances in Mathematics},
    VOLUME = {229},
      YEAR = {2012},
    NUMBER = {1},
     PAGES = {633--667},
      ISSN = {0001-8708,1090-2082},
   MRCLASS = {14M15 (05E15 13P10 14N15)},
  MRNUMBER = {2854186},
MRREVIEWER = {Changzheng\ Li},
       DOI = {10.1016/j.aim.2011.09.010},
       URL = {https://doi.org/10.1016/j.aim.2011.09.010},
}

@article {GaoHodgesYong2024LeviSpherical,
    AUTHOR = {Gao, Yibo and Hodges, Reuven and Yong, Alexander},
     TITLE = {Levi-spherical {S}chubert varieties},
   JOURNAL = {Adv. Math.},
  FJOURNAL = {Advances in Mathematics},
    VOLUME = {439},
      YEAR = {2024},
     PAGES = {Paper No. 109486, 14},
      ISSN = {0001-8708,1090-2082},
   MRCLASS = {14M15 (20G05 20G07)},
  MRNUMBER = {4690502},
MRREVIEWER = {J.\ Matthew\ Douglass},
       DOI = {10.1016/j.aim.2024.109486},
       URL = {https://doi.org/10.1016/j.aim.2024.109486},
}

@article {DiaconisZabell1991,
    AUTHOR = {Diaconis, Persi and Zabell, Sandy},
     TITLE = {Closed form summation for classical distributions: variations
              on a theme of de {M}oivre},
   JOURNAL = {Statist. Sci.},
  FJOURNAL = {Statistical Science. A Review Journal of the Institute of
              Mathematical Statistics},
    VOLUME = {6},
      YEAR = {1991},
    NUMBER = {3},
     PAGES = {284--302},
      ISSN = {0883-4237,2168-8745},
   MRCLASS = {60-03 (01A50 60E05 62E15)},
  MRNUMBER = {1144242},
MRREVIEWER = {C.\ C.\ Heyde},
       URL = {http://links.jstor.org/sici?sici=0883-4237(199108)6:3<284:CFSFCD>2.0.CO;2-R&origin=MSN},
}

@book {KnuthTAOCP1,
    AUTHOR = {Knuth, Donald E.},
     TITLE = {The art of computer programming. {V}ol. 1},
      NOTE = {Fundamental algorithms,
              Third edition [of MR0286317]},
 PUBLISHER = {Addison-Wesley, Reading, MA},
      YEAR = {1997},
     PAGES = {xx+650},
      ISBN = {0-201-89683-4},
   MRCLASS = {68-02},
  MRNUMBER = {3077152},
}

@article {LascouxSchutzenberger,
    AUTHOR = {Lascoux, Alain and Sch\"{u}tzenberger, Marcel-Paul},
     TITLE = {Polyn\^{o}mes de {S}chubert},
   JOURNAL = {C. R. Acad. Sci. Paris S\'{e}r. I Math.},
  FJOURNAL = {Comptes Rendus des S\'{e}ances de l'Acad\'{e}mie des Sciences.
              S\'{e}rie I. Math\'{e}matique},
    VOLUME = {294},
      YEAR = {1982},
    NUMBER = {13},
     PAGES = {447--450},
      ISSN = {0249-6291},
   MRCLASS = {14M17 (05A10 14N10)},
  MRNUMBER = {660739},
}

@article {Li-back-stable-conjecture,
    AUTHOR = {Li, Nan},
     TITLE = {A canonical expansion of the product of two {S}tanley
              symmetric functions},
   JOURNAL = {J. Algebraic Combin.},
  FJOURNAL = {Journal of Algebraic Combinatorics. An International Journal},
    VOLUME = {39},
      YEAR = {2014},
    NUMBER = {4},
     PAGES = {833--851},
      ISSN = {0925-9899,1572-9192},
   MRCLASS = {05E05},
  MRNUMBER = {3199028},
MRREVIEWER = {Domenico\ Senato},
       DOI = {10.1007/s10801-013-0469-2},
       URL = {https://doi.org/10.1007/s10801-013-0469-2},
}

@book {LevinPeres2017Markov,
    AUTHOR = {Levin, David A. and Peres, Yuval},
     TITLE = {Markov chains and mixing times},
      NOTE = {Second edition of [MR2466937],
              With contributions by Elizabeth L. Wilmer,
              With a chapter on ``Coupling from the past'' by James G. Propp
              and David B. Wilson},
 PUBLISHER = {American Mathematical Society, Providence, RI},
      YEAR = {2017},
     PAGES = {xvi+447},
      ISBN = {978-1-4704-2962-1},
   MRCLASS = {60J10 (60-01 60B15 60C05 60J27 60K35 68U20 82C22)},
  MRNUMBER = {3726904},
       DOI = {10.1090/mbk/107},
       URL = {https://doi.org/10.1090/mbk/107},
}

@article {Hersh05,
    AUTHOR = {Hersh, Patricia},
     TITLE = {On optimizing discrete {M}orse functions},
   JOURNAL = {Adv. in Appl. Math.},
  FJOURNAL = {Advances in Applied Mathematics},
    VOLUME = {35},
      YEAR = {2005},
    NUMBER = {3},
     PAGES = {294--322},
      ISSN = {0196-8858,1090-2074},
   MRCLASS = {55P15 (05E25 06A07)},
  MRNUMBER = {2164921},
MRREVIEWER = {Axel\ Hultman},
       DOI = {10.1016/j.aam.2005.04.001},
       URL = {https://doi.org/10.1016/j.aam.2005.04.001},
}

@incollection {MPR09,
    AUTHOR = {Merlev\`ede, Florence and Peligrad, Magda and Rio, Emmanuel},
     TITLE = {Bernstein inequality and moderate deviations under strong
              mixing conditions},
 BOOKTITLE = {High dimensional probability {V}: the {L}uminy volume},
    SERIES = {Inst. Math. Stat. (IMS) Collect.},
    VOLUME = {5},
     PAGES = {273--292},
 PUBLISHER = {Inst. Math. Statist., Beachwood, OH},
      YEAR = {2009},
      ISBN = {978-0-940600-78-2},
   MRCLASS = {60E15 (60F10)},
  MRNUMBER = {2797953},
MRREVIEWER = {M.\ Cs\"{o}rg\H{o}},
       DOI = {10.1214/09-IMSCOLL518},
       URL = {https://doi.org/10.1214/09-IMSCOLL518},
}

@article {Bradley05,
    AUTHOR = {Bradley, Richard C.},
     TITLE = {Basic properties of strong mixing conditions. {A} survey and
              some open questions},
      NOTE = {Update of, and a supplement to, the 1986 original},
   JOURNAL = {Probab. Surv.},
  FJOURNAL = {Probability Surveys},
    VOLUME = {2},
      YEAR = {2005},
     PAGES = {107--144},
      ISSN = {1549-5787},
   MRCLASS = {60G10},
  MRNUMBER = {2178042},
MRREVIEWER = {Lajos\ Horv\'{a}th},
       DOI = {10.1214/154957805100000104},
       URL = {https://doi.org/10.1214/154957805100000104},
}

\end{document}